\documentclass{amsart}
\usepackage{amsmath}
\usepackage{amsfonts}
\usepackage{amsthm}
\usepackage{amsbsy}
\usepackage{amssymb}
\usepackage{amsthm}
\usepackage{enumerate}
\usepackage[margin=1in]{geometry}
\usepackage{hyperref}
\usepackage{mathrsfs}
\usepackage{mathtools}
\usepackage{microtype}
\usepackage[nobysame]{amsrefs}

\newtheorem{thm}{Theorem}[section]
\newtheorem{lem}[thm]{Lemma}
\newtheorem{prop}[thm]{Proposition}
\newtheorem{cor}[thm]{Corollary}

\theoremstyle{definition}
\newtheorem{defn}[thm]{Definition}

\newtheorem{rem}[thm]{Remark}
\newtheorem{exam}[thm]{Example}

\newcommand{\bC}{{\mathbb{C}}}
\newcommand{\bE}{{\mathbb{E}}}
\newcommand{\bM}{{\mathbb{M}}}
\newcommand{\bN}{{\mathbb{N}}}
\newcommand{\bR}{{\mathbb{R}}}

\newcommand{\A}{{\mathcal{A}}}
\newcommand{\B}{{\mathcal{B}}}
\newcommand{\C}{{\mathcal{C}}}
\newcommand{\D}{{\mathcal{D}}}
\newcommand{\E}{{\mathcal{E}}}

\newcommand{\I}{{\mathcal{I}}}
\newcommand{\K}{{\mathcal{K}}}
\renewcommand{\L}{{\mathcal{L}}}

\newcommand{\N}{{\mathcal{N}}}

\newcommand{\R}{{\mathcal{R}}}

\newcommand{\ep}{\varepsilon}

\newcommand{\qand}{\quad\text{and}\quad}
\newcommand{\qqand}{\qquad\text{and}\qquad}

\newcommand{\alg}{\mathrm{alg}}
\newcommand{\Cf}{\mathrm{Cf}}

\newcommand{\bReta}{\mathrm{bReta}}
\newcommand{\bB}{\mathrm{b}\mathbb{B}}

\newcommand{\lat}{\mathrm{lat}}
\newcommand{\capp}{\mathrm{cap}}
\newcommand{\freestar}{\framebox[7pt]{$\star$}}

\usepackage{tikz}
\usetikzlibrary{shapes,snakes,calc,arrows}
\usetikzlibrary{decorations.pathreplacing,shapes.geometric}
\usetikzlibrary{calc,positioning}
\tikzset{Box/.style={very thick, rounded corners}}
\tikzset{marked/.style={star, star point height = .75mm, star points =5, fill=black,minimum size=2mm, inner sep=0mm} }
\tikzset{verythickline/.style = {line width=7pt}}
\tikzset{thickline/.style = {line width=5pt}}
\tikzset{medthick/.style = {line width=3pt}}
\tikzset{med/.style = {line width=2pt}}
\tikzset{count/.style = {fill=white,circle,draw,thin, inner sep=2pt}}
\tikzset{rcount/.style = {fill=white,rectangle,draw,thin,inner sep=2pt, rounded corners}}
\tikzset{cpr/.style = {draw,fill=white,rectangle,thin, rounded corners}}

\definecolor{ggreen}{HTML}{00BB33}

\begin{document}

\nocite{*}

\title[Bi-Boolean independence for pairs of algebras]{Bi-Boolean independence for pairs of algebras}

\author{Yinzheng Gu and Paul Skoufranis}

\address{Department of Mathematics and Statistics, Queen's University, Jeffery Hall, Kingston, Ontario, K7L 3N6, Canada}
\email{gu.y@queensu.ca}

\address{Department of Mathematics and Statistics, York University, 4700 Keele Street, Toronto, Ontario, M3J 1P3, Canada}
\email{pskoufra@yorku.ca}

\date{\today}
\thanks{The work of Yinzheng Gu was partially supported by CIMI (Centre International de Math\'{e}matiques et d'Informatique) Excellence program, ANR-11-LABX-0040-CIMI within the program ANR-11-IDEX-0002-02, while visiting the Institute of Mathematics of Toulouse. He would like to thank the institute for the generous hospitality and Serban Belinschi for his constant support and valuable advice when this research was conducted.}

\begin{abstract}
In this paper, the notion of bi-Boolean independence for non-unital pairs of algebras is introduced thereby extending the notion of Boolean independence to pairs of algebras. The notion of B-$(\ell, r)$-cumulants is defined via a bi-Boolean moment-cumulant formula over the lattice of bi-interval partitions, and it is demonstrated that bi-Boolean independence is equivalent to the vanishing of mixed B-$(\ell, r)$-cumulants. Furthermore, some of the simplest bi-Boolean convolutions are considered, and a bi-Boolean partial $\eta$-transform is constructed for the study of limit theorems and infinite divisibility with respect to the additive bi-Boolean convolution. In particular, a bi-Boolean L\'{e}vy-Hin\v{c}in formula is derived in perfect analogy with the bi-free case, and some Bercovici-Pata type bijections are provided. Additional topics considered include the additive bi-Fermi convolution, some relations between the $(\ell, r)$- and B-$(\ell, r)$-cumulants, and bi-Boolean independence in an amalgamated setting.

The last section of this paper also includes an errata that will be published with this copy of the paper.
\end{abstract}

\maketitle

\section{Introduction}

By the classification work \cites{S1997, M2002} there are only three symmetric notions of universal independences: classical independence, free independence, and Boolean independence \cite{SW1997}.  From a combinatorial point of view, the main difference between these notions of independence is the lattice of partitions used to describe the moment-cumulant formula.  The entire lattice of partitions is required to study classical independence whereas one restricts to the non-crossing partitions and further restricts to the interval partitions to study free independence and Boolean independence respectively.  

Recently in \cite{V2014}, Voiculescu introduced bi-free probability as a generalization of free probability to enable the study of non-commutative left and right actions of algebras on a reduced free product space simultaneously.  To study bi-free probability via moment-cumulant formula, permutations of the lattice of non-crossing partitions, known as bi-non-crossing partitions, are used.  Since its inception, the theory has attracted a lot of attention and quickly developed by the substantial work of various authors. For a summary of the current stage of bi-free probability, we refer to the survey \cite{V2016-2} by Voiculescu himself and the references therein.

These notions of independence implement the use of a single state.  By adding an additional state, extensions may be obtained.  For example, the notion of conditionally free independence \cites{BLS1996, BS1991} includes both free independence and Boolean independence as special cases by choosing the states appropriately.  In our previous paper \cite{GS2016}, the notion of conditionally bi-free independence was introduced as an extension of bi-free independence to the two-state setting and as an extension of conditionally free independence to pairs of algebras.  By selecting pairs of states in the same manner as one does to obtain Boolean independence from conditional free independence, a notion of Boolean independence for pairs of algebras, called bi-Boolean independence, is obtained as a specific instance of conditional bi-free independence.  The main purpose of this paper is to study bi-Boolean independence as an avenue for better understanding bi-free independence and as bi-Boolean independence is an intriguing case of conditional bi-free independence.

Excluding this introduction, this paper has seven sections organized as follows. In Section \ref{sec:prelims}, several notions of non-commutative independences are recalled; namely free, Boolean, c-free, bi-free, and c-bi-free independences, with an emphasis on the combinatorial aspects and moment-cumulant formulae.

In Section \ref{sec:biboolean}, the notion of bi-Boolean independence for pairs of algebras is introduced as a rule of calculating mixed moments and is related to c-bi-free independence. The lattice of bi-interval partitions is introduced, from which the family of B-$(\ell, r)$-cumulants can be defined. It is then demonstrated that B-$(\ell, r)$-cumulants have the required vanishing property and thus can be used to characterize bi-Boolean independence.

In Section \ref{sec:transforms}, some of the simplest bi-Boolean convolutions are considered. In particular, a bi-Boolean partial $\eta$-transform is constructed (which linearizes the additive bi-Boolean convolution) and a functional equation is derived relating it to the Cauchy transform. On the other hand, a slight curiosity arises when multiplication is involved as one needs to decide whether to use the usual multiplication or the opposite multiplication on the right variables. As it is not clear which multiplication is preferred, both convolutions will be studied, and it is demonstrated how the reduced bi-Boolean partial $\eta$-transform can be used to obtain the convolved distributions. However, it is not known whether there is a corresponding transform with the multiplicative property.

In Section \ref{sec:limitthms}, various limit theorems with respect to the additive bi-Boolean convolution are considered. Due to the fact that the additive bi-Boolean convolution is really just a special case of the additive c-bi-free convolution as studied in \cite{GS2016}*{Section 6}, only the statements are presented for illustration purposes. In particular, infinite divisibility is addressed and a bi-Boolean L\'{e}vy-Hin\v{c}in formula is derived which, together with the work of \cite{HW2016}, naturally leads to a two-dimensional Bercovici-Pata bijection. Although every probability measure on $\bR$ is infinitely divisible with respect to the additive Boolean convolution, the same is not true for probability measures on $\bR^2$ with respect to the additive bi-Boolean convolution.

In Section \ref{sec:bifermi}, another special case of the additive c-bi-free convolution, called the additive bi-Fermi convolution, is considered as the two-dimensional version of the notion of additive Fermi convolution introduced in \cite{O2002}. It is shown that the extension from Fermi convolution to bi-Fermi convolution is completely analogous to the extension from Boolean convolution to bi-Boolean convolution from both combinatorial and analytic viewpoints. In particular, the two notions (i.e., bi-Boolean and bi-Fermi) coincide under special circumstances.

In Section \ref{sec:bifreeness}, the (general) bi-free $\R$-transform and bi-Boolean $\eta$-transform are studied in an algebraic framework for two-faced families of non-commutative random variables. The coefficients of these transforms are, by definition, $(\ell, r)$- and B-$(\ell, r)$-cumulants of such families, and an explicit formula is derived relating these coefficients. Moreover, two bijections are defined and studied analogous to the results in \cite{BN2008-1}, and a special property of one of these bijections is proved. Consequently, another multiplicative bi-free convolution is re-considered where the usual multiplication is used on the left variables and the opposite multiplication is used on the right variables.

Finally, in Section \ref{sec:op-valued}, the notion of bi-Boolean independence is extended to an amalgamated setting. Operator-valued bi-Boolean cumulants are defined and shown to linearize operator-valued bi-Boolean independence as expected. Moreover, an operator-valued bi-Boolean partial $\eta$-transform is constructed, which is a function of two variables, and a functional equation relating it to the moment series is presented.

\section{Preliminaries}\label{sec:prelims}

In this section, we briefly review several notions of independence that are relevant to this paper. The main purpose is to develop notations that will be used later, and we shall only discuss the combinatorial aspects of the theories; namely the corresponding cumulants and moment-cumulant formulae.

\subsection{Free, Boolean, and c-free independences}

We begin with free and Boolean independences.

\begin{defn}
Let $(\A, \varphi)$ be a non-commutative probability space.
\begin{enumerate}[$\qquad(1)$]
\item A family $\{\A_k\}_{k \in K}$ of unital subalgebras of $\A$ is said to be \textit{freely independent} with respect to $\varphi$ if
\[\varphi(a_1\cdots a_n) = 0\]
whenever $a_j \in \A_{k_j}$, $k_j \in K$, $k_1 \neq \cdots \neq k_n$, and $\varphi(a_j) = 0$ for all $1 \leq j \leq n$.

\item A family $\{\A_k\}_{k \in K}$ of subalgebras of $\A$ is said to be \textit{Boolean independent} with respect to $\varphi$ if
\[\varphi(a_1\cdots a_n) = \varphi(a_1)\cdots\varphi(a_n)\]
whenever $a_j \in \mathcal{A}_{k_j}$, $k_j \in K$, and $k_1 \neq \cdots \neq k_n$.
\end{enumerate}
\end{defn}

Note that if $\A_1$ and $\A_2$ are two Boolean independent unital subalgebras in $(\A, \varphi)$, then for all $a \in \A_1$ and $n \geq 1$
\[\varphi(a^n) = \varphi(a1a1\cdots a1) = \varphi(a)\varphi(1)\varphi(a)\varphi(1)\cdots\varphi(a)\varphi(1) = \varphi(a)^n.\]
Similarly $\varphi(b^n) = \varphi(b)^n$ for all $b \in \A_2$. In this case, the independence is rather trivial, and hence we restrict to non-unital subalgebras. For the same reason, the notion of bi-Boolean independence will be defined for non-unital pairs of algebras.

For the combinatorial approach, the main idea is to express moments in terms of other quantities, called cumulants, by summing over certain partitions.

\begin{defn}
Let $\pi = \{V_1, \dots, V_r\}$ be a partition of $\{1, \dots, n\}$.
\begin{enumerate}[$\qquad(1)$]
\item The partition $\pi$ is said to be \textit{crossing} if there exist $a < b < c < d$ such that $a, c \in V_i$, $b, d \in V_j$, and $V_i \neq V_j$. If $\pi$ is not crossing, then it is said to be \textit{non-crossing}. The set of non-crossing partitions of $\{1, \dots, n\}$ is denoted by $\N\C(n)$.

\item A block $V_i$ of $\pi$ is said to be an \textit{interval} if $V_i$ is of the form $V_i = \{k, k + 1, \dots, k + \ell\}$ for some $k \geq 1$ and $0 \leq \ell \leq n - k$. The partition $\pi$ is said to be an \textit{interval partition} if every block of $\pi$ is an interval. The set of interval partitions of $\{1, \dots, n\}$ is denoted by $\I(n)$.
\end{enumerate}
\end{defn}

If $a_1, \dots, a_n$ are random variables in a non-commutative probability space $(\A, \varphi)$ and $V = \{v_1 < \cdots < v_s\}$ is a block of some partition of $\{1, \dots, n\}$, then we set $(a_1, \dots, a_n)|_V := (a_{v_1}, \dots, a_{v_s})$. The free and Boolean cumulants are recursively defined as follows.

\begin{defn}[\cites{SW1997, S1994}]
Let $(\A, \varphi)$ be a non-commutative probability space. The families of \textit{free cumulants} and \textit{Boolean cumulants} with respect to $\varphi$ are respectively the families of multilinear functionals
\[\{\kappa_n: \mathcal{A}^n \to \mathbb{C}\}_{n \geq 1} \qand \{B_n: \mathcal{A}^n \to \mathbb{C}\}_{n \geq 1}\]
uniquely determined by the requirements that for all $n \geq 1$ and $a_1, \dots, a_n \in \mathcal{A}$
\[\varphi(a_1\cdots a_n) = \sum_{\pi \in \N\C(n)}\kappa_\pi(a_1, \dots, a_n)\qqand \varphi(a_1\cdots a_n) = \sum_{\pi \in \I(n)}B_\pi(a_1, \dots, a_n)\]
where
\[
\kappa_\pi(a_1, \dots, a_n) =\prod_{V \in \pi}\kappa_{|V|}((a_1, \dots, a_n)|_V) \qqand B_\pi(a_1, \dots, a_n) = \prod_{V \in \pi}B_{|V|}((a_1, \dots, a_n)|_V).
\]
\end{defn}

Given a partition $\pi = \{V_1, \dots, V_r\}$ of $\{1, \dots, n\}$ with blocks $V_i = \{v_{i, 1} < \cdots < v_{i, s_i}\}$, set
\[\varphi_\pi(a_1, \dots, a_n) := \prod_{i = 1}^r\varphi(a_{v_{i, 1}}\cdots a_{v_{i, s_i}}).\]
For $\pi_1 \in \N\C(n)$ and $\pi_2 \in \I(n)$, we obtain via M\"{o}bius inversions (see \cite{S1994, SW1997}) that
\begin{align*}
\kappa_{\pi_1}(a_1, \dots, a_n) &= \sum_{\substack{\sigma_1 \in \N\C(n)\\\sigma_1 \leq \pi_1}}\varphi_{\sigma_1}(a_1, \dots, a_n)\mu_{\N\C}(\sigma_1, \pi_1) \\
B_{\pi_2}(a_1, \dots, a_n) &= \sum_{\substack{\sigma_2 \in \I(n)\\\sigma_2 \leq \pi_2}}\varphi_{\sigma_2}(a_1, \dots, a_n)\mu_{\I}(\sigma_2, \pi_2),
\end{align*}
where $\sigma \leq \pi$ denotes $\sigma$ is a refinement of $\pi$, $\mu_{\N\C}$ and $\mu_\I$ denote the M\"{o}bius functions on the lattices $\N\C$ and $\I$ of all non-crossing partitions and all interval partitions respectively.

One of the main advantages for switching from moments to cumulants is that the corresponding independence is much easier to describe on the level of cumulants. More specifically, a family $\{\A_k\}_{k \in K}$ of unital (respectively non-unital) subalgebras in a non-commutative probability space $(\A, \varphi)$ is freely (respectively Boolean) independent with respect to $\varphi$ if and only if $\kappa_n(a_1, \dots, a_n) = 0$ (respectively $B_n(a_1, \dots, a_n) = 0$) whenever $n \geq 2$, $a_j \in \A_{k_j}$, and there exist $i$ and $j$ such that $k_i \neq k_j$ (see \cites{S1994, SW1997}).

Consider now the framework where $\A$ is a unital algebra equipped with two unital linear functionals $\varphi$ and $\psi$. In this setting, Bo\.{z}ejko, Leinert, and Speicher \cites{BLS1996, BS1991} introduced the notion of c-free independence as follows.

\begin{defn}
Let $(\A, \varphi, \psi)$ be a two-state non-commutative probability space. A family $\{\mathcal{A}_k\}_{k \in K}$ of unital subalgebras of $\mathcal{A}$ is said to be \textit{conditionally freely independent} (\textit{c-free} for short) with respect to $(\varphi, \psi)$ if
\begin{enumerate}[$\qquad(1)$]
\item $\psi(a_1\cdots a_n) = 0$,

\item $\varphi(a_1\cdots a_n) = \varphi(a_1)\cdots\varphi(a_n)$
\end{enumerate}
whenever $a_j \in \A_{k_j}$, $k_j \in K$, $k_1 \neq \cdots \neq k_n$, and $\psi(a_j) = 0$ for all $1 \leq j \leq n$.
\end{defn}

Note that c-free independence with respect to $(\varphi, \psi)$ automatically implies free independence with respect to $\psi$, and hence implies the vanishing of mixed free cumulants. To characterize c-free independence in terms of cumulants, another family is required.

\begin{defn}
Given $\pi \in \N\C(n)$, a block $V$ of $\pi$ is said to be \textit{inner} if there exists another block $W$ of $\pi$ and $a, b \in W$ such that $a < v < b$ for some (hence all) $v \in V$. A block of $\pi$ is said to be \textit{outer} if it is not inner.
\end{defn}

\begin{defn}[\cites{BLS1996, BS1991}]
Let $(\A, \varphi, \psi)$ be a two-state non-commutative probability space. The family of \textit{c-free cumulants} with respect to $(\varphi, \psi)$ is the family of multilinear functionals $\{\K_n: \A^n \to \bC\}_{n \geq 1}$ uniquely determined by the requirement that for all $n \geq 1$ and $a_1, \dots, a_n \in \mathcal{A}$
\begin{align*}
\varphi(a_1\cdots a_n) &= \sum_{\pi \in \N\C(n)}\K_\pi(a_1, \dots, a_n)
\end{align*}
where
\[
\K_\pi(a_1, \dots, a_n) = \left(\prod_{\substack{V \in \pi\\V\,\mathrm{inner}}}\kappa_{|V|}((a_1, \dots, a_n)|_V)\right)\left(\prod_{\substack{V \in \pi\\V\,\mathrm{outer}}}\mathcal{K}_{|V|}((a_1, \dots, a_n)|_V)\right)
\]
and $\left\{\kappa_n: \A^n \to \bC\right\}_{n \geq 1}$ denotes the family of free cumulants with respect to $\psi$.
\end{defn}

The notion of c-free independence can now be described as the vanishing of mixed free and c-free cumulants (see \cites{BLS1996}).

\subsection{Bi-free and c-bi-free independences}

Bi-free independence was introduced by Voiculescu \cite{V2014} as a generalization of free independence for pairs of algebras, where a \textit{pair of algebra} in a non-commutative probability space $(\A, \psi)$ is an ordered pair $(\A_\ell, \A_r)$ of subalgebras of $\A$. We call $\A_\ell$ the \textit{left algebra} and $\A_r$ the \textit{right algebra} of the pair $(\A_\ell, \A_r)$. In \cite{GS2016}, we introduced the notion of c-bi-free independence for pairs of algebras in the setting of a two-state non-commutative probability space $(\A, \varphi, \psi)$ such that c-bi-free independence reduces to bi-free independence when $\varphi = \psi$ and reduces to c-free independence when only left or only right algebras are considered. For the theoretical definitions of bi-free and c-bi-free independences in terms of actions on a reduced free product space, we refer to \cite{V2014}*{Definition 2.6} and \cite{GS2016}*{Definition 3.4} respectively. To describe the mixed moments of a family $\{(\A_{k, \ell}, \A_{k, r})\}_{k \in K}$ of pairs of algebras and the corresponding cumulants, we introduce the following definitions and notations.

Throughout the rest, a map $\chi: \{1, \dots, n\} \to \{\ell, r\}$ is used to designate whether the $j^{\mathrm{th}}$ random variable in a sequence of $n$ random variables is a left variable (when $\chi(j) = \ell$) or a right variable (when $\chi(j) = r$), and a map $\omega: \{1, \dots, n\} \to K$ is used to designate the index in $K$ of the $j^{\mathrm{th}}$ random variable. Given a map $\chi: \{1, \dots, n\} \to \{\ell, r\}$ with $\chi^{-1}(\{\ell\}) = \{i_1 < \cdots < i_p\}$ and $\chi^{-1}(\{r\}) = \{i_{p + 1} > \cdots > i_n\}$, define a permutation $s_\chi$ on $\{1, \dots, n\}$ by $s_\chi(j) = i_j$ for $1 \leq j \leq n$, and define a total order $\prec_\chi$ on $\{1, \dots, n\}$ by $i_1 \prec_\chi \cdots \prec_\chi i_n$. A subset $V \subset \{1, \dots, n\}$ is said to be a \textit{$\chi$-interval} if it is an interval with respect to $\prec_\chi$. In addition, we define $\min_{\prec_\chi}(V)$ and $\max_{\prec_\chi}(V)$ to be the minimal and maximal elements of $V$ with respect to $\prec_\chi$ respectively.

\begin{defn}
A partition $\pi$ of $\{1, \dots, n\}$ is said to be \textit{bi-non-crossing} (with respect to $\chi$) if $s_\chi^{-1}\cdot\pi \in \N\C(n)$; that is, the partition formed by applying $s_\chi^{-1}$ to the blocks of $\pi$ is a non-crossing partition. Equivalently, $\pi$ is a non-crossing partition with respect to $\prec_\chi$. The set of bi-non-crossing partitions with respect to $\chi$ is denoted by $\B\N\C(\chi)$, which is a lattice and the minimal and maximal elements (with respect to $\leq$) are denoted by $0_\chi$ and $1_\chi$ respectively. 
\end{defn}

Bi-non-crossing partitions corresponding to a given map $\chi: \{1, \dots, n\} \to \{\ell, r\}$ can be represented diagrammatically by placing $n$ nodes labelled $1$ to $n$ on two vertical dashed lines from top to bottom in increasing order with node $j$ on the left or right depending on whether $\chi(j) = \ell$ or $\chi(j) = r$, and connecting the nodes which are in the same block using only horizontal and vertical lines in a non-crossing way. Moreover, given a bi-non-crossing partition $\pi \in \B\N\C(\chi)$, the vertical segment of a block $V$ of $\pi$ will be referred to as the \textit{spine} of $V$ and the horizontal segments connecting the nodes to the spine of $V$ will be referred to as the \textit{ribs} of $V$. To introduce the corresponding cumulants, we need the following analogues of inner and outer blocks.

\begin{defn}
Given $\chi: \{1, \dots, n\} \to \{\ell, r\}$ and $\pi \in \B\N\C(\chi)$, a block $V$ of $\pi$ is said to be \textit{interior} if there exists another block $W$ of $\pi$ such that $\min_{\prec_\chi}(W) \prec_\chi \min_{\prec_\chi}(V)$ and $\max_{\prec_\chi}(V) \prec_\chi \max_{\prec_\chi}(W)$. A block of $\pi$ is said to be \textit{exterior} if it is not interior.
\end{defn}

\begin{defn}[\cites{MN2015, GS2016}]
Let $(\A, \varphi, \psi)$ be a two-state non-commutative probability space. The families of \textit{$(\ell, r)$-cumulants} and \textit{c-$(\ell, r)$-cumulants} with respect to $\psi$ and $(\varphi, \psi)$ are respectively the families of multilinear functionals
\[\{\kappa_\chi: \A^n \to \bC\}_{n \geq 1, \chi: \{1, \dots, n\} \to \{\ell, r\}} \qand \{\K_\chi: \A^n \to \bC\}_{n \geq 1, \chi: \{1, \dots, n\} \to \{\ell, r\}}\]
uniquely determined by the requirements that for all $n \geq 1$, $\chi: \{1, \dots, n\} \to \{\ell, r\}$, and $a_1, \dots, a_n \in \A$
\[\psi(a_1\cdots a_n) = \sum_{\pi \in \B\N\C(\chi)}\kappa_\pi(a_1, \dots, a_n) \qqand \varphi(a_1\cdots a_n)=\sum_{\pi \in \B\N\C(\chi)}\K_\pi(a_1, \dots, a_n)
\]
where
\begin{align*}
\kappa_\pi(a_1, \dots, a_n) &= \prod_{V \in \pi}\kappa_{\chi|_V}((a_1, \dots, a_n)|_V) \qqand  \\
\K_\pi(a_1, \dots, a_n) &= \left(\prod_{\substack{V \in \pi\\V\,\mathrm{interior}}}\kappa_{\chi|_V}((a_1, \dots, a_n)|_V)\right)\left(\prod_{\substack{V \in \pi\\V\,\mathrm{exterior}}}\K_{\chi|_V}((a_1, \dots, a_n)|_V)\right).
\end{align*}
\end{defn}

Since $\B\N\C(\chi)$ inherits a special lattice structure from the set of all partitions of $\{1, \dots, n\}$, for $\pi \in \B\N\C(\chi)$ we obtain via M\"{o}bius inversion (see \cite{CNS2015-1}) that
\[\kappa_\pi(a_1, \dots, a_n) = \sum_{\substack{\sigma \in \B\N\C(\chi)\\\sigma \leq \pi}}\psi_\sigma(a_1, \dots, a_n)\mu_{\B\N\C}(\sigma, \pi),\]
where $\mu_{\B\N\C}$ denotes the M\"{o}bius function on the lattice $\B\N\C$ of all bi-non-crossing partitions. Due to similar lattice structures, we have that $\mu_{\B\N\C}(\sigma, \pi) = \mu_{\N\C}(s_\chi^{-1}\cdot\sigma, s_\chi^{-1}\cdot\pi)$ for $\sigma, \pi \in \B\N\C(\chi)$.

To obtain a moment formula and vanishing characterization for bi-free and c-bi-free independences, the following sets of shaded diagrams and terminology are required.

\begin{defn}
Let $n \geq 1$, $\chi: \{1, \dots, n\} \to \{\ell, r\}$, and $\omega: \{1, \dots, n\} \to K$ be given. Assign a shade to each $k \in K$.
\begin{enumerate}[$\qquad(1)$]
\item The set $\L\R(\chi, \omega)$ of \textit{shaded $\L\R$-diagrams} is recursively constructed as follows.
\begin{itemize}
\item For $n = 1$, $\L\R(\chi, \omega)$ consists of two vertical dashed lines with a single node shaded $\omega(1)$ on the left or right depending on whether $\chi(1) = \ell$ or $\chi(1) = r$. Then either this node remains isolated or a rib and spine shaded $\omega(1)$ are drawn connecting to the top of the diagram.

\item For $n \geq 2$, let $\chi_0 = \chi|_{\{2, \dots, n\}}$ and $\omega_0 = \omega|_{\{2, \dots, n\}}$. Each diagram $D \in \L\R(\chi_0, \omega_0)$ extends to two diagrams in $\L\R(\chi, \omega)$ via the following process: Add to the top of $D$ a node shaded $\omega(1)$ on the side corresponding to $\chi(1)$ and extend all spines of $D$ to the top. If at least one spine was extended and the spine nearest to the new node is shaded $\omega(1)$, then connect the nearest spine to the node with a rib and choose to either extend the spine to the top or not. Otherwise leave the new node isolated, or connect the new node with a rib to a new spine shaded $\omega(1)$ to the top.
\end{itemize}
For $0 \leq k \leq n$, let $\L\R_k(\chi, \omega)$ denote the subset of $\L\R(\chi, \omega)$ with exactly $k$ spines reaching the top.

\item For $0 \leq m \leq n$, let $\L\R_m^{\lat\capp}(\chi, \omega)$ denote the set of all diagrams with $m$ strings reaching the top that can be obtained from some diagram $D \in \L\R_k(\chi, \omega)$ for some $k\geq m$ by removing potions of spines in $D$ either between ribs or between a rib and the top of the diagram. For $D' \in \L\R_m^{\lat\capp}(\chi, \omega)$ and $D \in \L\R_k(\chi, \omega)$, $D'$ is said to be a \textit{lateral capping of} $D$, denoted $D \geq_{\lat\capp} D'$, if $D' = D$ or $D'$ can be obtained from $D$ by removing potions of spines. Moreover, let
\[\L\R^{\lat\capp}(\chi, \omega) = \bigcup_{m = 0}^n\L\R_m^{\lat\capp}(\chi, \omega).\]

\item For $D \in \L\R_m^{\lat\capp}(\chi, \omega)$, let $|D| = (\text{number of blocks of } D) + m$.
\end{enumerate}
\end{defn}

Finally, if $a_1, \dots, a_n$ are random variables in a two-state non-commutative probability space $(\A, \varphi, \psi)$, and $D \in \L\R^{\lat\capp}(\chi, \omega)$ with blocks $V_1, \dots, V_p$ whose spines do not reach the top and $W_1, \dots, W_q$ whose spines reach the top, writing $V_i = \{v_{i, 1} < \cdots < v_{i, s_i}\}$ and $W_j = \{w_{j, 1} < \cdots < w_{j, t_j}\}$ we set
\[\varphi_D(a_1, \dots, a_n) := \prod_{i = 1}^p\psi(a_{v_{i, 1}}\cdots a_{v_{i, s_i}})\prod_{j = 1}^q\varphi(a_{w_{j, 1}}\cdots a_{w_{j, t_j}}).\]

Under the above notation, the following moment type characterization and vanishing of mixed cumulants characterization were established in \cite{GS2016}*{Theorems 4.1 and 4.8}. Note that in equation \eqref{psi-Moment} below, the notation $\sigma \leq \omega$ denotes $\sigma$ is a refinement of the partition induced by $\omega$ with blocks $\{\omega^{-1}(\{k\})\}_{k \in K}$.

\begin{thm}\label{CBFMoments}
A family $\{(A_{k, \ell}, A_{k, r})\}_{k \in K}$ of pairs of algebras in a two-state non-commutative probability space $(\A, \varphi, \psi)$ is c-bi-free with respect to $(\varphi, \psi)$ if and only if
\begin{equation}\label{psi-Moment}
\psi(a_1\cdots a_n) = \sum_{\pi \in \B\N\C(\chi)}\left(\sum_{\substack{\sigma \in \B\N\C(\chi)\\\pi \leq \sigma \leq \omega}}\mu_{\B\N\C}(\pi, \sigma)\right)\psi_\pi(a_1, \dots, a_n)
\end{equation}
and
\begin{equation}\label{phi-Moment}
\varphi(a_1\cdots a_n) = \sum_{D \in \L\R^{\lat\capp}(\chi, \omega)}\left(\sum_{\substack{D' \in \L\R(\chi, \omega)\\D' \geq_{\lat\capp} D}}(-1)^{|D| - |D'|}\right)\varphi_D(a_1, \dots, a_n)
\end{equation}
for all $n \geq 1$, $\chi: \{1, \dots, n\} \to \{\ell, r\}$, $\omega: \{1, \dots, n\} \to K$, and $a_1, \dots, a_n \in \A$ with $a_j \in A_{\omega(j), \chi(j)}$.

Equivalently, for all $n \geq 2$, $\chi: \{1, \dots, n\} \to \{\ell, r\}$, $\omega: \{1, \dots, n\} \to K$, and $a_j$ as above, we have that
\[\kappa_\chi(a_1, \dots, a_n) = \K_\chi(a_1, \dots, a_n) = 0\]
whenever $\omega$ is not constant.
\end{thm}

\section{Bi-Boolean independence}\label{sec:biboolean}

In this section, the notions of bi-Boolean independence and B-$(\ell, r)$-cumulants are introduced.

\subsection{Definition and connection with c-bi-free independence}

Unlike bi-free and c-bi-free independences, the definition of bi-Boolean independence is given as a (more straightforward) rule of calculating mixed moments. We begin with the following definition of a special partition.

\begin{defn}\label{Partition}
Given $n \geq 1$, $\chi: \{1, \dots, n\} \to \{\ell, r\}$, and $\omega: \{1, \dots, n\} \to K$, let $\pi_{\omega, \chi}$ be the unique partition of $\{1, \dots, n\}$ with ordered blocks $V_1, \dots, V_m$ such that each $V_k$ is a $\chi$-interval, $\max_{\prec_\chi}(V_k) \prec_\chi \min_{\prec_\chi}(V_{k + 1})$, $\omega$ is constant on each $V_k$, and $\omega(V_k) \neq \omega(V_{k + 1})$ for all $1 \leq k \leq m - 1$.  That is, $\pi_{\omega, \chi}$ is the unique maximal partition $\pi$ such that $\pi \leq \{\omega^{-1}(k)\}_{k \in K}$ and each block of $\pi$ is a $\chi$-interval.
\end{defn}

As an example, suppose $K = \{k_1, k_2\}$, $n = 8$, $(\chi(1), \dots, \chi(8)) = (\ell, r, r, \ell, r, \ell, \ell, r)$, and $(\omega(1), \dots, \omega(8)) = (k_1, k_1, k_2, k_1, k_2, k_2, k_1, k_1)$. Then $\pi_{\omega, \chi} = \{ \{1, 4\}, \{6\},  \{7, 8\},  \{3, 5\}, \{2\}\}$. This can be easily seen via the following diagram induced by $\chi$ and $\omega$ where solid lines and $\bullet$ denote the shade of $k_1$, and dotted lines and $\circ$ denote the shade of $k_2$:
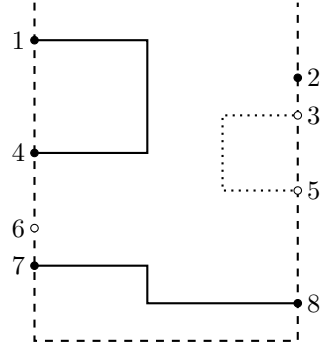
\begin{align*}
\begin{tikzpicture}[baseline]
\draw[thick, dashed, black] (0,4) -- (0,-.5) -- (3.5,-.5) -- (3.5,4);
\node[left] at (0,3.5) {1};
\draw[black,fill=black] (0,3.5) circle (0.05);
\node[right] at (3.5,3) {2};
\draw[black,fill=black] (3.5,3) circle (0.05);
\node[right] at (3.5,2.5) {3};
\draw[black, fill=white] (3.5,2.5) circle (0.05);
\node[left] at (0,2) {4};
\draw[black,fill=black] (0,2) circle (0.05);
\node[right] at (3.5,1.5) {5};
\draw[black, fill=white] (3.5,1.5) circle (0.05);
\node[left] at (0,1) {6};
\draw[black, fill=white] (0,1) circle (0.05);
\node[left] at (0,0.5) {7};
\draw[black,fill=black] (0,0.5) circle (0.05);
\node[right] at (3.5,0) {8};
\draw[black,fill=black] (3.5,0) circle (0.05);
\draw[thick, black] (0,3.5) -- (1.5,3.5) -- (1.5,2) -- (0,2);
\draw[thick, black] (0,0.5) -- (1.5,0.5) -- (1.5,0) -- (3.5,0);
\draw[thick, dotted, black] (3.45,2.5) -- (2.5,2.5) -- (2.5,1.5) -- (3.45,1.5);
\end{tikzpicture}
\end{align*}

\begin{defn}
Let $(\A, \varphi)$ be a non-commutative probability space.  A family $\{(\A_{k, \ell}, \A_{k, r})\}_{k \in K}$ of  pairs of non-unital algebras in $(\A, \varphi)$ is said to be \textit{bi-Boolean independent} with respect to $\varphi$ if for all $n \geq 1$, $\chi: \{1, \dots, n\} \to \{\ell, r\}$, $\omega: \{1, \dots, n\} \to K$, and $a_1, \dots, a_n \in \mathcal{A}$ with $a_j \in \A_{\omega(j), \chi(j)}$, we have that
\[\varphi(a_1\cdots a_n) = \varphi_{\pi_{\omega, \chi}}(a_1, \dots, a_n),\]
where $\pi_{\omega, \chi}$ is the partition of $\{1, \dots, n\}$ induced by $\chi$ and $\omega$ as described in Definition \ref{Partition}.
\end{defn}

\begin{rem}
Let $\{(\A_{k, \ell}, \A_{k, r})\}_{k \in K}$ be a bi-Boolean independent family in $(\A, \varphi)$, let $n \geq 1$, $\chi: \{1, \dots, n\} \to \{\ell, r\}$, $\omega: \{1, \dots, n\} \to K$, and $a_1, \dots, a_n \in \A$ with $a_j \in \A_{\omega(j), \chi(j)}$.
\begin{enumerate}[$\qquad(1)$]
\item If $\chi$ is constant and $\omega(k) \neq  \omega(k+1)$ for all $k$, then $\pi_{\omega, \chi} = \{\{1\}, \dots, \{n\}\}$ so that $\varphi(a_1\cdots a_n) = \varphi(a_1)\cdots\varphi(a_n)$. Therefore, the family $\{\A_{k, \ell}\}_{k \in K}$ (respectively $\{\A_{k, r}\}_{k \in K}$) of left algebras (respectively right algebras) is Boolean independent with respect to $\varphi$.

\item Suppose $n$ is even and that $\chi$ and $\omega$ are alternating; that is $\chi^{-1}(\{\ell\}) = \{1, 3, \dots, n - 1\}$ and 
\[
\omega(1) = \omega(3) = \cdots = \omega(n - 1) \neq \omega(2) = \omega(4) = \cdots = \omega(n).
\]
Then $\pi_{\omega, \chi} = \{\{1, 3, \dots, n - 1\}, \{2, 4, \dots, n\}\}$ so that 
\[
\varphi(a_1\cdots a_n) = \varphi(a_1a_3\cdots a_{n - 1})\varphi(a_2a_4\cdots a_n).
\]
Therefore, a left algebra $\A_{k_1, \ell}$ and a right algebra $\A_{k_2, r}$ are classically independent with respect to $\varphi$ (in the sense of \cite{V2014}*{Proposition 2.16}) whenever $k_1 \neq k_2$.
\end{enumerate}
\end{rem}

Bi-Boolean independence is a specific instance of c-bi-free independence as illustrated below.

\begin{prop}\label{BiBandCBF}
Let $\{(\A_{k, r}, \A_{k, r})\}_{k \in K}$ be a family of pairs of non-unital algebras in a non-commutative probability space $(\A, \varphi)$.  Then there exists a two-state non-commutative probability space $(\A', \varphi', \psi')$ such that $\A'$ contains $\A_0 = \ast_{k \in K} (\A_{k, r} \ast \A_{k, r})$, $\psi'|_{\A_0} \equiv 0$, and, if $\pi : \A_0 \to \A$ is the inclusion of $\A_0$ into $\A$, then $\varphi'|_{\A_0} = \varphi \circ \pi$.  Moreover $\{(\A_{k, r}, \A_{k, r})\}_{k \in K}$ is bi-Boolean independent with respect to $\varphi$ if and only if $\{(\A_{k, r}, \A_{k, r})\}_{k \in K}$ is c-bi-free with respect to $(\varphi', \psi')$.
\end{prop}
\begin{proof}
Let $\{(\A_{k, \ell}, \A_{k, r})\}_{k \in K}$ be a family of non-unital pairs of algebras.  For each $k \in K$ let $\mu_k: \A_{k, \ell} * \A_{k, r} \to \bC$ be the linear functional induced by $\varphi$. For $k \in K$, let $\widetilde{\A_{k, \ell}} := \bC 1 \oplus \A_{k, \ell}$ and $\widetilde{\A_{k, r}} := \bC 1 \oplus \A_{k, r}$ be the unitizations of $\A_{k, \ell}$ and $\A_{k, r}$ respectively.
Via the identification $\widetilde{\A_{k, \ell}} * \widetilde{\A_{k, r}} \cong \widetilde{\A_{k, \ell} * \A_{k, r}}$,  let $\widetilde{\mu}_k$ be the unique unital linear extension of $\mu_k$ to $\widetilde{\A_{k, \ell}} * \widetilde{\A_{k, r}}$, and let $\delta_k: \widetilde{\A_{k, \ell}} * \widetilde{\A_{k, r}} \to \bC$ be the delta state (that is, $\delta_k(\alpha 1 + a) = \alpha$). Furthermore, let $\mu, \nu: *_{k \in K}(\widetilde{\A_{k, \ell}} * \widetilde{\A_{k, r}}) \to \bC$ be the unique unital linear functionals such that $\nu = \ast_{k \in K} \delta_k$, and $\mu$ is the unique unital linear extension of $\varphi \circ \pi$ to $*_{k \in K}(\widetilde{\A_{k, \ell}} * \widetilde{\A_{k, r}})$.

Notice that $\nu(a_1\cdots a_n) = 0$ for all $n \geq 1$, $\chi: \{1, \dots, n\} \to \{\ell, r\}$, $\omega: \{1, \dots, n\} \to K$ and $a_j \in \A_{\omega(j), \chi(j)}$ by construction.  Hence equation \eqref{psi-Moment} will always be satisfied.  Therefore  $\{(\widetilde{\A_{k, \ell}}, \widetilde{\A_{k, r}})\}_{k \in K}$ is c-bi-free with respect to $(\mu, \nu)$ if and only if  for all $n \geq 1$, $\chi: \{1, \dots, n\} \to \{\ell, r\}$, $\omega: \{1, \dots, n\} \to K$, and $a_j \in \A_{\omega(j), \chi(j)}$ we have that
\[
\mu(a_1\cdots a_n) = \sum_{D \in \L\R^{\lat\capp}(\chi, \omega)}\left(\sum_{\substack{D' \in \L\R(\chi, \omega)\\D' \geq_{\lat\capp} D}}(-1)^{|D| - |D'|}\right)\mu_D(a_1, \dots, a_n).
\]
Note that if $D \in \L\R^{\lat\capp}(\chi, \omega)$ such that there is a block $V$ of $D$ whose spine does not reach the top, then $\nu_{D|_V}((a_1, \dots, a_n)|_V) = 0$ is a factor of $\mu_D(a_1, \dots, a_n)$. Thus the only non-vanishing $\mu_D(a_1, \dots, a_n)$ corresponds to the unique diagram $D \in \L\R^{\lat\capp}(\chi, \omega)$ such that every block of $D$ has a spine reaching the top.  For this $D$, the only $D' \in \L\R(\chi, \omega)$ such that $D' \geq_{\lat\capp} D$ is $D' = D$ so the coefficient of $\mu_D(a_1, \dots, a_n)$ is 1.  It is then elementary to see that
\[
\mu(a_1\cdots a_n) = \mu_D(a_1\cdots a_n) = \mu_{\pi_{\omega, \chi}}(a_1, \dots, a_n).
\]
Hence $\{(\A_{k, r}, \A_{k, r})\}_{k \in K}$ is bi-Boolean independent with respect to $\varphi$ if and only if $\{(\A_{k, r}, \A_{k, r})\}_{k \in K}$ is c-bi-free with respect to $(\varphi', \psi')$.
\end{proof}

\subsection{Combinatorial bi-Boolean independence and B-$(\ell, r)$-cumulants}

One can see by combining Proposition \ref{BiBandCBF} and the c-$(\ell, r)$-cumulants that it is easy to describe bi-Boolean cumulants.

\begin{defn}
Let $n \geq 1$ and $\chi: \{1, \dots, n\} \to \{\ell, r\}$. A partition $\pi \in \B\N\C (\chi)$ is said to be a \textit{bi-interval partition} if every block of $\pi$ is a $\chi$-interval. The set of bi-interval partitions is denoted by $\B\I(\chi)$, which is a sublattice of $\B\N\C(\chi)$ with the same minimal and maximal elements $0_\chi$ and $1_\chi$ respectively.
\end{defn}

\begin{prop}
Let $(\A, \varphi)$ be a non-commutative probability space. There exists a family of multilinear functionals $\{B_\chi: \A^n \to \bC\}_{n \geq 1, \chi: \{1, \dots, n\} \to \{\ell, r\}}$, called the \textit{B-$(\ell, r)$-cumulants}, 
uniquely determined by the requirement that for all $n \geq 1$, $\chi: \{1, \dots, n\} \to \{\ell, r\}$, and $a_1, \dots, a_n \in \A$
\begin{equation}\label{BB-M-C}
\varphi(a_1\cdots a_n) = \sum_{\pi \in \B\I(\chi)}B_\pi(a_1, \dots, a_n)
\end{equation}
where
\[
B_\pi(a_1, \dots, a_n)= \prod_{V \in \pi}B_{\chi|_V}((a_1, \dots, a_n)|_V).
\]
\end{prop}

\begin{proof}
The family $\{B_\chi: \A^n \to \bC\}_{n \geq 1, \chi: \{1, \dots, n\} \to \{\ell, r\}}$ is defined recursively similar to the way how many other families of cumulants are defined. More specifically, we let $B_{\chi_\ell} = B_{\chi_r} = \varphi$ for both $\chi_\ell: \{1\} \to \{\ell\}$ and $\chi_r: \{1\} \to \{r\}$, and then recursively define
\[B_\chi(a_1, \dots, a_n) = \varphi(a_1\cdots a_n) - \sum_{\substack{\pi \in \B\mathcal{I}(\chi)\\\pi \neq 1_\chi}}\left(\prod_{V \in \pi}B_{\chi|_V}((a_1, \dots, a_n)|_V)\right)\]
for all $n \geq 2$, $\chi: \{1, \dots, n\} \to \{\ell, r\}$, and $a_1, \dots, a_n \in \A$.
\end{proof}

Unsurprisingly the bi-Boolean cumulants characterize bi-Boolean independence.

\begin{thm}\label{VanishingEquiv}
A family $\{(\A_{k, \ell}, \A_{k, r})\}_{k \in K}$ of non-unital pairs of algebras in a non-commutative probability space $(\A, \varphi)$ is bi-Boolean independent with respect to $\varphi$ if and only if for all $n \geq 2$, $\chi: \{1, \dots, n\} \to \{\ell, r\}$, $\omega: \{1, \dots, n\} \to K$, and $a_1, \dots, a_n \in \mathcal{A}$ with $a_j \in \A_{\omega(j), \chi(j)}$, we have that
\[B_\chi(a_1, \dots, a_n) = 0\]
whenever $\omega$ is not constant.
\end{thm}

\begin{proof}
Let $\widetilde{\A} = \bC 1 \oplus \A$, let $\widetilde{\varphi}: \widetilde{\A} \to \bC$ be the unique unital linear extension of $\varphi$ to $\widetilde{\A}$, and let $\widetilde{\psi}: \widetilde{\A} \to \bC$ be the unital linear functional $\widetilde{\psi} = 1_{\bC 1} \oplus 0_\A$.  Note that $(\widetilde{\A}, \widetilde{\varphi}, \widetilde{\psi})$ is a two-state non-commutative probability space containing $\{(\A_{k, \ell}, \A_{k, r})\}_{k \in K}$ via the identification $a \mapsto 0 \oplus a$. Furthermore, let
\[
\{\widetilde{\kappa}_\chi: \widetilde{\A}^n \to \bC\}_{n \geq 1, \chi:\{1, \dots, n\} \to \{\ell, r\}} \qand \{\widetilde{\K}_\chi: \widetilde{\A}^n \to \bC\}_{n \geq 1, \chi: \{1, \dots, n\} \to \{\ell, r\}}
\]
be the families of $(\ell, r)$- and c-$(\ell, r)$-cumulants with respect to $\widetilde{\psi}$ and $(\widetilde{\varphi}, \widetilde{\psi})$ respectively.  As $\widetilde{\psi}|_\A \equiv 0$, $\widetilde{\kappa}_\chi(a_1, \dots, a_n) = 0$ for all $n \geq 1$ and $a_1, \dots, a_n \in \A$. Therefore, we have for all $n \geq 1$, $\chi: \{1, \dots, n\} \to \{\ell, r\}$, $\omega: \{1, \dots, n\} \to K$, and $a_1, \dots, a_n \in \A$ with $a_j \in \A_{\omega(j), \chi(j)}$ that
\begin{align*}
\varphi(a_1\cdots a_n) = \widetilde{\varphi}(a_1\cdots a_n) &= \sum_{\pi \in \B\N\C(\chi)}\left(\prod_{\substack{V \in \pi\\V\,\mathrm{interior}}}0\right)\left(\prod_{\substack{V \in \pi\\V\,\mathrm{exterior}}}\widetilde{\K}_{\chi|_V}((a_1, \dots, a_n)|_V)\right)\\
&= \sum_{\pi \in \B\I(\chi)}\widetilde{\K}_\pi(a_1, \dots, a_n).
\end{align*}
Hence $\widetilde{\K}_\chi(a_1, \dots, a_n) = B_\chi(a_1, \dots, a_n)$  by the uniqueness of the B-$(\ell, r)$-cumulants. By Proposition \ref{BiBandCBF}, $\{(\A_{k, \ell}, \A_{k, r})\}_{k \in K}$ is bi-Boolean independent with respect to $\widetilde{\varphi}$ if and only if it is c-bi-free with respect to $(\widetilde{\varphi}, \widetilde{\psi})$, and the result follows from Theorem \ref{CBFMoments}.
\end{proof}

\subsection{The lattice of bi-interval partitions}

As the proofs for c-bi-freeness in \cite{GS2016} are rather difficult, to conclude this section, we present an independent proof of Theorem \ref{VanishingEquiv}.  In addition, the tools developed will be of great use in subsequent sections.  

The \textit{lattice of bi-interval partitions} is
\[\B\I := \bigcup_{n \geq 1}\bigcup_{\chi: \{1, \dots, n\} \to \{\ell, r\}}\B\I(\chi),\]
where the lattice structure on $\B\I(\chi)$ is induced by the partial order $\leq$ of refinement with minimal and maximal elements $0_\chi$ and $1_\chi$ respectively. Given the lattice $\B\I$, the \textit{incidence algebra on $\B\I$}, denoted $\I\A(\B\I)$, is the set of all functions
\[f: \bigcup_{n \geq 1}\left(\bigcup_{\chi: \{1, \dots, n\} \to \{\ell, r\}}\B\I(\chi) \times \B\I(\chi)\right) \to \bC\]
such that $f(\sigma, \pi) = 0$ if $\sigma \not\leq \pi$ equipped with pointwise addition and a convolution product defined by
\[(f * g)(\sigma, \pi) = \sum_{\substack{\tau \in \B\I(\chi)\\\sigma \leq \tau \leq \pi}}f(\sigma, \tau)g(\tau, \pi)\]
for all $\sigma, \pi \in \B\I(\chi)$ and $f, g \in \I\A(\B\I)$. It is easy to verify that $\I\A(\B\I)$ is an associative algebra.

There are three important functions in $\I\A(\B\I)$: namely the \textit{delta function on $\B\I$}, denoted $\delta_{\B\I}$, the \textit{zeta function on $\B\I$}, denoted $\zeta_{\B\I}$, and the \textit{M\"{o}bius function on $\B\I$}, denoted $\mu_{\B\I}$.   The delta and zeta functions on $\B\I$ are defined by
\[
\delta_{\B\I}(\sigma, \pi) = \begin{cases}
1 &\text{if } \sigma = \pi\\
0 &\text{otherwise}
\end{cases} \qqand \zeta_{\B\I}(\sigma, \pi) = \begin{cases}
1 &\text{if } \sigma \leq \pi\\
0 &\text{otherwise}
\end{cases},
\]
whereas $\mu_{\B\I}$ is recursively defined by
\[
\sum_{\substack{\tau \in \B\I(\chi)\\\sigma \leq \tau \leq \pi}}\mu_{\B\I}(\tau, \pi) = \sum_{\substack{\tau \in \B\I(\chi)\\\sigma \leq \tau \leq \pi}}\mu_{\B\I}(\sigma, \tau) = \begin{cases}
1 &\text{if } \sigma = \pi\\
0 &\text{otherwise}
\end{cases}
\]
for $\sigma, \pi \in \B\I(\chi)$ with $\sigma \leq \pi$. 
Note $\delta_{\B\I}$ is the identity element in $\I\A(\B\I)$ and that $\zeta_{\B\I}$ and $\mu_{\B\I}$ are inverses in $\I\A(\B\I)$.  

Due to similar lattice structures, we note that
\[\mu_{\B\I}(\sigma, \pi) = \mu_\I(s_\chi^{-1}\cdot\sigma, s_\chi^{-1}\cdot\pi) = (-1)^{|\sigma| - |\pi|},\]
where $\mu_\I$ is the M\"{o}bius function on the lattice $\I$ of all interval partitions. In particular, $\mu_{\B\I}$ is multiplicative; that is, if $\sigma, \pi \in \B\I(\chi)$ are such that $\sigma \leq \pi$ and if $V_1, \dots, V_m$ are the blocks of $\pi$, then
\[\mu_{\B\I}(\sigma, \pi) = \mu_{\B\I}(\sigma|_{V_1}, \pi|_{V_1})\cdots\mu_{\B\I}(\sigma|_{V_m}, \pi|_{V_m}).\]

Given $\pi \in \B\I(\chi)$, equation \eqref{BB-M-C} implies
\[\varphi_\pi(a_1, \dots, a_n) = \sum_{\substack{\sigma \in \B\I(\chi)\\\sigma \leq \pi}}B_\sigma(a_1, \dots, a_n).\]
Using the M\"{o}bius function $\mu_{\B\I}$, the above equation can be inverted as
\begin{align*}
\sum_{\substack{\sigma \in \B\I(\chi)\\\sigma \leq \pi}}\varphi_\sigma(a_1, \dots, a_n)\mu_{\B\I}(\sigma, \pi) &= \sum_{\substack{\sigma \in \B\I(\chi)\\\sigma \leq \pi}}\sum_{\substack{\tau \in \B\I(\chi)\\\tau \leq \sigma}}B_\tau(a_1, \dots, a_n)\mu_{\B\I}(\sigma, \pi)\\
&= \sum_{\substack{\tau \in \B\I(\chi)\\\tau \leq \pi}}\left(\sum_{\substack{\sigma \in \B\I(\chi)\\\tau \leq \sigma \leq \pi}}\mu_{\B\I}(\sigma, \pi)\right)B_\tau(a_1, \dots, a_n)\\
&= B_\pi(a_1, \dots, a_n).
\end{align*}

\begin{proof}[Second proof of Theorem \ref{VanishingEquiv}]
Given $n \geq 1$, $\chi: \{1, \dots, n\} \to \{\ell, r\}$, and $\omega: \{1, \dots, n\} \to K$, note that if $\pi \in \B\I(\chi)$ and $\pi \leq \omega$, then $\pi \leq \pi_{\omega, \chi} \in \B\I(\chi)$. Suppose first that $\{(\A_{k, \ell}, \A_{k, r})\}_{k \in K}$ has vanishing mixed bi-Boolean cumulants.  Then
\begin{align*}
\varphi(a_1\cdots a_n) &= \sum_{\sigma \in \B\mathcal{I}(\chi)}B_\sigma(a_1, \dots, a_n) \\
& = \sum_{\substack{\sigma \in \B\mathcal{I}(\chi)\\\sigma \leq \omega}}B_\sigma(a_1, \dots, a_n) \\
&= \sum_{\substack{\sigma \in \B\mathcal{I}(\chi)\\\sigma \leq \pi_{\omega, \chi}}}B_\sigma(a_1, \dots, a_n)\\
&= \sum_{\substack{\sigma \in \B\mathcal{I}(\chi)\\\sigma \leq \pi_{\omega, \chi}}}\sum_{\substack{\pi \in \B\mathcal{I}(\chi)\\\pi \leq \sigma}}\varphi_\pi(a_1, \dots, a_n)\mu_{\B\mathcal{I}}(\pi, \sigma)\\
&= \sum_{\substack{\pi \in \B\mathcal{I}(\chi)\\\pi \leq \pi_{\omega, \chi}}}\left(\sum_{\substack{\sigma \in \B\mathcal{I}(\chi)\\\pi \leq \sigma \leq \pi_{\omega, \chi}}}\mu_{\B\mathcal{I}}(\pi, \sigma)\right)\varphi_\pi(a_1, \dots, a_n)\\
&= \varphi_{\pi_{\omega, \chi}}(a_1, \dots, a_n).
\end{align*}
Hence $\{(\A_{k, \ell}, \A_{k, r})\}_{k \in K}$ is bi-Boolean independent with respect to $\varphi$ by definition.

Conversely, suppose $\{(\A_{k, \ell}, \A_{k, r})\}_{k \in K}$ is bi-Boolean independent with respect to $\varphi$.  Then
\[
\sum_{\sigma \in \B\mathcal{I}(\chi)}B_\sigma(a_1, \dots, a_n) = \varphi(a_1\cdots a_n) = \sum_{\substack{\sigma \in \B\mathcal{I}(\chi)\\\sigma \leq \omega}}B_\sigma(a_1, \dots, a_n)
\]
where the first equality follows from the bi-Boolean moment-cumulant formula and the second equality follows from reversing the above computations. We will proceed inductively to show that $\{(\A_{k, \ell}, \A_{k, r})\}_{k \in K}$ has vanishing mixed B-$(\ell, r)$-cumulants with respect to $\varphi$.  Note the base case $n = 2$ is trivial. For the inductive step, if $\omega: \{1, \dots, n\} \to K$ is not constant notice by the inductive hypothesis that
\begin{align*}
\sum_{\sigma \in \B\mathcal{I}(\chi)}B_\sigma(a_1, \dots, a_n) &= B_\chi(a_1, \dots, a_n) + \sum_{\substack{\sigma \in \B\mathcal{I}(\chi)\\\sigma \neq 1_\chi}}B_\sigma(a_1, \dots, a_n)\\
&= B_\chi(a_1, \dots, a_n) + \sum_{\substack{\sigma \in \B\mathcal{I}(\chi)\\\sigma \leq \omega}}B_\sigma(a_1, \dots, a_n).
\end{align*}
Combining these equations, we have that $B_\chi(a_1, \dots, a_n) = 0$, completing the inductive step.
\end{proof}

\section{Bi-Boolean partial transforms}\label{sec:transforms}

In this section, we study some of the simplest convolutions on bi-Boolean independent pairs.

\subsection{Single-variable convolutions and transforms}

We begin by recalling some notation and single-variable results.

\begin{defn}
Let $a$ be a random variable in a non-commutative probability space $(\A, \varphi)$. For $m \geq 1$, let $B_m(a)$ denote the $m^{\mathrm{th}}$ Boolean cumulant of $a$; that is, in the notation of B-$(\ell, r)$-cumulants, $B_m(a) = B_\chi(a, \dots, a)$ where $\chi: \{1, \dots, m\} \to \{\ell, r\}$ is constant.
\begin{enumerate}[$\qquad(1)$]
\item The \textit{Boolean $\eta$-transform} of $a$ is the power series
\[\eta_a(z) = \sum_{m \geq 1}B_m(a)z^m.\]

\item The \textit{moment series} of $a$ is the power series
\[M_a(z) = 1 + \sum_{m \geq 1}\varphi(a^m)z^m.\]

\item The \textit{Cauchy transform} of $a$ is the power series
\[G_a(z) = \varphi((z - a)^{-1}) = \frac{1}{z} + \sum_{m \geq 1}\frac{\varphi(a^m)}{z^{m + 1}} = \frac{1}{z} M_a\left(\frac{1}{z}\right).
\]
\end{enumerate}
\end{defn}

The Boolean $\eta$-transform linearizes the additive Boolean convolution; that is, if $a_1$ and $a_2$ are Boolean independent, then
\[\eta_{a_1 + a_2}(z) = \eta_{a_1}(z) + \eta_{a_2}(z).\]
Moreover, as shown in \cite{SW1997}*{Proposition 2.1}, we have the following functional equation
\[\eta_a(z) = 1 - \frac{1}{M_a(z)},\]
which allows one to calculate the distribution of the sum $a_1 + a_2$ by first computing $\eta_{a_1 + a_2}(z)$ and then obtaining $M_{a_1 + a_2}(z)$. Note that if we define the \textit{self-energy} of $a$ (see \cite{SW1997}*{Section 2}) by
\[E_a(z) := \sum_{m \geq 1}\frac{B_m(a)}{z^{m - 1}} =z\eta_a(1/z)  = z - \frac{1}{G_a(z)},\]
then the self-energy also linearizes the additive Boolean convolution.
From an analytic point of view, the self-energy is sometimes more suitable for the calculation of the additive Boolean convolution.

If $a_1$ and $a_2$ are Boolean independent random variables in a non-commutative probability space $(\A, \varphi)$, then one may also consider the distribution of the product $a_1a_2$. However, since
\[\varphi((a_1a_2)^m) = \varphi(a_1a_2\cdots a_1a_2) = \varphi(a_1)^m\varphi(a_2)^m\]
for all $m \geq 1$, the distribution is rather uninteresting as only the first moments of $a_1$ and $a_2$ appear. As noted by Franz \cite{F2008}, a more interesting convolution is obtained by assuming $a_1 - 1$ and $a_2 - 1$ are Boolean independent or, equivalently, by considering the distribution of the product $(1 + a_1)(1 + a_2)$. Under this assumption, it was proved in \cite{F2008}*{Theorem 2.2} (see also \cites{B2006, P2009} for different approaches) that the corresponding transform with the multiplicative property is the Boolean $\eta$-transform divided by $z$; that is,
\[\frac{1}{z}\eta_{(1 + a_1)(1 + a_2)}(z) = \frac{1}{z}\eta_{1 + a_1}(z)\cdot\frac{1}{z}\eta_{1 + a_2}(z).\]
We shall also adopt this convention when considering convolutions on bi-Boolean independent pairs.

\subsection{The bi-Boolean partial $\eta$-transform}

We now turn to two-variable transforms. Note that all series below are in commuting variables $z$ and $w$.

\begin{defn}
Let $(a, b)$ be a pair of elements in a non-commutative probability space $(\A, \varphi)$. For $m, n \geq 0$ with $m + n \geq 1$, let
\[B_{m, n}(a, b) = B_{\chi_{m, n}}(\underbrace{a, \dots, a}_{m\,\mathrm{times}}, \underbrace{b, \dots, b}_{n\,\mathrm{times}}),\]
where $\chi_{m, n}: \{1, \dots, m + n\} \to \{\ell, r\}$ is defined by $\chi_{m, n}(k) = \ell$ if $k \leq m$ and $\chi_{m, n}(k) = r$ if $k > m$. Note that $B_{m, 0}(a, b) = B_m(a)$ and $B_{0, n}(a, b) = B_n(b)$.
\begin{enumerate}[$\qquad(1)$]
\item The \textit{bi-Boolean partial $\eta$-transform} of $(a, b)$ is the power series
\[\eta_{(a, b)}(z, w) = \sum_{\substack{m, n \geq 0\\m + n \geq 1}}B_{m, n}(a, b)z^mw^n.\]

\item The \textit{two-variable partial moment series} of $(a, b)$ is the power series
\[M_{(a, b)}(z, w) = 1 + \sum_{\substack{m, n \geq 0\\m + n \geq 1}}\varphi(a^mb^n)z^mw^n.\]

\item The \textit{two-variable Cauchy transform} of $(a, b)$ is the power series
\[G_{(a, b)}(z, w) = \varphi((z - a)^{-1}(w - b)^{-1}) = \frac{1}{zw} + \sum_{\substack{m, n \geq 0\\m + n \geq 1}}\frac{\varphi(a^mw^n)}{z^{m + 1}w^{n + 1}}.\]
\end{enumerate}
\end{defn}

The bi-Boolean partial $\eta$-transform plays a similar role as the Boolean $\eta$-transform when it comes to the additive bi-Boolean convolution. Indeed, if $(a_1, b_1)$ and $(a_2, b_2)$ are bi-Boolean independent two-faced pairs, then by Theorem \ref{VanishingEquiv}
\[\eta_{(a_1 + a_2, b_1 + b_2)}(z, w) = \eta_{(a_1, b_1)}(z, w) + \eta_{(a_2, b_2)}(z, w).\]

We now present a functional equation relating $\eta_{(a, b)}$ to $G_{(a, b)}$.

\begin{thm}\label{PartialEta}
Let $(a, b)$ be a pair of elements in a non-commutative probability space $(\A, \varphi)$. The bi-Boolean partial $\eta$-transform $\eta_{(a, b)}$ of $(a, b)$ is given by
\[\eta_{(a, b)}(z, w) = \eta_{a}(z) + \eta_{b}(w) + \frac{G_{(a, b)}(1/z, 1/w)}{G_{a}(1/z)G_{b}(1/w)} - 1.\]
\end{thm}

\begin{proof}
This immediately follows from the operator-valued version in Theorem \ref{OpVPartialEta} below. Alternatively, one may appeal to the functional equation for the c-bi-free partial $\R$-transform in \cite{GS2016}*{Corollary 5.7} via Theorem \ref{VanishingEquiv}.
\end{proof}

Utilizing the above idea, i.e., by appealing to the c-bi-free case, \cite{GS2016}*{Lemmata 4.17 and 4.18} imply the following properties of B-$(\ell, r)$-cumulants under certain conditions.

\begin{lem}\label{Swapping}
Let $\chi: \{1, \dots, n\} \to \{\ell, r\}$ be such that $\chi(k_0) = \ell$ and $\chi(k_0 + 1) = r$ for some $k_0 \in \{1, \dots, n - 1\}$, and define $\chi': \{1, \dots n\} \to \{\ell, r\}$ by
\[\chi'(k) = \begin{cases}
r &\text{if } k = k_0\\
\ell &\text{if } k = k_0 + 1\\
\chi(k) &\text{otherwise}
\end{cases}.\]
If $a$ and $b$ are random variables in a non-commutative probability space $(\A, \varphi)$ such that $\varphi(cabc') = \varphi(cbac')$ for all $c, c' \in \A$, then
\[B_\chi(c_1, \dots, c_{k_0 - 1}, a, b, c_{k_0 + 2}, \dots, c_n) = B_{\chi'}(c_1, \dots,c_{k_0 - 1}, b, a, c_{k_0 + 2}, \dots, c_n)\]
for all $c_1, \dots, c_{k_0 - 1}, c_{k_0 + 2}, \dots, c_n \in \A$.
\end{lem}

\begin{lem}\label{Changing}
Let $\chi: \{1, \dots, n\} \to \{\ell, r\}$ be such that $\chi(n) = \ell$, and define $\chi': \{1, \dots, n\} \to \{\ell, r\}$ by
\[\chi'(k) = \begin{cases}
r &\text{if } k = n\\
\chi(k) &\text{otherwise}
\end{cases}.\]
If $a$ and $b$ are random variables in a non-commutative probability space $(\A, \varphi)$ such that $\varphi(ca) = \varphi(cb)$ for all $c \in \A$, then
\[B_\chi(c_1, \dots, c_{n - 1}, a) = B_{\chi'}(c_1, \dots, c_{n - 1}, b)\]
for all $c_1, \dots, c_{n - 1} \in \A$.
\end{lem}

As a particularly important consequence of Lemma \ref{Swapping}, if $(a, b)$ is a commuting pair in $(\A, \varphi)$, then
\[B_\chi(c_{\chi(1)}, \dots, c_{\chi(m + n)}) = B_{m, n}(a, b)\]
for all $\chi: \{1, \dots, m + n\} \to \{\ell, r\}$ such that $|\chi^{-1}(\{\ell\})| = m$ where $c_\ell = a$ and $c_r = b$. Hence, the bi-Boolean partial $\eta$-transform (as well as the two-variable partial moment series and the two-variable Cauchy transform) completely determines the joint distribution of a commuting pair.

Another immediate consequence of Lemmata \ref{Swapping} and \ref{Changing} is the following bi-Boolean analogue of \cite{CNS2015-2}*{Theorem 10.2.1}.

\begin{thm}
If $\{(\A_{k, \ell}, \A_{k, r})\}_{k \in K}$ is a family of pairs of non-unital algebras in a non-commutative probability space $(\A, \varphi)$ such that
\begin{enumerate}[$\qquad(1)$]
\item $\A_{m, \ell}$ and $\A_{n, r}$ commute in distribution for all $m, n \in K$, and

\item for every $b \in \A_{k, r}$, there exists an $a \in \A_{k, \ell}$ such that $\varphi(ca) = \varphi(cb)$ for all $c \in \A$,
\end{enumerate}
then $\{(\A_{k, \ell}, \A_{k, r})\}_{k \in K}$ is bi-Boolean independent with respect to $\varphi$ if and only if $\{\A_{k, \ell}\}_{k \in K}$ is Boolean independent with respect to $\varphi$. Therefore, if $\{\A_{k, \ell}\}_{k \in K}$ is Boolean independent with respect to $\varphi$, then so is $\{\A_{k, r}\}_{k \in K}$.
\end{thm}

\begin{rem}\label{PartialSelfEnergy}
It is sometimes more convenient to consider the \textit{partial self-energy} of a pair $(a, b)$ which is the function $E_{(a, b)}$ defined by $E_{(a, b)}(z, w) = \eta_{(a, b)}\left(1/z, 1/w\right)$.  Note the partial self-energy also linearizes the additive bi-Boolean convolution (see \cite{GS2016}*{Remark 6.30}). Using the equations $z\eta_{a}(1/z) = E_{a}(z)$ and $w\eta_{b}(1/w) = E_{b}(w)$, we have that
\[E_{(a, b)}(z, w) = \frac{1}{z}E_{a}(z) + \frac{1}{w}E_{b}(w) + \frac{G_{(a, b)}(z, w)}{G_{a}(z)G_{b}(w)} - 1.\]
\end{rem}

\subsection{Properties of B-$(\ell, r)$-cumulants}

In order to study other types of convolution on bi-Boolean independent pairs, we need to discuss some properties of B-$(\ell, r)$-cumulants. The first property deals with the situation when at least one of the arguments is a scalar. Unlike in many other non-commutative probability theories where the corresponding cumulant simply vanishes, the bi-Boolean cumulants need not vanish. The following extends \cite{P2009}*{Proposition 3.3} to the bi-Boolean setting. Note that for $\chi: \{1, \dots, n\} \to \{\ell, r\}$ and $j \in \{1, \dots, n\}$, we denote by $\chi|_{\setminus j}$ the restriction of $\chi$ to $\{1, \dots, n\} \setminus \{j\}$.

\begin{prop}\label{ScalarEntry}
Let $(\A, \varphi)$ be a non-commutative probability space. If $n \geq 2$, $a_1, \dots, a_n \in \A$, $\chi: \{1, \dots, n\} \to \{\ell, r\}$, and $a_j = 1$ for some $1 \leq j \leq n$, then
\[B_\chi(a_1, \dots, a_n) = \begin{cases}
0 &\text{if } j \in \{\min_{\prec_\chi}(\{1, \dots, n\}), \max_{\prec_\chi}(\{1, \dots, n\})\}\\
B_{\chi|_{\setminus j}}(a_1, \dots, a_{j - 1}, a_{j + 1}, \dots, a_n) &\text{otherwise}
\end{cases}.\]
\end{prop}

\begin{proof}
First suppose that $j = \min_{\prec_\chi}(\{1, \dots, n\})$.  To proceed by induction on $n$ note the base case $n = 2$ is trivial as $B_\chi(a_1, a_2) = \varphi(a_1a_2) - \varphi(a_1)\varphi(a_2)$ for all $\chi: \{1, 2\} \to \{\ell, r\}$. For the inductive step, notice by the inductive hypothesis that
\begin{align*}
\varphi(a_1\cdots a_{j - 1}1a_{j + 1}\cdots a_n) &= B_\chi(a_1, \dots, a_{j - 1}, 1, a_{j + 1}, \dots, a_n) + \sum_{\substack{\pi \in \B\I(\chi)\\\pi \neq 1_\chi}}B_\pi(a_1, \dots, a_{j - 1}, 1, a_{j + 1}, \dots, a_n)\\
&= B_\chi(a_1, \dots, a_{j - 1}, 1, a_{j + 1}, \dots, a_n) + \sum_{\substack{\pi \in \B\I(\chi)\\\{j\} \in \pi}}B_\pi(a_1, \dots, a_{j - 1}, 1, a_{j + 1}, \dots, a_n)\\
&= B_\chi(a_1, \dots, a_{j - 1}, 1, a_{j + 1}, \dots, a_n) + \sum_{\pi \in \B\I(\chi|_{\setminus j})}B_\pi(a_1, \dots, a_{j - 1}, a_{j + 1}, \dots, a_n)\\
&= B_\chi(a_1, \dots, a_{j - 1}, 1, a_{j + 1}, \dots, a_n) + \varphi(a_1\cdots a_{j - 1}a_{j + 1}\cdots a_n),
\end{align*}
from which the assertion follows.   The proof for the case $j = \max_{\prec_\chi}(\{1, \dots, n\})$ is similar.

Suppose $j \not\in \{\min_{\prec_\chi}(\{1, \dots, n\}), \max_{\prec_\chi}(\{1, \dots, n\})\}$.  Again we will proceed by induction.  For the base case $n = 3$, note either $j=2$ or $j=3$. We will assume that $j = 3$ as the case $j = 2$ is similar.  In this case, $\{\chi(1), \chi(2)\} = \{\ell, r\}$ so by the above work we have that
\[\varphi(a_1a_21) = \varphi(a_1a_2a_3) = \sum_{\pi \in \B\I(\chi)} B_\pi(a_1, a_2, 1) =  B_\chi(a_1, a_2, 1) + \varphi(a_1)\varphi(a_2).\]
On the other hand, we have that
\[\varphi(a_1a_2) = B_{\chi|_{\setminus 3}}(a_1, a_2) + \varphi(a_1)\varphi(a_2),\]
and thus $B_\chi(a_1, a_2, 1) = B_{\chi|_{\setminus 3}}(a_1, a_2)$. 

For the inductive step, let $V_\pi$ be the block of $\pi$ containing $\min_{\prec_\chi}(\{1, \dots, n\})$, we have that
\begin{align*}
&\varphi(a_1\cdots a_{j - 1}1a_{j + 1}\cdots a_n)\\
 &= \varphi(a_1\cdots a_n)\\
&= \sum_{\substack{\pi \in \B\I(\chi)\\j \not\in V_\pi}}B_\pi(a_1, \dots, a_n) + \sum_{\substack{\pi \in \B\I(\chi)\\j = \max_{\prec_\chi}(V_\pi)}}B_\pi(a_1, \dots, a_n) + \sum_{\substack{\pi \in \B\I(\chi)\\j \in V_\pi\\j \neq \max_{\prec_\chi}(V_\pi)}}B_\pi(a_1, \dots, a_n).
\end{align*}
By the $\{\min_{\prec_\chi}(\{1, \dots, n\}), \max_{\prec_\chi}(\{1, \dots, n\})\}$ case the second sum vanishes.  By the inductive hypothesis the first and third sums add up to
\[B_\chi(a_1, \dots, a_{j - 1}, 1, a_{j + 1}, \dots, a_n) + \sum_{\substack{\pi \in \B\I(\chi|_{\setminus j})\\\pi \neq 1_{\chi|_{\setminus j}}}}B_\pi(a_1, \dots, a_{j - 1}, a_{j + 1}, \dots, a_n).\]
On the other hand, we have that
\[\varphi(a_1\cdots a_{j - 1}a_{j + 1}\cdots a_n) = B_{\chi|_{\setminus j}}(a_1, \dots, a_{j - 1}, a_{j + 1}, \dots, a_n) + \sum_{\substack{\pi \in \B\I(\chi|_{\setminus j})\\\pi \neq 1_{\chi|_{\setminus j}}}}B_\pi(a_1, \dots, a_{j - 1}, a_{j + 1}, \dots, a_n),\]
from which the assertion follows.
\end{proof}

\begin{cor}\label{OnePlus}
Given a pair $(a, b)$ in a non-commutative probability space $(\A, \varphi)$, for all $m, n \geq 1$ we have that
\[B_{m, n}(1 + a, b) = \sum_{i = 0}^{m - 1}\binom{m - 1}{i}B_{i + 1, n}(a, b).\]
The obvious analogue for $B_{m, n}(a, 1 + b)$ also holds.
\end{cor}

\begin{proof}
By the multilinearity of B-$(\ell, r)$-cumulants, we have that
\[B_{m, n}(1 + a, b) = \sum_{\substack{c_j \in \{1, a\}\\1 \leq j \leq m}}B_{m, n}(c_1, \dots, c_m, \underbrace{b, \dots, b}_{n\,\mathrm{times}}),\]
and the assertion immediately follows from Proposition \ref{ScalarEntry} as $c_1$ must be $a$ and there are $m - 1$ positions remaining to have $a$ as the argument.
\end{proof}

Next, we investigate B-$(\ell, r)$-cumulants with products among their arguments. Remark that although it might be possible to find a general formula for arbitrary products by applying the product formula for c-$(\ell, r)$-cumulants in \cite{GS2016}*{Theorem 4.22}, the following result is sufficient for our purposes.

\begin{lem}\label{ProductEntry}
Let $(a_1, b_1)$ and $(a_2, b_2)$ be bi-Boolean independent two-faced pairs in a non-commutative probability space $(\A, \varphi)$. The following properties of B-$(\ell, r)$-cumulants hold.
\begin{enumerate}[$\qquad(1)$]
\item For all $m, n \geq 0$, $0 \leq k \leq m$, $c_i \in \{a_1, a_2, a_1a_2\}$, and $d_j \in \{b_1, b_2, b_1b_2, b_2b_1\}$,
\begin{align*}
&B_{m + 1, n}(c_1, \dots, c_k, a_1a_2, c_{k + 1}, \dots, c_m, d_1, \dots, d_n)\\
&= B_{k + 1}(c_1, \dots, c_k, a_1)B_{m - k + 1, n}(a_2, c_{k + 1}, \dots, c_m, d_1, \dots, d_n).
\end{align*}

\item For all $m, n \geq 0$, $0 \leq k \leq m$, $c_i \in \{a_1, a_2, a_1a_2\}$, and $d_j \in \{b_1, b_2, b_1b_2, b_2b_1\}$,
\begin{align*}
&B_{m + 2, n}(c_1, \dots, c_k, a_1a_2, a_1a_2, c_{k + 1}, \dots, c_m, d_1, \dots, d_n)\\
&= B_{m + 2, n}(c_1, \dots, c_k, a_2, a_1a_2, c_{k + 1}, \dots, c_m, d_1, \dots, d_n)\\
&= B_{m + 2, n}(c_1, \dots, c_k, a_1a_2, a_1, c_{k + 1}, \dots, c_m, d_1, \dots, d_n) = 0.
\end{align*}
\end{enumerate}
\end{lem}

\begin{proof}
For assertion $(1)$, we proceed by induction on $m$. If $m = 0$, then
\begin{align*}
B_{1, n}(a_1a_2, d_1, \dots, d_n) &= \varphi(a_1a_2d_1\cdots d_n) - \sum_{\substack{\pi \in \B\I(\chi_{1, n})\\\pi \neq 1_{\chi_{1, n}}}}B_\pi(a_1a_2, d_1, \dots, d_n)\\
&= \varphi(a_1a_2d_1\cdots d_n) - \sum_{\substack{\pi \in \B\I(\chi_{1, n})\\\pi \neq 1_{\chi_{1, n}}}}\sum_{\substack{\sigma \in \B\I(\chi_{1, n})\\\sigma \leq \pi}}\varphi_\sigma(a_1a_2, d_1, \dots, d_n)\mu_{\B\I}(\sigma, \pi).
\end{align*}
By the definition of bi-Boolean independence, we have that
\[\varphi_\sigma(a_1a_2, d_1, \dots, d_n) = \varphi(a_1)\varphi_\sigma(a_2, d_1, \dots, d_n)\]
for all $\sigma \in \B\I(\chi_{1, n})$.  Hence
\begin{align*}
B_{1, n}(a_1a_2, d_1, \dots, d_n) &= \varphi(a_1)\left(\varphi(a_2d_1\cdots d_n) - \sum_{\substack{\pi \in \B\I(\chi_{1, n})\\\pi \neq 1_{\chi_{1, n}}}}\sum_{\substack{\sigma \in \B\I(\chi_{1, n})\\\sigma \leq \pi}}\varphi_\sigma(a_2, d_1, \dots, d_n)\mu_{\B\I}(\sigma, \pi)\right)\\
&= B_1(a_1)B_{1, n}(a_2, d_1, \dots, d_n).
\end{align*}

For the inductive step, notice
\begin{align*}
&B_{m + 1, n}(c_1, \dots, c_k, a_1a_2, c_{k + 1}, \dots, c_m, d_1, \dots, d_n)\\
&= \varphi(c_1\cdots c_ka_1a_2c_{k + 1}\cdots c_md_1\cdots d_n) - \sum_{\substack{\pi \in \B\I(\chi_{m + 1, n})\\\pi \neq 1_{\chi_{m + 1, n}}}}B_\pi(c_1, \dots, c_k, a_1a_2, c_{k + 1}, \dots, c_m, d_1, \dots, d_n)\\
&= \varphi(c_1\cdots c_ka_1)\varphi(a_2c_{k + 1}\cdots c_md_1\cdots d_n) - \sum_{\substack{\pi \in \B\I(\chi_{m + 1, n})\\\pi \neq 1_{\chi_{m + 1, n}}}}B_\pi(c_1, \dots, c_k, a_1a_2, c_{k + 1}, \dots, c_m, d_1, \dots, d_n)\\
&= \sum_{\pi_1 \in \B\I(\chi_{k + 1, 0})}B_{\pi_1}(c_1, \dots, c_k, a_1)\sum_{\pi_2 \in \B\I(\chi_{m - k + 1, n})}B_{\pi_2}(a_2, c_{k + 1}, \dots, c_m, d_1, \dots, d_n)\\
&\quad -  \sum_{\substack{\pi \in \B\I(\chi_{m + 1, n})\\\pi \neq 1_{\chi_{m + 1, n}}}}B_\pi(c_1, \dots, c_k, a_1a_2, c_{k + 1}, \dots, c_m, d_1, \dots, d_n).
\end{align*}
For every pair $(\pi_1, \pi_2)$ where $\pi_1 \in \B\I(\chi_{k + 1, 0})$ and $\pi_2 \in \B\I(\chi_{m - k + 1, n})$ with $(\pi_1, \pi_2) \neq (1_{\chi_{k + 1, 0}}, 1_{\chi_{m - k + 1, n}})$ there is a unique $\pi \in \B\I(\chi_{m + 1, n})$ obtained by taking $\pi_1 \cup \pi_1$, merging the block of $\pi_1$ containing $\max_{\prec_\chi}(\pi_1)$ with the block of $\pi_2$ containing $\min_{\prec_\chi}(\pi_2)$, and identifying $\max_{\prec_\chi}(\pi_1)$ and $\min_{\prec_\chi}(\pi_2)$ as the same element. This map is a bijection and, by the induction hypothesis, we have that
\[B_{\pi_1}(c_1, \dots, c_k, a_1)B_{\pi_2}(a_2, c_{k + 1}, \dots, c_m, d_1, \dots, d_n) = B_\pi(c_1, \dots, c_k, a_1a_2, c_{k + 1}, \dots, c_m, d_1, \dots, d_n)\]
under this bijection. Hence, the only remaining term is
\[B_{k + 1}(c_1, \dots, c_k, a_1)B_{m - k + 1, n}(a_2, c_{k + 1}, \dots, c_m, d_1, \dots, d_n)\]
after cancellation thereby yielding (1).

For assertion $(2)$, we have that
\begin{align*}
&B_{m + 2, n}(c_1, \dots, c_k, a_1a_2, a_1a_2, c_{k + 1}, \dots, c_m, d_1, \dots, d_n)\\
&= B_{k + 1}(c_1, \dots, c_k, a_1)B_2(a_2, a_1)B_{m - k + 1, n}(a_2, c_{k + 1}, \dots, c_m, d_1, \dots, d_n)\\
&= 0
\end{align*}
by assertion $(1)$ and the vanishing of mixed B-$(\ell, r)$-cumulants. The other two statements can be proved analogously by noting that regardless of what $c_1, \dots, c_k$ are, there will be a B-$(\ell, r)$-cumulant as a factor which contains both $a_1$ and $a_2$, and all other arguments are either $a_1$ or $a_2$, i.e., $a_1a_2$ is not one of the arguments of this factor.
\end{proof}

In addition, the following results hold by similar arguments.

\begin{lem}\label{ProductEntry2}
Let $(a_1, b_1)$ and $(a_2, b_2)$ be bi-Boolean independent  pairs in a non-commutative probability space $(\A, \varphi)$. The following properties of B-$(\ell, r)$-cumulants hold.
\begin{enumerate}[$\qquad(1)$]
\item For all $m, n \geq 0$, $0 \leq k \leq n$, $c_i \in \{a_1, a_2, a_1a_2, a_2a_1\}$, and $d_j \in \{b_1, b_2, b_2b_1\}$,
\begin{align*}
&B_{m, n + 1}(c_1, \dots, c_m, d_1, \dots, d_k, b_2b_1, d_{k + 1}, \dots, d_n)\\
&= B_{m, n - k + 1}(c_1, \dots, c_m, b_1, d_{k + 1}, \dots, d_n)B_{k + 1}(d_1, \dots, d_k, b_2).
\end{align*}

\item For all $m, n \geq 0$, $0 \leq k \leq n$, $c_i \in \{a_1, a_2, a_1a_2, a_2a_1\}$, and $d_j \in \{b_1, b_2, b_2b_1\}$,
\begin{align*}
&B_{m, n + 2}(c_1, \dots, c_m, d_1, \dots, d_k, b_2b_1, b_2b_1, d_{k + 1}, \dots, d_n)\\
&= B_{m, n + 2}(c_1, \dots, c_m, d_1, \dots, d_k, b_1, b_2b_1, d_{k + 1}, \dots, d_n)\\
&= B_{m, n + 2}(c_1, \dots, c_m, d_1, \dots, d_k, b_2b_1, b_2, d_{k + 1}, \dots, d_n) = 0.
\end{align*}
\end{enumerate}
\end{lem}

\subsection{The reduced bi-Boolean partial $\eta$-transform}

Recall that if $(a, b)$ is a commuting two-faced pair, then the bi-Boolean partial $\eta$-transform $\eta_{(a, b)}$ uniquely determines the joint distribution of $(a, b)$. Equivalently, if the marginal distributions of $a$ and $b$ are known, then it suffices to compute the sum $\sum_{m, n \geq 1}B_{m, n}(a, b)z^mw^n$ in order to obtain the joint distribution of $(a, b)$.

\begin{defn}
Let $(a, b)$ be a two-faced pair in a non-commutative probability space $(\A, \varphi)$. The \textit{reduced bi-Boolean partial $\eta$-transform} of $(a, b)$ is defined by
\[\widetilde{\eta}_{(a, b)}(z, w) = \sum_{m, n \geq 1}B_{m, n}(a, b)z^mw^n = \eta_{(a, b)}(z, w) - \eta_{a}(z) - \eta_{b}(w).\]
\end{defn}

For all of the convolutions considered below, the marginal distributions of the resulting pair may be obtained using the single-variable transforms mentioned above. Therefore, it remains to find the formulae for the reduced bi-Boolean partial $\eta$-transforms in terms of the individual ones. Quite surprisingly, the single-variable transforms also appear.

Let $(a_1, b_1)$ and $(a_2, b_2)$ be bi-Boolean independent pairs. We begin by computing the bi-Boolean partial $\eta$-transform of the pair $((1 + a_1)(1 + a_2), b_1 + b_2)$; that is, multiplication on the left variables and addition on the right variables. From the single-variable results, the marginal distributions of $(1 + a_1)(1 + a_2)$ and $b_1 + b_2$ may be obtained from the marginal distributions of the two-faced pairs $(1 + a_1, b_1)$ and $(1 + a_2, b_2)$ as
\[\frac{1}{z}\eta_{(1 + a_1)(1 + a_2)}(z) = \frac{1}{z}\eta_{1 + a_1}(z)\cdot\frac{1}{z}\eta_{1 + a_2}(z) \qand \eta_{b_1 + b_2}(w) = \eta_{b_1}(w) + \eta_{b_2}(w).\]
The following result provides a formula for the reduced bi-Boolean partial $\eta$-transform.

\begin{thm}\label{T-Transform}
If $(a_1, b_1)$ and $(a_2, b_2)$ are bi-Boolean independent pairs in a non-commutative probability space $(\A, \varphi)$, then 
\[\widetilde{\eta}_{((1 + a_1)(1 + a_2), b_1 + b_2)}(z, w) = \widetilde{\eta}_{(1 + a_1, b_1)}(z, w) + \frac{1}{z}\eta_{1 + a_1}(z)\cdot\widetilde{\eta}_{(1 + a_2, b_2)}(z, w)\]
and 
\[\widetilde{\eta}_{(a_1 + a_2, (1 + b_2)(1 + b_1))}(z, w) = \widetilde{\eta}_{(a_2, 1 + b_2)}(z, w) + \frac{1}{w}\eta_{1 + b_2}(w)\cdot\widetilde{\eta}_{(a_1, 1 + b_1)}(z, w).\]
\end{thm}

\begin{proof}
To begin, note for all $m, n \geq 1$ that the coefficient of $z^mw^n$ in $\widetilde{\eta}_{((1 + a_1)(1 + a_2), b_1 + b_2)}(z, w)$ is given by
\[B_{m, n}(1 + a_1 + a_2 + a_1a_2, b_1 + b_2) = \sum_{k = 0}^{m - 1}\binom{m - 1}{k}B_{k + 1, n}(a_1 + a_2 + a_1a_2, b_1 + b_2).\]
For $k \geq 0$ and $n \geq 1$, we have by Lemma \ref{ProductEntry} and by bi-Boolean independence that
\begin{align*}
&B_{k + 1, n}(a_1 + a_2 + a_1a_2, b_1 + b_2)\\
&= \sum_{c_i \in \{a_1, a_2\}}B_{k + 1, n}(c_1, \dots, c_{k + 1}, b_1 + b_2, \dots, b_1 + b_2)  \\
& \quad + \sum_{\substack{c_i \in \{a_1, a_2, a_1a_2\}\\\text{at least one }c_i = a_1a_2}}B_{k + 1, n}(c_1, \dots, c_{k + 1}, b_1 + b_2, \dots, b_1 + b_2)\\
&= B_{k + 1, n}(a_1, b_1) + B_{k + 1, n}(a_2, b_2) + \sum_{p = 0}^kB_{k + 1, n}(\underbrace{a_1, \dots, a_1}_{p\,\mathrm{times}}, a_1a_2, \underbrace{a_2, \dots, a_2}_{k - p\,\mathrm{times}}, b_1 + b_2, \dots, b_1 + b_2)\\
&= B_{k + 1, n}(a_1, b_1) + B_{k + 1, n}(a_2, b_2) + \sum_{p = 0}^kB_{p + 1}(a_1)B_{k - p + 1, n}(a_2, b_1 + b_2)\\
&= B_{k + 1, n}(a_1, b_1) + B_{k + 1, n}(a_2, b_2) + \sum_{p = 0}^kB_{p + 1}(a_1)B_{k - p + 1, n}(a_2, b_2)\\
&= B_{k + 1, n}(a_1, b_1) + B_{k + 1, n}(a_2, b_2) \sum_{p = 0}^kB_{k + 1, n}(\underbrace{a_1, \dots, a_1}_{p\,\mathrm{times}}, a_1a_2, \underbrace{a_2, \dots, a_2}_{k - p\,\mathrm{times}}, b_1 + b_2, \dots, b_2)\\
&= B_{k + 1, n}(a_1, b_1) + B_{k + 1, n}(a_1 + a_2 + a_1a_2, b_2).
\end{align*}
Hence the coefficient of $z^mw^n$ in $\widetilde{\eta}_{((1 + a_1)(1 + a_2), b_1 + b_2)}(z, w)$ is
\begin{align*}
B_{m, n}(1 + a_1 + a_2 + a_1a_2, b_1 + b_2) &= \sum_{k = 0}^{m - 1}\binom{m - 1}{k}(B_{k + 1, n}(a_1, b_1) + B_{k + 1, n}(a_1 + a_2 + a_1a_2, b_2))\\
&= B_{m, n}(1 + a_1, b_1) + B_{m, n}(1 + a_1 + a_2 + a_1a_2, b_2).
\end{align*}

Now consider $\widetilde{\eta}_{(1 + a_1, b_1)}(z, w) + \frac{1}{z}\eta_{1 + a_1}(z)\cdot\widetilde{\eta}_{(1 + a_2, b_2)}(z, w)$.  Let $\alpha_{m, n}$ denote the coefficient of $z^mw^n$ in $\frac{1}{z}\eta_{1 + a_1}(z)\cdot\widetilde{\eta}_{(1 + a_2, b_2)}(z, w)$.  Notice for all $m, n \geq 1$ that
\begin{align*}
\alpha_{m, n} &= \sum_{k = 0}^{m - 1}B_{k + 1}(1 + a_1)B_{m - k, n}(1 + a_2, b_2)\\
&= B_1(1 + a_1)B_{m, n}(1 + a_2, b_2) + \sum_{k = 1}^{m - 1}B_{k + 1}(1 + a_1)B_{m - k, n}(1 + a_2, b_2)\\
&= (1 + B_1(a_1))B_{m, n}(1 + a_2, b_2) + \sum_{k = 1}^{m - 1}\left(\sum_{p = 0}^{k - 1}\binom{k - 1}{p}B_{p + 2}(a_1)\right)\left(\sum_{q = 0}^{m - k - 1}\binom{m - k - 1}{q}B_{q + 1, n}(a_2, b_2)\right)
\end{align*}
by \cite{P2009}*{Corollary 3.4} and Corollary \ref{OnePlus}. Note the sum on the right-hand side is
\[S := \sum_{\substack{2 \leq s \leq m\\1 \leq t \leq m - 1\\3 \leq s + t \leq m + 1}}\beta_{s, t}B_s(a_1)B_{t, n}(a_2, b_2),\]
where
\[
\beta_{s, t} = \sum_{k = s - 1}^{m - t}\binom{k - 1}{s - 2}\binom{m - k - 1}{t - 1} = \sum_{k = s - 2}^{(m - 2) - (t - 1)}\binom{k}{s - 2}\binom{m - 2 - k}{t - 1} = \binom{m - 1}{s + t - 2}
\]
by the identity
\[\sum_{k = a}^{N - b}\binom{k}{a}\binom{N - k}{b} = \binom{N + 1}{a + b + 1}.\]
Therefore $\alpha_{m, n} = B_{m, n}(1 + a_2, b_2) + B_1(a_1)B_{m, n}(1 + a_2, b_2) + S$.

Notice
\[S = \sum_{k = 1}^{m - 1}\binom{m - 1}{k}\sum_{p = 0}^{k - 1}B_{p + 2}(a_1)B_{k - p, n}(a_2, b_2).\]
Therefore, by Lemma \ref{ProductEntry} and Theorem \ref{VanishingEquiv},
\begin{align*}
S &= \sum_{k = 1}^{m - 1}\binom{m - 1}{k}\sum_{p = 0}^{k - 1}B_{k + 1, n}(a_1, \underbrace{a_1, \dots, a_1}_{p\,\mathrm{times}}, a_1a_2, \underbrace{a_2, \dots, a_2}_{k - p - 1\,\mathrm{times}}, b_2, \dots, b_2)\\
&= \sum_{k = 1}^{m - 1}\binom{m - 1}{k}\sum_{c_i \in \{a_1, a_2, a_1a_2\}}B_{k + 1, n}(a_1, c_1, \dots, c_k, b_2, \dots, b_2)\\
&= \sum_{k = 1}^{m - 1}\binom{m - 1}{k}B_{k + 1, n}(a_1, \underbrace{a_1 + a_2 + a_1a_2, \dots, a_1 + a_2 + a_1a_2}_{k\,\mathrm{times}}, b_2, \dots, b_2)\\
&= B_{m, n}(1 + a_1, \underbrace{1 + a_1 + a_2 + a_1a_2, \dots, 1 + a_1 + a_2 + a_1a_2}_{m - 1\,\mathrm{times}}, b_2, \dots, b_2).
\end{align*}
As
\[B_{m, n}(a_2, \underbrace{1 + a_1 + a_2 + a_1a_2, \dots, 1 + a_1 + a_2 + a_1a_2}_{m - 1\,\mathrm{times}}, b_2, \dots, b_2) = B_{m, n}(1 + a_2, b_2)\]
and
\[B_{m, n}(a_1a_2, \underbrace{1 + a_1 + a_2 + a_1a_2, \dots, 1 + a_1 + a_2 + a_1a_2}_{m - 1\,\mathrm{times}}, b_2, \dots, b_2) = B_1(a_1)B_{m, n}(1 + a_2, b_2),\]
We obtain that $\alpha_{m, n} = B_{m, n}(1 + a_1 + a_2 + a_1a_2, b_2)$.  Hence, as the coefficient of $z^mw^n$ in $\widetilde{\eta}_{(1 + a_1, b_1)}(z, w)$ is $B_{m, n}(1 + a_1, b_1)$, the first result follows.

Note the proof of the second assertion follow by using Lemma \ref{ProductEntry2} instead of Lemma \ref{ProductEntry}.
\end{proof}

Note that if we want to consider multiplication on the right variables instead, then there is a choice of whether the product $(1 + b_1)(1 + b_2)$ or the product $(1 + b_2)(1 + b_1)$ is preferred as it is not clear whether to use the usual multiplication or the opposite multiplication. This choice was irrelevant to the above transforms due to the commutativity of addition and the ability to interchange pairs.  Thus Theorems \ref{S-Transform} and \ref{S-Transform2} will analyze the transforms for these two possibilities.  

\begin{thm}\label{S-Transform}
If $(a_1, b_1)$ and $(a_2, b_2)$ are bi-Boolean independent pairs in a non-commutative probability space $(\A, \varphi)$, then
\[\widetilde{\eta}_{((1 + a_1)(1 + a_2), (1 + b_2)(1 + b_1))}(z, w) = \frac{1}{z}\eta_{1 + a_1}(z)\cdot\widetilde{\eta}_{(1 + a_2, 1 + b_2)}(z, w) + \frac{1}{w}\eta_{1 + b_2}(w)\cdot\widetilde{\eta}_{(1 + a_1, 1 + b_1)}(z, w).\]
\end{thm}

\begin{proof}
For $m, n \geq 1$, the coefficient of $z^mw^n$ in $\widetilde{\eta}_{((1 + a_1)(1 + a_2), (1 + b_2)(1 + b_1))}(z, w)$ is given by
\[
B_{m, n}(1 + a_1 + a_2 + a_1a_2, 1 + b_1 + b_2 + b_2b_1) = \sum_{k = 0}^{m - 1}\sum_{\ell = 0}^{n - 1}\binom{m - 1}{k}\binom{n - 1}{\ell}B_{k + 1, \ell + 1}(a_1 + a_2 + a_1a_2, b_1 + b_2 + b_2b_1).\]
For $k, \ell \geq 0$, notice that
\begin{align*}
&B_{k + 1, \ell + 1}(a_1 + a_2 + a_1a_2, b_1 + b_2 + b_2b_1)\\
&= \sum_{c_i \in \{a_1, a_2\}}B_{k + 1, \ell + 1}(c_1, \dots, c_{k + 1}, b_1 + b_2 + b_2b_1, \dots, b_1 + b_2 + b_2b_1)\\
&\quad + \sum_{\substack{c_i \in \{a_1, a_2, a_1a_2\}\\\text{at least one }c_i = a_1a_2}}B_{k + 1, \ell + 1}(c_1, \dots, c_{k + 1}, b_1 + b_2 + b_2b_1, \dots, b_1 + b_2 + b_2b_1)\\
&:= S_1 + S_2.
\end{align*}
Using the vanishing of mixed bi-Boolean cumulants together with Lemmata \ref{ProductEntry} and \ref{ProductEntry2}, it is possible to calculate $S_1$ and $S_2$.  Indeed
\begin{align*}
S_1 &= \sum_{c_i \in \{a_1, a_2\}}\sum_{d_j \in \{b_1, b_2\}}B_{k + 1, \ell + 1}(c_1, \dots, c_{k + 1}, d_1, \dots, d_{\ell + 1})\\
&\quad + \sum_{c_i \in \{a_1, a_2\}}\sum_{\substack{d_j \in \{b_1, b_2, b_2b_1\}\\\text{at least one }d_j = b_2b_1}}B_{k + 1, \ell + 1}(c_1, \dots, c_{k + 1}, d_1, \dots, d_{\ell + 1})\\
&= B_{k + 1, \ell + 1}(a_1, b_1) + B_{k + 1, \ell + 1}(a_2, b_2) + \sum_{c_i \in \{a_1, a_2\}}\sum_{p = 0}^\ell B_{k + 1, \ell + 1}(c_1, \dots, c_{k + 1}, \underbrace{b_2, \dots, b_2}_{p\,\mathrm{times}}, b_2b_1, \underbrace{b_1, \dots, b_1}_{\ell - p\,\mathrm{times}})\\
&= B_{k + 1, \ell + 1}(a_1, b_1) + B_{k + 1, \ell + 1}(a_2, b_2) + \sum_{c_i \in \{a_1, a_2\}}\sum_{p = 0}^\ell B_{k + 1, \ell - p + 1}(c_1, \dots, c_{k + 1}, b_1, \dots, b_1)B_{p + 1}(b_2)\\
&= B_{k + 1, \ell + 1}(a_1, b_1) + B_{k + 1, \ell + 1}(a_2, b_2) + \sum_{p = 0}^\ell B_{k + 1, \ell - p + 1}(a_1, b_1)B_{p + 1}(b_2)\\
&= B_{k + 1, \ell + 1}(a_1, b_1) + B_{k + 1, \ell + 1}(a_2, b_2) + \sum_{p = 0}^\ell B_{k + 1, \ell + 1}(a_1, \dots, a_1, \underbrace{b_2, \dots, b_2}_{p\,\mathrm{times}}, b_2b_1, \underbrace{b_1, \dots, b_1}_{\ell - p\,\mathrm{times}})\\
&= B_{k + 1, \ell + 1}(a_1, b_1) + B_{k + 1, \ell + 1}(a_2, b_2) + \sum_{c_i \in \{a_1, a_2\}}\sum_{\substack{d_j \in \{b_1, b_2, b_2b_1\}\\\text{at least one }d_j = b_2b_1}}B_{k + 1, \ell + 1}(a_1, \dots, a_1, d_1, \dots, d_{\ell + 1}) \\
&= B_{k + 1, \ell + 1}(a_2, b_2) + B_{k + 1, \ell + 1}(a_1, b_1 + b_2 + b_2b_1).
\end{align*}
Furthermore notice
\begin{align*}
S_2 &= \sum_{\substack{c_i \in \{a_1, a_2, a_1a_2\}\\\text{at least one }c_i = a_1a_2}}\sum_{d_j \in \{b_1, b_2\}}B_{k + 1, \ell + 1}(c_1, \dots, c_{k + 1}, d_1, \dots, d_{\ell + 1})\\
&\quad + \sum_{\substack{c_i \in \{a_1, a_2, a_1a_2\}\\\text{at least one }c_i = a_1a_2}}\sum_{\substack{d_j \in \{b_1, b_2, b_2b_1\}\\\text{at least one }d_j = b_2b_1}}B_{k + 1, \ell + 1}(c_1, \dots, c_{k + 1}, d_1, \dots, d_{\ell + 1})\\
&:= T_1 + T_2.
\end{align*}
Further computations yield
\begin{align*}
T_1 &= \sum_{d_j \in \{b_1, b_2\}}\sum_{p = 0}^kB_{k + 1, \ell + 1}(\underbrace{a_1, \dots, a_1}_{p\,\mathrm{times}}, a_1a_2, \underbrace{a_2, \dots, a_2}_{k - p\,\mathrm{times}}, d_1, \dots, d_{\ell + 1})\\
&= \sum_{d_j \in \{b_1, b_2\}}\sum_{p = 0}^kB_{p + 1}(a_1)B_{k - p + 1, \ell + 1}(a_2, \dots, a_2, d_1, \dots, d_{\ell + 1})\\
&= \sum_{p = 0}^kB_{p + 1}(a_1)B_{k - p + 1, \ell + 1}(a_2, b_2)\\
&= \sum_{d_j \in \{b_1, b_2\}}\sum_{p = 0}^kB_{k + 1, \ell + 1}(\underbrace{a_1, \dots, a_1}_{p\,\mathrm{times}}, a_1a_2, \underbrace{a_2, \dots, a_2}_{k - p\,\mathrm{times}}, b_2, \dots, b_2)\\
&= B_{k + 1, \ell + 1}(a_1 + a_2 + a_1a_2, b_2) - B_{k + 1, \ell + 1}(a_2, b_2),
\end{align*}
and
\begin{align*}
T_2 &= \sum_{\substack{d_j \in \{b_1, b_2, b_2b_1\}\\\text{at least one }d_j = b_2b_1}}\sum_{p = 0}^kB_{k + 1, \ell + 1}(\underbrace{a_1, \dots, a_1}_{p\,\mathrm{times}}, a_1a_2, \underbrace{a_2, \dots, a_2}_{k - p\,\mathrm{times}}, d_1, \dots, d_{\ell + 1})\\
&= \sum_{p = 0}^k\sum_{q = 0}^\ell B_{k + 1, \ell + 1}(\underbrace{a_1, \dots, a_1}_{p\,\mathrm{times}}, a_1a_2, \underbrace{a_2, \dots, a_2}_{k - p\,\mathrm{times}}, \underbrace{b_2, \dots, b_2}_{q\,\mathrm{times}}, b_2b_1, \underbrace{b_1, \dots, b_1}_{\ell - q\,\mathrm{times}})\\
&= \sum_{p = 0}^k\sum_{q = 0}^\ell B_{p + 1}(a_1)B_{k - p + 1, \ell - q + 1}(a_2, b_1)B_{q + 1}(b_2)\\
&= 0.
\end{align*}
Therefore, we have that
\begin{align*}
&B_{k + 1, \ell + 1}(a_1 + a_2 + a_1a_2, b_1 + b_2 + b_2b_1)\\
&= S_1 + T_1 + T_2\\
&= B_{k + 1, \ell + 1}(a_2, b_2) + B_{k + 1, \ell + 1}(a_1, b_1 + b_2 + b_2b_1) + B_{k + 1, \ell + 1}(a_1 + a_2 + a_1a_2, b_2) - B_{k + 1, \ell + 1}(a_2, b_2)\\
&= B_{k + 1, \ell + 1}(a_1, b_1 + b_2 + b_2b_1) + B_{k + 1, \ell + 1}(a_1 + a_2 + a_1a_2, b_2).
\end{align*}
Hence
\begin{align*}
&B_{m, n}(1 + a_1 + a_2 + a_1a_2, 1 + b_1 + b_2 + b_2b_1)\\
&= \sum_{k = 0}^{m - 1}\sum_{\ell = 0}^{n - 1}\binom{m - 1}{k}\binom{n - 1}{\ell}(B_{k + 1, \ell + 1}(a_1, b_1 + b_2 + b_2b_1) + B_{k + 1, \ell + 1}(a_1 + a_2 + a_1a_2, b_2))\\
&= B_{m, n}(1 + a_1, 1 + b_1 + b_2 + b_2b_1) + B_{m, n}(1 + a_1 + a_2 + a_1a_2, 1 + b_2).
\end{align*}

Next, for $m, n \geq 1$, let $\alpha_{m, n}$ denote the coefficient of $z^mw^n$ in $\frac{1}{z}\eta_{1 + a_1}(z)\cdot\widetilde{\eta}_{(1 + a_2, 1 + b_2)}(z, w)$. Then $\alpha_{m, n}$ is given by
\begin{align*}
\alpha_{m, n} &= \sum_{k = 0}^{m - 1}B_{k + 1}(1 + a_1)B_{m - k, n}(1 + a_2, 1 + b_2).
\end{align*}
By an identical proof to that used in Theorem \ref{T-Transform} (where, in the proof one replaces $b_2$ with $1 + b_2)$, one obtains that
\[\alpha_{m, n} = B_{m, n}(1 + a_1 + a_2 + a_1a_2, 1 + b_2).\]
Similarly the coefficient of $z^mw^n$ in $\frac{1}{w}\eta_{1 + b_2}(w)\cdot\widetilde{\eta}_{(1 + a_1, 1 + b_1)}(z, w)$ is equal to $B_{m, n}(1 + a_1, 1 + b_1 + b_2 + b_2b_1)$ for all $m, n \geq 1$. Hence the result follows.
\end{proof}

\begin{thm}\label{S-Transform2}
If $(a_1, b_1)$ and $(a_2, b_2)$ are bi-Boolean independent pairs in a non-commutative probability space $(\A, \varphi)$, then 
\[\widetilde{\eta}_{((1 + a_1)(1 + a_2), (1 + b_1)(1 + b_2))}(z, w) = \widetilde{\eta}_{(1 + a_1, 1 + b_1)}(z, w) + \frac{1}{zw}\eta_{1 + a_1}(z)\cdot\eta_{1 + b_1}(w)\cdot\widetilde{\eta}_{(1 + a_2, 1 + b_2)}(z, w).\]
\end{thm}

\begin{proof}
Let $\alpha_{m,n}$ denote the coefficient of $z^mw^n$ in $\frac{1}{z}\eta_{1 + a_1}(z)\cdot\widetilde{\eta}_{(1 + a_2, 1 + b_2)}(z, w)\cdot\frac{1}{w}\eta_{1 + b_1}(w)$.
By the proof of Theorem \ref{S-Transform}, for $m, n \geq 1$ the coefficient of $z^mw^n$ in $\frac{1}{z}\eta_{1 + a_1}(z)\cdot\widetilde{\eta}_{(1 + a_2, 1 + b_2)}(z, w)$ is $B_{m, n}(1 + a_1 + a_2 + a_1a_2, 1 + b_2)$. Hence
\[\sum_{m, n \geq 1}\alpha_{m, n}z^mw^n = \left(\sum_{m, n \geq 1}B_{m, n}(1 + a_1 + a_2 + a_1a_2, 1 + b_2)z^mw^n\right)\frac{1}{w}\eta_{1 + b_1}(w).\]
For $m, n \geq 1$, notice
\begin{align*}
\alpha_{m, n} &= \sum_{k = 0}^{n - 1}B_{m, n - k}(1 + a_1 + a_2 + a_1a_2, 1 + b_2)B_{k + 1}(1 + b_1)\\
&= B_{m, n}(1 + a_1 + a_2 + a_1a_2, 1 + b_2)(1 + B_1(b_1))\\
&\quad + \sum_{k = 1}^{n - 1}\left(\sum_{p = 0}^{k - 1}\binom{k - 1}{p}B_{p + 2}(b_1)\right)\left(\sum_{q = 0}^{n - k - 1}\binom{n - k - 1}{q}B_{m, q + 1}(1 + a_1 + a_2 + a_1a_2, b_2)\right).
\end{align*}
By similar arguments to those used in the proof of Theorem \ref{T-Transform}, the above sum can be written as
\[S := \sum_{k = 1}^{n - 1}\binom{n - 1}{k}\sum_{p = 0}^{k - 1}B_{p + 2}(b_1)B_{m, k - p}(1 + a_1 + a_2 + a_1a_2, b_2).\]
Thus by, Lemma \ref{ProductEntry2}, we obtain that
\begin{align*}
S &= \sum_{k = 1}^{n - 1}\binom{n - 1}{k}\sum_{p = 0}^{k - 1}B_{m, k + 1}(1 + a_1 + a_2 + a_1a_2, \dots, 1 + a_1 + a_2 + a_1a_2, b_1, \underbrace{b_1, \dots, b_1}_{p\,\mathrm{times}}, b_1b_2, \underbrace{b_2, \dots, b_2}_{k - p - 1\,\mathrm{times}})\\
&= \sum_{k = 1}^n\binom{n - 1}{k}\sum_{d_j \in \{b_1, b_2, b_1b_2\}}B_{m, k + 1}(1 + a_1 + a_2 + a_1a_2, \dots, 1 + a_1 + a_2 + a_1a_2, b_1, d_1, \dots, d_k)\\
&\quad - \sum_{k = 1}^n\binom{n - 1}{k}\sum_{d_j \in \{b_1, b_2\}}B_{m, k + 1}(1 + a_1 + a_2 + a_1a_2, \dots, 1 + a_1 + a_2 + a_1a_2, b_1, d_1, \dots, d_k)\\
&:= T_1 - T_2.
\end{align*}
However 
\begin{align*}
T_1 &= \sum_{k = 1}^{n - 1}\binom{n - 1}{k}B_{m, k + 1}(1 + a_1 + a_2 + a_1a_2, \dots, 1 + a_1 + a_2 + a_1a_2, b_1, \underbrace{b_1 + b_2 + b_1b_2, \dots, b_1 + b_2 + b_1b_2}_{k\,\mathrm{times}})\\
&= B_{m, n}(1 + a_1 + a_2 + a_1a_2, \dots, 1 + a_1 + a_2 + a_1a_2, 1 + b_1, \underbrace{1 + b_1 + b_2 + b_1b_2, \dots, 1 + b_1 + b_2 + b_1b_2}_{n - 1\,\mathrm{times}})\\
&\quad - B_{m, 1}(1 + a_1 + a_2 + a_1a_2, b_1),
\end{align*}
and, since for $\ell \geq 0$ and $k \geq 1$, we have that
\begin{align*}
&\sum_{d_j \in \{b_1, b_2\}}B_{\ell + 1, k + 1}(a_1 + a_2 + a_1a_2, \dots, a_1 + a_2 + a_1a_2, b_1, d_1, \dots, d_k)\\
&= \sum_{d_j \in \{b_1, b_2\}}\sum_{c_i \in \{a_1, a_2\}}B_{\ell + 1, k + 1}(c_1, \dots, c_{\ell + 1}, b_1, d_1, \dots, d_k)\\
&\quad + \sum_{d_j \in \{b_1, b_2\}}\sum_{\substack{c_i \in \{a_1, a_2, a_1a_2\}\\\text{at least one }c_i = a_1a_2}}B_{\ell + 1, k + 1}(c_1, \dots, c_{\ell + 1}, b_1, d_1, \dots, d_k)\\
&= B_{\ell + 1, k + 1}(a_1, b_1) + \sum_{d_j \in \{b_1, b_2\}}\sum_{p = 0}^\ell B_{\ell + 1, k + 1}(\underbrace{a_1, \dots, a_1}_{p\,\mathrm{times}}, a_1a_2, \underbrace{a_2, \dots, a_2}_{\ell - p\,\mathrm{times}}, b_1, d_1, \dots, d_k)\\
&= B_{\ell + 1, k + 1}(a_1, b_1) + \sum_{d_j \in \{b_1, b_2\}}\sum_{p = 0}^\ell B_{p + 1}(a_1)B_{\ell - p + 1, k + 1}(a_2, \dots, a_2, b_1, d_1, \dots, d_k)\\
&= B_{\ell + 1, k + 1}(a_1, b_1),
\end{align*}
we obtain that
\begin{align*}
T_2 &= \sum_{\ell = 0}^{m - 1}\sum_{k = 1}^{n - 1}\binom{m - 1}{\ell}\binom{n - 1}{k}\sum_{d_j \in \{b_1, b_2\}}B_{\ell + 1, k + 1}(a_1 + a_2 + a_1a_2, \dots, a_1 + a_2 + a_1a_2, b_1, d_1, \dots, d_k) \\
&= \sum_{\ell = 0}^{m - 1}\sum_{k = 1}^{n - 1}\binom{m - 1}{\ell}\binom{n - 1}{k} B_{\ell + 1, k + 1}(a_1, b_1) \\
&= B_{m, n}(1 + a_1, 1 + b_1) - B_{m, 1}(1 + a_1, b_1).
\end{align*}
Hence, as $B_{m, 1}(1 + a_1 + a_1 + a_1a_2, b_1) = B_{m, 1}(1 + a_1, b_1)$,  
\begin{align*}
S &= B_{m, n}(1 + a_1 + a_2 + a_1a_2, \dots, 1 + a_1 + a_2 + a_1a_2, 1 + b_1, \underbrace{1 + b_1 + b_2 + b_1b_2, \dots, 1 + b_1 + b_2 + b_1b_2}_{n - 1\,\mathrm{times}})\\
&\quad - B_{m, n}(1 + a_1, 1 + b_1) \\
&= B_{m, n}(1 + a_1 + a_2 + a_1a_2, 1 + b_1 + b_2 + b_1b_2) - B_{m, n}(1 + a_1, 1 + b_1)\\
&\quad -B_{m, n}(1 + a_1 + a_2 + a_1a_2, \dots, 1 + a_1 + a_2 + a_1a_2, b_2, \underbrace{1 + b_1 + b_2 + b_1b_2, \dots, 1 + b_1 + b_2 + b_1b_2}_{n - 1\,\mathrm{times}})\\
&\quad -B_{m, n}(1 + a_1 + a_2 + a_1a_2, \dots, 1 + a_1 + a_2 + a_1a_2, b_1b_2, \underbrace{1 + b_1 + b_2 + b_1b_2, \dots, 1 + b_1 + b_2 + b_1b_2}_{n - 1\,\mathrm{times}})\\
&:= B_{m, n}(1 + a_1 + a_2 + a_1a_2, 1 + b_1 + b_2 + b_1b_2) - B_{m, n}(1 + a_1, 1 + b_1) - T_3 - T_4.
\end{align*}
By similar calculations as above, it can be verified that
\[T_3 = B_{m, n}(1 + a_1 + a_2 + a_1a_2, 1 + b_2) \qand T_4 = B_1(b_1)B_{m, n}(1 + a_1 + a_2 + a_1a_2, 1 + b_2).\]
Therefore, we obtain that
\[\alpha_{m, n} = B_{m, n}(1 + a_1 + a_2 + a_1a_2, 1 + b_1 + b_2 + b_1b_2) - B_{m, n}(1 + a_1, 1 + b_1)\]
for all $m, n \geq 1$, from which the assertion follows.
\end{proof}

\section{Additive bi-Boolean limit theorems and infinite divisibility}\label{sec:limitthms}

In this section, we study limit theorems and infinite divisibility with respect to the additive bi-Boolean convolution. Essentially all of the results below either can be directly obtained using the techniques developed in \cite{HW2016} for the bi-free case or immediately follow from the results in \cite{GS2016}*{Section 6} using the fact that the additive bi-Boolean convolution is in fact a special case of the additive c-bi-free convolution. Consequently, we shall often omit the details and only present the statements with some remarks.

\subsection{Combinatorial aspects}

Recall from \cite{V2014}*{Definition 2.5} that if $\widehat{a} = ((a_i)_{i \in I}, (a_j)_{j \in J})$ is a two-faced family in a non-commutative probability space $(\A, \varphi)$ then the \textit{joint distribution} $\mu_{\widehat{a}}$ of $\widehat{a}$ is by definition the unital linear functional
\[\mu_{\widehat{a}} := \varphi \circ \tau : \bC\langle Z_k : k \in I \sqcup J\rangle \to \bC\]
where $\tau: \bC\langle Z_k : k \in I \sqcup J\rangle \to \A$ is the unital homomorphism such that $\tau(Z_k) = a_k$ for $k \in I \sqcup J$. Given two bi-Boolean independent two-faced families $\widehat{a} = ((a_i)_{i \in I}, (a_j)_{j \in J})$ and $\widehat{b} = ((b_i)_{i \in I}, (b_j)_{j \in J})$ in $(\A, \varphi)$ with joint distributions $\mu_{\widehat{a}}$ and $\mu_{\widehat{b}}$ respectively, the \textit{additive bi-Boolean convolution} of $\mu_{\widehat{a}}$ and $\mu_{\widehat{b}}$, denoted $\mu_{\widehat{a}} \uplus\uplus \mu_{\widehat{b}}$ is defined to be the joint distribution of the two-faced family
\[\widehat{a} + \widehat{b} := ((a_i + b_i)_{i \in I}, (a_j + b_j)_{j \in J}).\]

A two-faced family $\widehat{a} =((a_i)_{i \in I}, (a_j)_{j \in J})$ in a non-commutative probability space $(\A, \varphi)$ has a \textit{bi-Boolean central limit distribution} (or \textit{centred bi-Boolean Gaussian distribution}) with covariance matrix $C$ if there exists a complex matrix $C = (C_{k, \ell})_{k, \ell \in I \sqcup J}$ such that
\begin{enumerate}[$\qquad(1)$]
\item $\varphi(a_{\alpha(1)}a_{\alpha(2)}) = C_{\alpha(1), \alpha(2)}$ for all $\alpha: \{1, 2\} \to I \sqcup J$,

\item $B_{\chi_\alpha}(a_{\alpha(1)}, \dots, a_{\alpha(n)}) = 0$ for all $n \neq 2$ and $\alpha: \{1, \dots, n\} \to I \sqcup J$, where $\chi_\alpha: \{1, \dots, n\} \to \{\ell, r\}$ such that $\chi_\alpha(k) = \ell$ if $\alpha(k) \in I$ and $\chi_\alpha(k) = r$ if $\alpha(k) \in J$ for $1 \leq k \leq n$.
\end{enumerate}
The algebraic bi-Boolean central limit theorem immediately follows from \cite{GS2016}*{Theorem 6.2} with the given sequence being bi-Boolean independent and the limiting distribution being a centred bi-Boolean Gaussian distribution. Similarly, it is easy to see that a Kac/Loeve type theorem also holds in the bi-Boolean setting (see \cite{S2016-2}*{Theorem 3.2} for the bi-free version) and a general bi-Boolean limit theorem can be deduced from \cite{GS2016}*{Theorem 6.5}, which we record as follows without proof.

\begin{thm}
For every $N \in \bN$, let $\{\widehat{a}_{N, m} = ((a_{N, m, i})_{i \in I}, (a_{N, m, j})_{j \in J})\}_{m = 1}^N$ be a sequence of bi-Boolean independent, identically distributed two-faced families in a non-commutative probability space $(\A_N, \varphi_N)$. Furthermore, let $\widehat{S}_N = ((S_{N, i})_{i \in I}, (S_{N, j})_{j \in J})$ be the two-faced family in $(\A_N, \varphi_N)$ defined by
\[S_{N, k} = \sum_{m = 1}^Na_{N, m, k}, \quad k \in I \sqcup J.\]
The following assertions are equivalent.
\begin{enumerate}[$\qquad(1)$]
\item There exists a two-faced family $\widehat{s} = ((s_i)_{i \in I}, (s_j)_{j \in J})$ in a non-commutative probability space $(\A, \varphi)$ such that $\widehat{S}_N$ converges in distribution to $\widehat{s}$ as $N \to \infty$.

\item For all $n \geq 1$ and $\alpha: \{1, \dots, n\} \to I \sqcup J$, the limits
\[\lim_{N \to \infty}N\cdot\varphi_N(a_{N, m, \alpha(1)}\cdots a_{N, m, \alpha(n)})\]
exist and are independent of $m$.
\end{enumerate}
Moreover, if these assertions hold, then the B-$(\ell, r)$-cumulants of $\widehat{s}$ are given by
\[B_{\chi_\alpha}(s_{\alpha(1)}, \dots, s_{\alpha(n)}) = \lim_{N \to \infty}N\cdot\varphi_N(a_{N, m, \alpha(1)}\cdots a_{N, m, \alpha(n)})\]
for all $n \geq 1$ and $\alpha: \{1, \dots, n\} \to I \sqcup J$.
\end{thm}

We now turn our attention to Borel probability measures on $\bR^2$. Let $\delta_{(0, 0)}$ denote the point mass at the origin of $\bR^2$.  Consider first the case of measures with compact supports. Given two such measures $\mu_1$ and $\mu_2$ the \textit{additive bi-Boolean convolution} of $\mu_1$ and $\mu_2$, denoted $\mu_1 \uplus\uplus \mu_2$, is defined to be the measure $\mu$ such that
\[(\mu, \delta_{(0, 0)}) = (\mu_1, \delta_{(0, 0)}) \boxplus\boxplus_{\mathrm{c}} (\mu_2, \delta_{(0, 0)}),\]
where $\boxplus\boxplus_{\mathrm{c}}$ denotes the additive c-bi-free convolution. If $\mu$ is a compactly supported Borel probability measure on $\bR^2$ and $m, n \geq 0$ with $m + n \geq 1$, we denote the $(m, n)^{\mathrm{th}}$ moment of $\mu$ by
\[\bM_{m, n}(\mu) := \int_{\bR^2}s^mt^n\,d\mu(s, t).\]
Moreover, since $\mu$ can be realized as the joint distribution of a self-adjoint commuting pair in a $C^*$-probability space, Lemma \ref{Swapping} implies that the B-$(\ell, r)$-cumulants of $\mu$ are completely determined by cumulants of the form $B_{m, n}(\mu)$, which can be obtained from the moments of $\mu$ by the bi-Boolean moment-cumulant formula. 

For $\lambda > 0$, denote by $D_\lambda\mu$ the dilation of $\mu$ by the factor $\lambda$ (i.e., $D_\lambda\mu(B) = \mu(\lambda^{-1}B)$ for all Borel subsets $B$ of $\bR^2$).  The probabilistic version of the bi-Boolean central limit theorem states that if $\bM_{1, 0}(\mu) = \bM_{0, 1}(\mu) = 0$, $\bM_{2, 0}(\mu) = a$, $\bM_{0, 2}(\mu) = b$, and $\bM_{1, 1}(\mu) = c$, then
\begin{equation}\label{eqn:CentLimit}
\lim_{N \to \infty}\underbrace{D_{1/\sqrt{N}}\mu \uplus\uplus \cdots \uplus\uplus D_{1/\sqrt{N}}\mu}_{N\,\mathrm{times}} = \mu_{(a, b, c)},
\end{equation}
where $\mu_{(a, b, c)}$ has a centred bi-Boolean Gaussian distribution such that the only non-vanishing B-$(\ell, r)$-cumulants are $B_{2, 0}(\mu_{(a, b, c)}) = a$, $B_{0, 2}(\mu_{(a, b, c)}) = b$, and $B_{1, 1}(\mu_{(a, b, c)}) = c$.

In analogy with other types of compound Poisson distributions, given $\lambda \geq 0$ and a compactly supported Borel probability measure $\sigma \neq \delta_{(0, 0)}$ on $\bR^2$, the \textit{compound bi-Boolean Poisson distribution} $\pi_{\lambda, \sigma}^{\uplus\uplus}$ on $\bR^2$ with rate $\lambda$ and jump distribution $\sigma$ is characterized by the requirement that
\[B_{m, n}(\pi_{\lambda, \sigma}^{\uplus\uplus}) = \lambda\cdot\bM_{m, n}(\sigma)
\]
for all $m, n \geq 0$ with $m + n \geq 1$. The following compound bi-Boolean Poisson limit theorem justifies the name of $\pi_{\lambda, \sigma}^{\uplus\uplus}$. The proof easily follows from the additivity of the B-$(\ell, r)$-cumulants and the bi-Boolean moment-cumulant formula (see also \cite{GS2016}*{Theorem 6.11} for the c-bi-free version).

\begin{thm}
Let $\lambda \geq 0$ and let $\sigma \neq \delta_{(0, 0)}$ be a compactly supported Borel probability measure on $\bR^2$. For $N \in \bN$, let $\mu_N = \left(1 - \frac{\lambda}{N}\right)\delta_{(0, 0)} + \frac{\lambda}{N}\sigma$.  Then $\lim_{N \to \infty}\underbrace{\mu_N \uplus\uplus \cdots \uplus\uplus \mu_N}_{N\,\mathrm{times}} = \pi_{\lambda, \sigma}^{\uplus\uplus}$.
\end{thm}

Before moving on to the analytic aspects, we mention that if $\mu$ is a compactly supported Borel probability measure on $\bR^2$, then the bi-Boolean partial $\eta$-transform $\eta_\mu$ of $\mu$ is naturally defined as the B-$(\ell, r)$-cumulant generating series
\[\eta_\mu(z, w) = \sum_{\substack{m, n \geq 0\\m + n \geq 1}}B_{m, n}(\mu)z^mw^n,\]
which converges absolutely for $|z|$, $|w|$ sufficiently small. However, in view of Theorem \ref{PartialEta} and Remark \ref{PartialSelfEnergy}, we will actually take $E_\mu$ as the linearizing transform with respect to $\uplus\uplus$, which has the advantage of being an analytic function on $(\bC \setminus \bR)^2$.

\subsection{Analytic aspects}

Recall the \textit{Cauchy transform} of a finite Borel measure $\mu$ on $\bR$ is defined by
\[G_\mu(z) = \int_{\bR}\frac{1}{z - s}\,d\mu(s), \quad z \in (\bC \setminus \bR),\]
which is an analytic function and determines the underlying measure uniquely. Note that $G_\mu(\bar{z}) = \overline{G_\mu(z)}$, thus the behaviour of $G_\mu$ on $\bC^{-}$ is determined by $G_\mu$ on $\bC^{+}$. Denote the reciprocal of $G_\mu$ by $F_\mu$ so that $F_\mu(z) = 1/G_\mu(z)$ for $z \in (\bC \setminus \bR)$. For $\alpha, \beta > 0$, the \textit{Stolz angle} and \textit{truncated Stolz angle} are defined by
\[\Gamma_\alpha = \{z = x + iy \in \bC^{+}\,|\,|x| < \alpha y\} \qand \Gamma_{\alpha, \beta} = \{z = x + iy \in \Gamma_\alpha\,|\,y > \beta\},\]
respectively. As shown in \cite{BV1993}, for every $\alpha > 0$, there exists a $\beta = \beta(\mu, \alpha) > 0$ such that the compositional inverse $F_\mu^{-1}$ of $F_\mu$ is defined on $\Gamma_{\alpha, \beta}$. The \textit{free Voiculescu transform} of $\mu$ is defined by
\[\phi_\mu(z) = F_\mu^{-1}(z) - z, \quad z \in \Gamma_{\alpha, \beta},\]
which linearizes the additive free convolution $\boxplus$ in the sense that
\[\phi_{\mu_1 \boxplus \mu_2}(z) = \phi_{\mu_1}(z) + \phi_{\mu_2}(z)\]
on the common domain of the three functions involved. On the other hand, as mentioned in Section \ref{sec:transforms} above, the \textit{self-energy} of $\mu$ is defined by
\[E_\mu(z) = z - F_\mu(z), \quad z \in \bC \setminus \bR,\]
which linearizes the additive Boolean convolution $\uplus$ in the sense that
\[E_{\mu_1 \uplus \mu_2}(z) = E_{\mu_1}(z) + E_{\mu_2}(z).\]

In classical probability theory, an important class of measures occurs in connection with the study of limit theorems; namely the class of infinitely divisible measures. Analogously, a Borel probability measure $\mu$ on $\bR$ is said to be \textit{$\boxplus$-infinitely divisible} (respectively \textit{$\uplus$-infinitely divisible}) if for every $n \in \bN$, there exists a Borel probability measure $\mu_n$ on $\bR$ such that
\[\mu = \underbrace{\mu_n \boxplus \cdots \boxplus \mu_n}_{n\,\mathrm{times}} \quad \left(\mathrm{respectively}\,\mu = \underbrace{\mu_n \uplus \cdots \uplus \mu_n}_{n\,\mathrm{times}}\right).\]
It is well-known that the L\'{e}vy-Hin\v{c}in formula provides a complete charactization of infinitely divisible measures. For the free situation, the most general case was obtained in \cite{BV1993} that a Borel probability measure $\mu$ on $\bR$ is $\boxplus$-infinitely divisible if and only if there exist a real number $\gamma \in \bR$ and a finite positive Borel measure $\sigma$ on $\bR$ such that
\[\phi_\mu(z) = \gamma + \int_{\bR}\frac{1 + sz}{z - s}\,d\sigma(s), \quad z \in \bC \setminus \bR.\]
The pair $(\gamma, \sigma)$ is unique, and every such pair $(\gamma, \sigma)$ determines a $\boxplus$-infinitely divisible measure, which is usually denoted as $\mu_{\boxplus}^{(\gamma, \sigma)}$ to indicate this correspondence. For the Boolean situation, things are much simpler as every self-energy can be written as
\begin{equation}\label{BooleanLH}
E_\mu(z) = \gamma + \int_{\bR}\frac{1 + sz}{z - s}\,d\sigma(s), \quad z \in \bC \setminus \bR
\end{equation}
for $\gamma \in \bR$ a real number and $\sigma$ a finite positive Borel measure on $\bR$, and conversely every such pair $(\gamma, \sigma)$ uniquely determines a Borel probability measure $\mu$ on $\bR$ such that equation \eqref{BooleanLH} holds. We denote $\mu$ by $\mu_{\uplus}^{(\gamma, \sigma)}$ in this case, and it follows that all Borel probability measures on $\bR$ are $\uplus$-infinitely divisible. This bijection between the set of $\uplus$-infinitely divisible/all measures and the set of $\boxplus$-infinitely divisible measures is known as the Boolean Bercovici-Pata bijection, where a more explicit relation between $\mu_{\uplus}^{(\gamma, \sigma)}$ and $\mu_{\boxplus}^{(\gamma, \sigma)}$ appears in \cite{BP1999}*{Theorem 6.3} as follows. Note that we have removed the statement about the additive classical convolution (which is actually one of the main results of \cite{BP1999}) as it is irrelevant here.

\begin{thm}
Fix a sequence $\{\mu_n\}_{n = 1}^\infty$ of Borel probability measures on $\bR$, a sequence $\{k_n\}_{n = 1}^\infty$ of positive integers with $\lim_{n \to \infty}k_n = \infty$, and a pair $(\gamma, \sigma)$ where $\gamma \in \bR$ is a real number and $\sigma$ is a finite positive Borel measure on $\bR$. The following assertions are equivalent:
\begin{enumerate}[$\qquad(1)$]
\item The sequence $\underbrace{\mu_n \boxplus \cdots \boxplus \mu_n}_{k_n\,\mathrm{times}}$ converges weakly to $\mu_{\boxplus}^{(\gamma, \sigma)}$.

\item The sequence $\underbrace{\mu_n \uplus \cdots \uplus \mu_n}_{k_n\,\mathrm{times}}$ converges weakly to $\mu_{\uplus}^{(\gamma, \sigma)}$.

\item The limit
\[\lim_{n \to \infty}k_n\int_{\bR}\frac{s}{1 + s^2}\,d\mu_n(s) = \gamma\]
holds, and the finite positive Borel measures
\[d\sigma_n(s) = k_n\frac{s^2}{1 + s^2}\,d\mu_n(s)\]
converge weakly to $\sigma$.
\end{enumerate}
\end{thm}

Note that both $\boxplus$ and $\uplus$ can be studied in terms of the additive c-free convolution $\boxplus_{\mathrm{c}}$ which is defined on pairs of Borel probability measures on $\bR$. Given such a pair $(\mu, \nu)$, the \textit{c-free Voiculescu transform} of $(\mu, \nu)$ is defined by
\[\Phi_{(\mu, \nu)}(z) = F_{\nu}^{-1}(z) - F_\mu(F_{\nu}^{-1}(z))\]
on a domain where $F_{\nu}^{-1}$ is defined. Given two pairs $(\mu_1, \nu_1)$ and $(\mu_2, \nu_2)$, their additive c-free convolution is another pair $(\mu, \nu)$ where $\nu = \nu_1 \boxplus \nu_2$ and $\mu$ is the unique measure such that
\[\Phi_{(\mu, \nu)}(z) = \Phi_{(\mu_1, \nu_1)}(z) + \Phi_{(\mu_2, \nu_2)}(z)\]
on the common domain of the three involved functions. It is immediate that
\[\Phi_{(\mu, \mu)} = \phi_\mu \qand \Phi_{(\mu, \delta_0)} = E_\mu.\]
The notion of $\boxplus_{\mathrm{c}}$-infinite divisibility is defined analogously, and a c-free L\'{e}vy-Hin\v{c}in formula was obtained in \cite{K2007}*{Theorem 8.1} in the compactly supported case and extended in full generality in \cite{W2011}*{Theorem 4.1} that given a pair $(\mu, \nu)$ of Borel probability measures on $\bR$ with $\nu$ being $\boxplus$-infinitely divisible, the pair $(\mu, \nu)$ is $\boxplus_{\mathrm{c}}$-infinitely divisible if and if there exists a pair $(\gamma, \sigma)$ similar as above such that
\[\Phi_{(\mu, \nu)}(z) = \gamma + \int_{\bR}\frac{1 + sz}{z - s}\,d\sigma(s), \quad z \in \bC \setminus \bR.\]

Next we discuss measures on $\bR^2$. Recall the \textit{Cauchy transform} of a finite positive Borel measure $\mu$ on $\bR^2$ is defined by
\[G_\mu(z, w) = \int_{\bR^2}\frac{1}{(z - s)(w - t)}\,d\mu(s, t), \quad (z, w) \in (\bC \setminus \bR)^2,\]
and is an analytic function which uniquely determines the underlying measure.  Let $\mu^{(1)}$ and $\mu^{(2)}$ denote the marginal distributions of $\mu$, namely
\[\mu^{(1)}(B) = \mu(B \times \bR) \qand \mu^{(2)}(B) = \mu(\bR \times B)\]
for all Borel subsets $B$ of $\bR$. For $\alpha, \beta > 0$, let $\overline{\Gamma_{\alpha, \beta}} = \{\bar{z}\,|\,z \in \Gamma_{\alpha, \beta}\}$ and set
\[\Omega_{\alpha, \beta} = \left\{(z, w) \in (\bC \setminus \bR)^2\,|\,z, w \in \Gamma_{\alpha, \beta} \cup \overline{\Gamma_{\alpha, \beta}}\right\}.\]
The \textit{bi-free partial Voiculescu transform} of $\mu$ is defined by
\[\phi_\mu(z, w) = \frac{1}{z}\phi_{\mu^{(1)}}(z) + \frac{1}{w}\phi_{\mu^{(2)}}(w) + \widetilde{\phi}_\mu(z, w), \quad (z, w) \in \Omega_{\alpha, \beta},\]
for some $\alpha, \beta > 0$, where $\widetilde{\phi}_\mu$ is the \textit{reduced bi-free partial Voiculescu transform} of $\mu$ defined by
\[\widetilde{\phi}_\mu(z, w) = 1 - \frac{1}{zwG_\mu\left(F_{\mu^{(1)}}^{-1}(z), F_{\mu^{(2)}}^{-1}(w)\right)}.\]

The additive bi-free convolution $\boxplus\boxplus$ is characterized by
\[\phi_{\mu_1 \boxplus\boxplus \mu_2}(z, w) = \phi_{\mu_1}(z, w) + \phi_{\mu_2}(z, w)\]
on the common domain of the three functions involved whenever $\mu_1 \boxplus\boxplus \mu_2$ is defined. This is due to the fact that the operation $\boxplus\boxplus$ is only defined for compactly supported and/or infinitely divisible measures as of now. These restrictions are also in place for the operation $\boxplus\boxplus_{\mathrm{c}}$ in \cite{GS2016}*{Section 6} and the operation $\uplus\uplus$ to be considered below. Consequently a Borel probability measure $\mu$ on $\bR^2$ is said to be \textit{$\boxplus\boxplus$-infinitely divisible} if for every $n \in \bN$ there exists a Borel probability measure $\mu_n$ on $\bR^2$ such that $\phi_\mu = n\phi_{\mu_n}$ on a common domain of the form $\Omega_{\alpha, \beta}$. The notion of $\boxplus\boxplus_{\mathrm{c}}$-infinite divisibility for pairs of measures was similarly defined in \cite{GS2016}*{Definition 6.24} in terms of the corresponding linearizing transforms, and hence $\mu$ is said to be \textit{$\uplus\uplus$-infinitely divisible} if for every $n \in \bN$, $E_\mu = E_{\mu_n}$ (see Definition \ref{defn:partialselfenergy} below) for some $\mu_n$.

As shown in \cites{HW2016, GS2016}, the additive bi-free and c-bi-free limit theorems do not depend on whether or not $\boxplus\boxplus$ and $\boxplus\boxplus_{\mathrm{c}}$ are defined for arbitrary measures. In particular, a bi-free L\'{e}vy-Hin\v{c}in formula was obtained in \cite{HW2016} demonstrating that a Borel probability measure $\mu$ on $\bR^2$ is $\boxplus\boxplus$-infinitely divisible if and only if there exists a unique quintuple $(\gamma_1, \gamma_2, \sigma_1, \sigma_2, \sigma)$, where $\gamma_1, \gamma_2 \in \bR$ are real numbers, $\sigma_1$ and $\sigma_2$ are finite positive Borel measures on $\bR^2$, and $\sigma$ is a finite signed Borel measure on $\bR^2$ such that
\begin{itemize}
\item $\frac{t}{\sqrt{1 + t^2}}\,d\sigma_1(s, t) = \frac{s}{\sqrt{1 + s^2}}\,d\sigma(s, t)$,

\item $\frac{s}{\sqrt{1 + s^2}}\,d\sigma_2(s, t) = \frac{t}{\sqrt{1 + t^2}}\,d\sigma(s, t)$,

\item $|\sigma(\{(0, 0)\})|^2 \leq \sigma_1(\{(0, 0)\})\sigma_2(\{(0, 0)\})$,
\end{itemize}
and the bi-free partial Voiculescu transform $\phi_\mu$ of $\mu$ can be continued analytically to $(\bC \setminus \bR)^2$ via
\begin{equation}\label{BiFreeLH}
\begin{split}
\phi_\mu(z, w) &= \frac{1}{z}\left(\gamma_1 + \int_{\bR^2}\frac{1 + sz}{z - s}\,d\sigma_1(s, t)\right) + \frac{1}{w}\left(\gamma_2 + \int_{\bR^2}\frac{1 + tw}{w - t}\,d\sigma_2(s, t)\right)\\
&\quad + \int_{\bR^2}\frac{\sqrt{1 + s^2}\sqrt{1 + t^2}}{(z - s)(w - t)}\,d\sigma(s, t).
\end{split}
\end{equation}
Moreover, the marginal distributions $\mu^{(1)}$ and $\mu^{(2)}$ of $\mu$ are $\boxplus$-infinitely divisible determined by $\mu^{(1)} = \mu_{\boxplus}^{(\gamma_1, \sigma_1^{(1)})}$ and $\mu^{(2)} = \mu_{\boxplus}^{(\gamma_2, \sigma_2^{(2)})}$. Conversely, every such quintuple $(\gamma_1, \gamma_2, \sigma_1, \sigma_2, \sigma)$ uniquely determines a Borel probability measure $\mu$ on $\bR^2$ such that equation \eqref{BiFreeLH} holds, and we denote $\mu$ by $\mu_{\boxplus\boxplus}^{(\gamma_1, \gamma_2, \sigma_1, \sigma_2, \sigma)}$ in analogy with the free and Boolean situations. Similarly, a c-bi-free L\'{e}vy-Hin\v{c}in formula was obtained in \cite{GS2016}*{Section 6} characterizing $\boxplus\boxplus_{\mathrm{c}}$-infinitely divisible pairs of measures where each such pair consists of two measures $\mu$ and $\nu$ such that $\nu$ is $\boxplus\boxplus$-infinitely divisible and $\mu$ is determined by a quintuple $(\gamma_1, \gamma_2, \sigma_1, \sigma_2, \sigma)$ similar as above.

We now move on to the bi-Boolean situation. As mentioned above, the following function linearizes $\uplus\uplus$ and will be used as the bi-Boolean analogue of $\phi_\mu$.

\begin{defn}\label{defn:partialselfenergy}
Let $\mu$ be a Borel probability measure on $\bR^2$. The \textit{partial self-energy} of $\mu$ is defined by
\[E_\mu(z, w) = \frac{1}{z}E_{\mu^{(1)}}(z) + \frac{1}{w}E_{\mu^{(2)}}(w) + \widetilde{E}_\mu(z, w), \quad (z, w) \in (\bC \setminus \bR)^2\]
where
\[\widetilde{E}_\mu(z, w) = \frac{G_\mu(z, w)}{G_{\mu^{(1)}}(z)G_{\mu^{(2)}}(w)} - 1.\]
The function $\widetilde{E}_\mu$ will be referred to as the \textit{reduced partial self-energy} of $\mu$.
\end{defn}

Recall that in the one-dimensional situation, the self-energy of a measure is equal to the c-free Voiculescu transform of the measure with $\delta_0$. Similarly, if $\Phi_{(\mu, \nu)}$ denotes the c-bi-free partial Voiculescu transform (as defined in \cite{GS2016}*{Definition 6.15}) of a pair of measures $(\mu, \nu)$ on $\bR^2$, then it is easy to check that $E_\mu = \Phi_{(\mu, \delta_{(0, 0)})}$. Consequently, the results in \cite{GS2016}*{Subsections 6.3 to 6.6} immediately imply the following bi-Boolean analogues by taking the second component to be $\delta_{(0, 0)}$, which we record as follows. Note that by $z \to \infty$ \textit{non-tangentially} we mean $|z| \to \infty$ but $z$ stays within a Stolz angle $\Gamma_\alpha$ for some $\alpha > 0$.

To begin, we have the following basic properties of the partial self-energy function which follow from \cite{GS2016}*{Lemma 6.16, Corollary 6.17, and Proposition 6.18}.

\begin{lem}
If $E_\mu: (\bC \setminus \bR)^2 \to \bC$ is the partial self-energy of some Borel probability measure $\mu$ on $\mathbb{R}^2$, then
\[\lim_{|z| \to \infty}E_\mu(z, w) = \frac{1}{w}E_{\mu^{(2)}}(w),\quad \lim_{|w| \to \infty}E_\mu(z, w) = \frac{1}{z}E_{\mu^{(1)}}(z), \qand \lim_{|z|, |w| \to \infty}E_\mu(z, w) = 0\]
non-tangentially.
\end{lem}

\begin{cor}
If $\mu_1$ and $\mu_2$ are two Borel probability measures on $\bR^2$ such that $E_{\mu_1} = E_{\mu_2}$, then $\mu_1 = \mu_2$.
\end{cor}

\begin{prop}
Let $\{\mu_n\}_{n = 1}^\infty$ be a sequence of Borel probability measures on $\mathbb{R}^2$. The following assertions are equivalent:
\begin{enumerate}[$\qquad(1)$]
\item The sequence $\{\mu_n\}_{n = 1}^\infty$ converges weakly to a Borel probability measure $\mu$ on $\bR^2$.

\item The pointwise limits $\lim_{n \to \infty}E_{\mu_n}(z, w) = E(z, w)$ exist for all $(z, w) \in (\bC \setminus \bR)^2$, and the limit $E_{\mu_n}(z, w) \to 0$ holds uniformly in $n$ as $|z|, |w| \to \infty$ non-tangentially.
\end{enumerate}
Moreover, if these assertions hold, then $E = E_\mu$ on $(\bC \setminus \bR)^2$.
\end{prop}

Furthermore, \cite{GS2016}*{Theorem 6.19, Proposition 6.20, Theorems 6.23 and 6.25} imply the following additive bi-Boolean limit theorems.

\begin{thm}\label{LimitThm}
Let $\{\mu_n\}_{n = 1}^\infty$ be a sequence of Borel probability measures on $\bR^2$ and $\{k_n\}_{n = 1}^\infty$ be a sequence of positive integers with $\lim_{n \to \infty}k_n = \infty$. Assume the sequences $\{[\mu_n^{(1)}]^{\uplus k_n}\}_{n = 1}^\infty$ and $\{[\mu_n^{(2)}]^{\uplus k_n}\}_{n = 1}^\infty$ converge weakly to $\mu_{\uplus}^{(\gamma_1, \sigma_1)}$ and $\mu_{\uplus}^{(\gamma_2, \sigma_2)}$ respectively, where $\gamma_1, \gamma_2 \in \bR$ are real numbers, and $\sigma_1, \sigma_2$ are finite positive Borel measures on $\mathbb{R}$. The following assertions are equivalent.
\begin{enumerate}[$\qquad(1)$]
\item The pointwise limits
\[\lim_{n \to \infty}k_nE_{\mu_n}(z, w) = E(z, w)\]
exist for all $(z, w) \in (\bC \setminus \bR)^2$.

\item The pointwise limits
\[\lim_{n \to \infty}k_n\int_{\bR^2}\frac{st}{(z - s)(w - t)}\,d\mu_n(s, t) = \widetilde{E}(z, w)\]
exist for all $(z, w) \in (\mathbb{C} \setminus \mathbb{R})^2$.

\item The finite signed Borel measures
\[d\widetilde{\sigma}_n(s, t) = k_n\frac{st}{\sqrt{1 + s^2}\sqrt{1 + t^2}}\,d\mu_n(s, t)\]
converge weakly to a finite signed Borel measure $\sigma$ on $\mathbb{R}^2$.
\end{enumerate}
Moreover, if these assertions hold, then the function $\widetilde{E}$ from assertion $(2)$ has a unique integral representation
\[\widetilde{E}(z, w) = \int_{\mathbb{R}^2}\frac{\sqrt{1 + s^2}\sqrt{1 + t^2}}{(z - s)(w - t)}\,d\sigma(s, t),\]
and the function $E$ from assertion $(1)$ can be written as
\[E(z, w) = \frac{1}{z}E_{\mu_{\uplus}^{(\gamma_1, \sigma_1)}}(z) + \frac{1}{w}E_{\mu_{\uplus}^{(\gamma_2, \sigma_2)}}(w) + \widetilde{E}(z, w)\]
for all $(z, w) \in (\mathbb{C} \setminus \mathbb{R})^2$.
\end{thm}

\begin{prop}
Let $\{\mu_n\}_{n = 1}^\infty$ and $\{k_n\}_{n = 1}^\infty$ be sequences of measures and positive integers satisfying the assumptions of Theorem \ref{LimitThm}. Assume furthermore that each $\mu_n$ is compactly supported.  Then the sequence $\{\mu_n^{\uplus\uplus k_n}\}_{n = 1}^\infty$ converges weakly to a Borel probability measure on $\mathbb{R}^2$ if and only if the sequences $\{[\mu_n^{(1)}]^{\uplus k_n}\}_{n = 1}^\infty$, $\{[\mu_n^{(2)}]^{\uplus k_n}\}_{n = 1}^\infty$, and $\{\widetilde{\sigma}_n\}_{n = 1}^\infty$ are weakly convergent, where $\{\widetilde{\sigma}_n\}_{n = 1}^\infty$ is defined as in assertion $(3)$ of Theorem \ref{LimitThm}. Moreover, if $\{\mu_n^{\uplus\uplus k_n}\}_{n = 1}^\infty$, $\{[\mu_n^{(1)}]^{\uplus k_n}\}_{n = 1}^\infty$, $\{[\mu_n^{(2)}]^{\uplus k_n}\}_{n = 1}^\infty$, and $\{\widetilde{\sigma}_n\}_{n = 1}^\infty$ converge weakly to $\mu$, $\mu_{\uplus}^{(\gamma_1, \sigma_1)}$, $\mu_{\uplus}^{(\gamma_2, \sigma_2)}$, and $\sigma$ respectively, then $\mu^{(1)} = \mu_{\uplus}^{(\gamma_1, \sigma_1)}$, $\mu^{(2)} = \mu_{\uplus}^{(\gamma_2, \sigma_2)}$, and
\[G_\mu(z, w) = G_{\mu^{(1)}}(z)G_{\mu^{(2)}}(w)\left(G_{\sigma'}(z, w) + 1\right), \quad (z, w) \in (\bC \setminus \bR)^2,\]
where $d\sigma'(s, t) = \sqrt{1 + s^2}\sqrt{1 + t^2}\,d\sigma(s, t)$.
\end{prop}

\begin{thm}
Let $\{\mu_n\}_{n = 1}^\infty$ and $\{k_n\}_{n = 1}^\infty$ be sequences of measures and positive integers satisfying the assumptions of Theorem \ref{LimitThm}. If the pointwise limits $\lim_{n \to \infty}k_nE_{\mu_n}(z, w) = E(z, w)$ exist for all $(z, w) \in (\mathbb{C} \setminus \mathbb{R})^2$, then there exists a unique Borel probability measure $\mu$ on $\mathbb{R}^2$ such that $E = E_\mu$ on $(\mathbb{C} \setminus \mathbb{R})^2$.
\end{thm}

\begin{thm}
Let $\mu$ be a Borel probability measure on $\mathbb{R}^2$ with partial self-energy $E_\mu$. The measure $\mu$ is $\uplus\uplus$-infinitely divisible if and only if there exist a sequence $\{\mu_n\}_{n = 1}^\infty$ of Borel probability measures on $\mathbb{R}^2$ and a sequence $\{k_n\}_{n = 1}^\infty$ of positive integers with $\lim_{n \to \infty}k_n = \infty$ such that the sequences $\{[\mu_n^{(1)}]^{\uplus k_n}\}_{n = 1}^\infty$ and $\{[\mu_n^{(2)}]^{\uplus k_n}\}_{n = 1}^\infty$ converge weakly to $\mu^{(1)}$ and $\mu^{(2)}$ respectively, and $\lim_{n \to \infty}k_nE_{\mu_n} = E_\mu$ on $(\mathbb{C} \setminus \mathbb{R})^2$.
\end{thm}

Consequently, every $\uplus\uplus$-infinitely divisible measure $\mu$ has a partial self-energy of the form
\[E_\mu(z, w) = \frac{1}{z}E_{\mu_{\uplus}^{(\gamma_1, \sigma_1)}}(z) + \frac{1}{w}E_{\mu_{\uplus}^{(\gamma_2, \sigma_2)}}(w) + \int_{\mathbb{R}^2}\frac{\sqrt{1 + s^2}\sqrt{1 + t^2}}{(z - s)(w - t)}\,d\sigma(s, t), \quad (z, w) \in (\bC \setminus \bR)^2,\]
where $\mu^{(1)} = \mu_{\uplus}^{(\gamma_1, \sigma_1)}$ and $\mu^{(2)} = \mu_{\uplus}^{(\gamma_2, \sigma_2)}$. Note that if $\widehat{\sigma}_1$ is a finite positive Borel measure on $\bR^2$ such that $\widehat{\sigma}_1^{(1)} = \sigma_1$, then
\[\frac{1}{z}E_{\mu_{\uplus}^{(\gamma_1, \sigma_1)}}(z) = \frac{1}{z}\left(\gamma_1 + \int_{\bR^2}\frac{1 + sz}{z - s}\,d\widehat{\sigma}_1(s, t)\right).\]
Clearly, such an extension from $\sigma_1$ to $\widehat{\sigma}_1$ is not unique. However, if we impose the additional requirement that
\[\frac{t}{\sqrt{1 + t^2}}\,d\widehat{\sigma}_1(s, t) = \frac{s}{\sqrt{1 + s^2}}\,d\sigma(s, t),\]
then it follows from \cite{HW2016}*{Lemma 3.10} that $\widehat{\sigma}_1$ is unique. Similarly, there is a unique finite positive Borel measure $\widehat{\sigma}_2$ on $\bR^2$ such that $\widehat{\sigma}_2^{(2)} = \sigma_2$ and
\[\frac{s}{\sqrt{1 + s^2}}\,d\widehat{\sigma}_2(s, t) = \frac{t}{\sqrt{1 + t^2}}\,d\sigma(s, t).\]
Replacing $\sigma_1$ and $\sigma_2$ by their unique extensions to $\bR^2$ with the above conditions, every $\uplus\uplus$-infinitely divisible measure $\mu$ has a unique quintuple $(\gamma_1, \gamma_2, \sigma_1, \sigma_2, \sigma)$ associated to it such that
\begin{equation}\label{BiBooleanLH}
\begin{split}
E_\mu(z, w) &= \frac{1}{z}\left(\gamma_1 + \int_{\bR^2}\frac{1 + sz}{z - s}\,d\sigma_1(s, t)\right) + \frac{1}{w}\left(\gamma_2 + \int_{\bR^2}\frac{1 + tw}{w - t}\,d\sigma_2(s, t)\right)\\
&\quad + \int_{\bR^2}\frac{\sqrt{1 + s^2}\sqrt{1 + t^2}}{(z - s)(w - t)}\,d\sigma(s, t), \quad (z, w) \in (\bC \setminus \bR)^2.
\end{split}
\end{equation}
Conversely, it follows from \cite{GS2016}*{Proposition 6.27} that every such quintuple $(\gamma_1, \gamma_2, \sigma_1, \sigma_2, \sigma)$ uniquely determines a Borel probability measure $\mu$ on $\bR^2$ such that equation \eqref{BiBooleanLH} holds. Therefore, it makes sense to refer to equation \eqref{BiBooleanLH} as the \textit{bi-Boolean L\'{e}vy-Hin\v{c}in representation} of the $\uplus\uplus$-infinitely divisible measure $\mu$, and we denote $\mu$ by $\mu_{\uplus\uplus}^{(\gamma_1, \gamma_2, \sigma_1, \sigma_2, \sigma)}$ to indicate this correspondence.

\begin{exam}
All product measures on $\bR^2$ are $\uplus\uplus$-infinitely divisible. Indeed, if $\mu$ is a Borel probability measure on $\bR^2$ such that $\mu = \mu^{(1)} \otimes \mu^{(2)}$, then as $\mu^{(1)}$ and $\mu^{(2)}$ are $\uplus$-infinitely divisible and as $\widetilde{E}_\mu \equiv 0$, there exist $(\gamma_1, \sigma_1)$ and $(\gamma_2, \sigma_2)$ such that
\[
E_\mu(z, w) = \frac{1}{z}\left(\gamma_1 + \int_{\bR}\frac{1 + sz}{z - s}\,d\sigma_1(s)\right) + \frac{1}{w}\left(\gamma_2 + \int_{\bR}\frac{1 + tw}{w - t}\,d\sigma_2(t)\right), \quad (z, w) \in (\bC \setminus \bR)^2.
\]
In this case, it is easy to see that $E_\mu/n$ is the partial self-energy of the measure $\mu_{\uplus}^{(\gamma_1/n, \sigma_1/n)} \otimes \mu_{\uplus}^{(\gamma_2/n, \sigma_2/n)}$ for $n \geq 1$.

In general, more examples of $\uplus\uplus$-infinitely divisible measures can be obtained from the $\boxplus\boxplus$-infinitely divisible ones (e.g., the bi-free central and Poisson limits from \cite{HW2016}) via the two-dimensional Bercovici-Pata bijection (Theorem \ref{TDBP}) below. For instance, if $\mu$ denotes the bi-Boolean Gaussian distribution $\mu_{(1, 1, c)}$ from equation \eqref{eqn:CentLimit} above, then
\[E_\mu(z, w) = \frac{1}{z^2} + \frac{1}{w^2} + \frac{c}{zw}, \quad (z, w) \in (\bC \setminus \bR)^2.\]
It is known (see \cite{SW1997}*{Section 3}) that $\mu^{(1)} = \mu^{(2)}$ is the Bernoulli distribution $\frac{1}{2}(\delta_{-1} + \delta_1)$ with Cauchy transform $G_{\mu^{(j)}}(z) = \frac{z}{z^2 - 1}$, from which we obtain
\[G_\mu(z, w) = \frac{c + zw}{(z^2 - 1)(w^2 - 1)}, \quad (z, w) \in (\bC \setminus \bR)^2,\]
and the Stieltjes inversion formula (see, e.g., \cite{HW2016}*{Section 2}) gives
\begin{align*}
d\mu(s, t) &= -\frac{1}{4\pi^2}\lim_{\ep \to 0^+}[G_\mu(s + i\ep, t + i\ep) - G_\mu(s + i\ep, t - i\ep) - G_\mu(s - i\ep, t + i\ep) + G_\mu(s - i\ep, t - i\ep)]dsdt\\
&= \frac{1}{\pi^2}\lim_{\ep \to 0^+}\left[\frac{\ep^2[4cst + (s^2 + \ep^2 + 1)(t^2 + \ep^2 + 1)]}{[s^4 + 2s^2(\ep^2 - 1) + \ep^2(\ep^2 + 2) + 1][t^4 + 2t^2(\ep^2 - 1) + \ep^2(\ep^2 + 2) + 1]}\right]dsdt.
\end{align*}
The above formula shows that the measure $\mu$ vanishes outside of the square $\{(s, t) \in \bR^2\,:\,|s| = 1, |t| = 1\}$, and the uncorrelated case (i.e., $c = 0$) corresponds to $\mu = \frac{1}{2}(\delta_{-1} + \delta_1) \otimes \frac{1}{2}(\delta_{-1} + \delta_{1})$.
\end{exam}

\subsection{Some Bercovici-Pata type bijections}

The discussion in the previous subsection together with results in \cites{GS2016, HW2016} imply some two-dimensional Bercovici-Pata type bijections, which are summarized as follows. Note that we implicitly assume the additive convolutions are well-defined, which is the case at least for compactly supported and/or infinitely divisible measures. The limiting distributions, on the other hand, are allowed to have unbounded supports.

\begin{thm}\label{TDBP}
Fix a sequence $\{\mu_n\}_{n = 1}^\infty$ of Borel probability measures on $\mathbb{R}^2$, a sequence $\{k_n\}_{n = 1}^\infty$ of positive integers with $\lim_{n \to \infty}k_n = \infty$, a quintuple $(\gamma_1, \gamma_2, \sigma_1, \sigma_2, \sigma)$, where $\gamma_1, \gamma_2 \in \bR$ are real numbers, $\sigma_1$ and $\sigma_2$ are finite positive Borel measures on $\mathbb{R}^2$, and $\sigma$ is a finite signed Borel measure on $\mathbb{R}^2$ such that
\begin{itemize}
\item $\frac{t}{\sqrt{1 + t^2}}\,d\sigma_1(s, t) = \frac{s}{\sqrt{1 + s^2}}\,d\sigma(s, t)$,

\item $\frac{s}{\sqrt{1 + s^2}}\,d\sigma_2(s, t) = \frac{t}{\sqrt{1 + t^2}}\,d\sigma(s, t)$,

\item $|\sigma(\{(0, 0)\})|^2 \leq \sigma_1(\{(0, 0)\})\sigma_2(\{(0, 0)\})$.
\end{itemize}
The following assertions are equivalent.
\begin{enumerate}[$\qquad(1)$]
\item The sequence $\underbrace{\mu_n \boxplus\boxplus \cdots \boxplus\boxplus \mu_n}_{k_n\,\mathrm{times}}$ converges weakly to $\mu_{\boxplus\boxplus}^{(\gamma_1, \gamma_2, \sigma_1, \sigma_2, \sigma)}$.

\item The sequence $\underbrace{\mu_n \uplus\uplus \cdots \uplus\uplus \mu_n}_{k_n\,\mathrm{times}}$ converges weakly to $\mu_{\uplus\uplus}^{(\gamma_1, \gamma_2, \sigma_1, \sigma_2, \sigma)}$.

\item The limits
\[\lim_{n \to \infty}k_n\int_{\mathbb{R}^2}\frac{s}{1 + s^2}\,d\mu_n(s, t) = \gamma_1 \qand \lim_{n \to \infty}k_n\int_{\mathbb{R}^2}\frac{t}{1 + t^2}\,d\mu_n(s, t) = \gamma_2\]
hold, the finite positive Borel measures
\[d\sigma_{n}^{(1)}(s) = k_n\frac{s^2}{1 + s^2}\,d\mu_n^{(1)}(s) \qand d\sigma_{n}^{(2)}(t) = k_n\frac{t^2}{1 + t^2}\,d\mu_n^{(2)}(t)\]
converge weakly to $\sigma_1^{(1)}$ and $\sigma_2^{(2)}$ respectively, and the finite signed Borel measures
\[d\widetilde{\sigma}_n(s, t) = k_n\frac{st}{\sqrt{1 + s^2}\sqrt{1 + t^2}}\,d\mu_n(s, t)\]
converge weakly to $\sigma$.
\end{enumerate}
Moreover, if these assertions hold, then
\[\phi_{\mu_{\boxplus\boxplus}^{(\gamma_1, \gamma_2, \sigma_1, \sigma_2, \sigma)}}(z, w) = E_{\mu_{\uplus\uplus}^{(\gamma_1, \gamma_2, \sigma_1, \sigma_2, \sigma)}}(z, w)\]
for all $(z, w) \in (\bC \setminus \bR)^2$.
\end{thm}

\begin{proof}
The equivalence of assertions $(1)$ and $(3)$ was achieved in \cite{HW2016}, and the equivalence of assertions $(2)$ and $(3)$ follow from the one-dimensional Bercovici-Pata bijection and the results in the previous subsection.
\end{proof}

If in addition $\{\nu_n\}_{n = 1}^\infty$ is another sequence of Borel probability measures on $\bR^2$ such that the sequence
\[\underbrace{\nu_n \boxplus\boxplus \cdots \boxplus\boxplus \nu_n}_{k_n\,\mathrm{times}}\]
converges weakly to some $\boxplus\boxplus$-infinitely divisible measure $\nu_{\boxplus\boxplus}^{(\gamma_1', \gamma_2', \sigma_1', \sigma_2', \sigma')}$, then the results in \cite{GS2016}*{Section 6} imply that assertion $(2)$ of Theorem \ref{TDBP} is equivalent to the weak convergence of the sequence
\[\underbrace{(\mu_n, \nu_n) \boxplus\boxplus_{\mathrm{c}} \cdots \boxplus\boxplus_{\mathrm{c}} (\mu_n, \nu_n)}_{k_n\,\mathrm{times}}\]
to $(\mu, \nu_{\boxplus\boxplus}^{(\gamma_1', \gamma_2', \sigma_1', \sigma_2', \sigma')})$ for some Borel probability measure $\mu$ on $\bR^2$. In this case, the pair $(\mu, \nu_{\boxplus\boxplus}^{(\gamma_1', \gamma_2', \sigma_1', \sigma_2', \sigma')})$ is $\boxplus\boxplus_{\mathrm{c}}$-infinitely divisible and
\[\Phi_{(\mu, \nu_{\boxplus\boxplus}^{(\gamma_1', \gamma_2', \sigma_1', \sigma_2', \sigma')})}(z, w) = E_{\mu_{\uplus\uplus}^{(\gamma_1, \gamma_2, \sigma_1, \sigma_2, \sigma)}}(z, w)\]
for all $(z, w) \in (\bC \setminus \bR)^2$.

We conclude this section with an example/counterexample. As mentioned above, all Borel probability measures on $\bR$ are $\uplus$-infinitely divisible. On the other hand, given a Borel probability measure $\mu$ on $\bR^2$, due to the (relatively) simple form of $E_\mu$ (compared to $\phi_\mu$ for example) as given in Definition \ref{defn:partialselfenergy}, it is tempting to hypothesize that all Borel probability measures on $\bR^2$ are $\uplus\uplus$-infinitely divisible. However, this is not true as illustrated by the following example. Consequently, the bi-Boolean L\'{e}vy-Hin\v{c}in formula \eqref{BiBooleanLH} best characterizes the class of $\uplus\uplus$-infinitely divisible Borel probability measures on $\bR^2$.

\begin{exam}
Let $\mu$ be the probability measure on $\bR^2$ defined by
\[\mu = \frac{1}{2}\delta_{(1, 0)} + \frac{1}{2}\delta_{(0, 1)}.\]
Then for $p, q \geq 0$, the $(p, q)^{\mathrm{th}}$ moment of $\mu$ is given by
\[\bM_{p, q}(\mu) = \int_{\bR^2}s^pt^q\,d\mu(s, t) = \begin{cases}
1 & \text{if } p = q = 0\\
\frac{1}{2} & \text{if } (p, q) \in \{(k, 0), (0, k)\,\mid\,k \geq 1\}\\
0 & \text{otherwise}
\end{cases}.\]
Using the bi-Boolean moment-cumulant formula \eqref{BB-M-C}, it is easy to compute some of the low order B-$(\ell, r)$-cumulants, which are given as follows:
\[B_{1, 0}(\mu) = B_{0, 1}(\mu) = \frac{1}{2}, \quad B_{2, 0}(\mu) = B_{0, 2}(\mu) = \frac{1}{4}, \quad B_{1, 1}(\mu) = -\frac{1}{4},\]
\[B_{2, 1}(\mu) = B_{1, 2}(\mu) = -\frac{1}{8}, \qand B_{2, 2}(\mu) = -\frac{1}{16}.\]

If $\mu$ is $\uplus\uplus$-infinitely divisible, then for every $n \in \bN$, there exists a probability measure $\mu_n$ on $\bR^2$ such that $E_\mu = nE_{\mu_n}$, and hence
\[B_{p, q}(\mu_n) = \frac{1}{n}B_{p, q}(\mu)\]
for all $p, q \geq 0$. Consequently, we have that
\begin{align*}
\bM_{2, 2}(\mu_n) &= B_{2, 2}(\mu_n) + B_{2, 1}(\mu_n)B_{0, 1}(\mu_n) + B_{1, 0}(\mu_n)B_{1, 2}(\mu_n) + B_{2, 0}(\mu_n)B_{0, 2}(\mu_n)\\
& \quad + B_{2, 0}(\mu_n)B_{1, 0}(\mu_n)^2 + B_{1, 0}(\mu_n)B_{1, 1}(\mu_n)B_{0, 1}(\mu_n) + B_{1, 0}(\mu_n)^2B_{0, 2}(\mu_n) + B_{1, 0}(\mu_n)^2B_{0, 1}(\mu_n)^2\\
&= -\frac{1}{16n} - \frac{1}{16n^2} - \frac{1}{16n^2} + \frac{1}{16n^2} + \frac{1}{16n^3} - \frac{1}{16n^3} + \frac{1}{16n^3} + \frac{1}{16n^4}\\
&= -\frac{1}{16n} - \frac{1}{16n^2} + \frac{1}{16n^3} + \frac{1}{16n^4}
\end{align*}
which is negative for $n > 1$, contradicting the assumption that $\mu_n$ is a positive measure.
\end{exam}

\section{Additive bi-Fermi convolution}\label{sec:bifermi}

In this section, we consider another special case of the additive c-bi-free convolution.

\subsection{Additive Fermi convolution}

In \cite{O2002}, Oravecz defined the additive Fermi convolution using the additive Boolean convolution with shifts, and it was shown in \cite{O2004} that the two convolutions are minimal in a certain sense in terms of the corresponding combinatorics. For simplicity, we assume all measures are compactly supported Borel probability measures so that all moments are finite. Given two such measures $\mu_1$ and $\mu_2$ on $\bR$ let $\widetilde{\mu}_j$ denote the zero-mean shift of $\mu_j$; that is, 
\[d\widetilde{\mu}_j(s) = d\mu_j(s + \bM_1(\mu_j)).
\]
The \textit{additive Fermi convolution} of $\mu_1$ and $\mu_2$, denoted $\mu_1 \bullet \mu_2$, is defined by
\[d\mu_1 \bullet \mu_2(s) = d\mu(s - (\bM_1(\mu_1) + \bM_1(\mu_2))),\]
where $\mu = \widetilde{\mu}_1 \uplus \widetilde{\mu}_2$ is the additive Boolean convolution of $\widetilde{\mu}_1$ and $\widetilde{\mu}_2$. Equivalently, $\mu_1 \bullet \mu_2$ can be defined in terms of the additive c-free convolution as
\[(\mu_1 \bullet \mu_2, \delta_{\bM_1(\mu_1) + \bM_1(\mu_2)}) = (\mu_1, \delta_{\bM_1(\mu_1)}) \boxplus_{\mathrm{c}} (\mu_2, \delta_{\bM_1(\mu_2)}).\]

On the combinatorial side, the \textit{Fermi cumulants} of a given compactly supported Borel probability measure $\mu$ on $\bR$ is the sequence $\{\gamma_n(\mu)\}_{n \geq 1}$ defined by
\[\bM_n(\mu) = \sum_{\pi \in \A\I(n)}\left(\prod_{V \in \pi}\gamma_{|V|}(\mu)\right)\]
for all $n \geq 1$, where $\A\I(n)$ denotes the set of \textit{almost interval partitions} of $\{1, \dots, n\}$; that is, if $\pi \in \A\I(n)$, then $\pi$ does not contain inner blocks other than singletons. As shown in \cite{O2002}*{Corollary 2.1},
\[\gamma_n(\mu_1 \bullet \mu_2) = \gamma_n(\mu_1) + \gamma_n(\mu_2)\]
for all $n \geq 1$, which justifies the name `Fermi cumulants'. On the analytic side, the linearizing transform with respect to $\bullet$ is
\[H_\mu(z) = \sum_{n \geq 1}\gamma_n(\mu)z^n\]
and it follows from \cite{O2002}*{Proposition 2.1} that
\[H_\mu(z) = \bM_1(\mu)z + \eta_{\widetilde{\mu}}(z),\]
where $\widetilde{\mu}$ denotes the zero-mean shift of $\mu$.

\subsection{Bi-Fermi cumulants and transform}

In view of the relation between the additive Fermi and c-free convolutions, we define the additive bi-Fermi convolution as follows. Recall that for $\mu$ a compactly supported Borel probability measure on $\bR^2$ and $m, n \geq 0$ with $m + n \geq 1$, we denote the $(m, n)^{\mathrm{th}}$ moment of $\mu$ by $\bM_{m, n}(\mu)$.

\begin{defn}
Let $\mu_1$ and $\mu_2$ be compactly supported Borel probability measures on $\bR^2$. The \textit{additive bi-Fermi convolution} of $\mu_1$ and $\mu_2$, denoted $\mu_1 \bullet\bullet \mu_2$, is defined by
\[(\mu_1 \bullet\bullet \mu_2, \delta_{(\bM_{1, 0}(\mu_1) + \bM_{1, 0}(\mu_2), \bM_{0, 1}(\mu_1) + \bM_{0, 1}(\mu_2))}) = (\mu_1, \delta_{(\bM_{1, 0}(\mu_1), \bM_{0, 1}(\mu_1))}) \boxplus\boxplus_{\mathrm{c}} (\mu_2, \delta_{(\bM_{1, 0}(\mu_2), \bM_{0, 1}(\mu_2))}),\]
where $\boxplus\boxplus_{\mathrm{c}}$ denotes the additive c-bi-free convolution.
\end{defn}

Like the additive Fermi convolution, $\mu_1 \bullet\bullet \mu_2$ can be equivalently defined using the additive bi-Boolean convolution as follows: Let
\[d\widetilde{\mu}_j(s, t) = d\mu_j(s + \bM_{1, 0}(\mu_j), t + \bM_{0, 1}(\mu_j))\]
for $j = 1, 2$, then $\mu_1 \bullet\bullet \mu_2$ is defined by
\[d\mu_1 \bullet\bullet \mu_2(s, t) = d\mu(s - (\bM_{1, 0}(\mu_1) + \bM_{1, 0}(\mu_2)), t - (\bM_{0, 1}(\mu_1) + \bM_{0, 1}(\mu_2))),\]
where $\mu = \widetilde{\mu}_1 \uplus\uplus \widetilde{\mu}_2$ is the additive bi-Boolean convolution of $\widetilde{\mu}_1$ and $\widetilde{\mu}_2$. As will be seen below, it is more convenient to work with this definition in terms of $\uplus\uplus$, but we choose to define $\bullet\bullet$ in terms of $\boxplus\boxplus_{\mathrm{c}}$ because it does not involve any shifts, and $\uplus\uplus$ is itself a special case of $\boxplus\boxplus_{\mathrm{c}}$.

To define the cumulants corresponding to the additive bi-Fermi convolution, we introduce the following set of partitions.

\begin{defn}
Let $n \geq 1$ and $\chi: \{1, \dots, n\} \to \{\ell, r\}$. A partition $\pi$ of $\{1, \dots, n\}$ is said to be an \textit{almost bi-interval partition} (with respect to $\chi$) if $s_\chi^{-1}\cdot\pi \in \A\I(n)$. Equivalently, $\pi$ is an almost bi-interval partition if $\pi$ does not contain interior blocks other than singletons. The set of almost bi-interval partitions is denoted by $\A\B\I(\chi)$. 
\end{defn}

Note that both $0_\chi$ and $1_\chi$ are elements of $\A\B\I(\chi)$, but $\A\B\I(\chi)$ is not a lattice (with respect to the partial order $\leq$ of refinement) for the same reason that $\A\I(n)$ is not a lattice. Moreover, if $V$ is a block of $\pi$, then denote
\[V_\ell := V \cap \chi^{-1}(\{\ell\}) \qand V_r := V \cap \chi^{-1}(\{r\}).\]
Finally, recall that for $m, n \geq 0$ with $m + n \geq 1$, $\chi_{m, n}: \{1, \dots, m + n\} \to \{\ell, r\}$ is the map such that $\chi_{m, n}(k) = \ell$ if $k \leq m$ and $\chi_{m, n}(k) = r$ if $k > m$.

\begin{defn}\label{BiFermiCumulants}
Let $\mu$ be a compactly supported Borel probability measure on $\bR^2$. The \textit{bi-Fermi cumulants} of $\mu$ is the sequence $\{\gamma_{m, n}(\mu)\}_{m, n \geq 0, m + n \geq 1}$ defined by
\[\bM_{m, n}(\mu) = \sum_{\pi \in \A\B\I(\chi_{m, n})}\left(\prod_{V \in \pi}\gamma_{|V_\ell|, |V_r|}(\mu)\right)\]
for all $m, n \geq 0$ with $m + n \geq 1$.
\end{defn}

Note that if $\mu$ has zero-mean marginal distributions then
\[\bM_{m, n}(\mu) = \sum_{\pi \in \B\I^*(\chi_{m, n})}\left(\prod_{V \in \pi}\gamma_{|V_\ell|, |V_r|}(\mu)\right),\]
where $\B\I^*(\chi_{m, n})$ denotes the set of bi-interval partitions with respect to $\chi_{m, n}$ without any singletons. On the other hand, we have that
\[\bM_{m, n}(\mu) = \sum_{\pi \in \B\I^*(\chi_{m, n})}\left(\prod_{V \in \pi}B_{|V_\ell|, |V_r|}(\mu)\right),\]
where $\{B_{m, n}(\mu)\}_{m, n \geq 0, m + n \geq 1}$ denotes the set of B-$(\ell, r)$-cumulants of $\mu$. Hence, in this case, $\gamma_{m, n}(\mu) = B_{m, n}(\mu)$ for all $m, n \geq 0$ with $m + n \geq 1$.

Note that as shown in \cite{O2002}*{Proposition 2.1}, we have that $\gamma_{m, 0}(\mu) = \gamma_{m, 0}(\widetilde{\mu})$ and $\gamma_{0, n}(\mu) = \gamma_{0, n}(\widetilde{\mu})$ for all $m, n \geq 2$. Therefore, the following shows that for all $m, n \geq 0$ with $m + n \geq 2$, we have that $\gamma_{m, n}(\mu) = B_{m, n}(\widetilde{\mu})$, while $\gamma_{1, 0}(\mu) = \bM_{1, 0}(\mu)$, $\gamma_{0, 1}(\mu) = \bM_{0, 1}(\mu)$, and $B_{1, 0}(\widetilde{\mu}) = B_{0, 1}(\widetilde{\mu}) = 0$.

\begin{prop}\label{ZeroMeanShift}
Let $\mu$ be a compactly supported Borel probability measure on $\bR^2$ and define $\widetilde{\mu}$ by
\[d\widetilde{\mu}(s, t) = d\mu(s + \bM_{1, 0}(\mu), t + \bM_{0, 1}(\mu)).\]
Then $\gamma_{m, n}(\mu) = \gamma_{m, n}(\widetilde{\mu})$ for all $m, n \geq 1$.
\end{prop}

\begin{proof}
For $m, n \geq 1$, we have that
\begin{align*}
\bM_{m, n}(\mu) &= \sum_{\pi \in \A\B\I(\chi_{m, n})}\left(\prod_{V \in \pi}\gamma_{|V_\ell|, |V_r|}(\mu)\right) \\
&= \sum_{p = 0}^m\sum_{q = 0}^n\binom{m}{m - p}\binom{n}{n - q}\bM_{1, 0}(\mu)^{m - p}\bM_{0, 1}(\mu)^{n - q}\sum_{\pi \in \B\I^*(\chi_{p, q})}\left(\prod_{V \in \pi}\gamma_{|V_\ell|, |V_r|}(\mu)\right)
\end{align*}
by Definition \ref{BiFermiCumulants} whereas, since $\widetilde{\mu}$ has zero-mean marginal distributions, we have that
\[\bM_{m, n}(\widetilde{\mu}) = \sum_{\pi \in \B\I^*(\chi_{m, n})}\left(\prod_{V \in \pi}\gamma_{|V_\ell|, |V_r|}(\widetilde{\mu})\right).\]

On the other hand, we have that
\begin{align*}
\bM_{m, n}(\mu) &= \int_{\bR^2}s^mt^n\,d\mu(s, t)\\
&= \int_{\bR^2}s^mt^n\,d\widetilde{\mu}(s - \bM_{1, 0}(\mu), t - \bM_{0, 1}(\mu))\\
&= \int_{\bR^2}(x + \bM_{1, 0}(\mu))^m(y + \bM_{0, 1}(\mu))^n\,d\widetilde{\mu}(x, y)\\
&= \int_{\bR^2}\sum_{p = 0}^m\binom{m}{p}x^p\bM_{1, 0}(\mu)^{m - p}\sum_{q = 0}^n\binom{n}{q}y^q\bM_{0, 1}(\mu)^{n - q}\,d\widetilde{\mu}(s, t)\\
&= \sum_{p = 0}^m\sum_{q = 0}^n\binom{m}{p}\binom{n}{q}\bM_{1, 0}(\mu)^{m - p}\bM_{0, 1}(\mu)^{n - q}\bM_{p, q}(\widetilde{\mu})\\
&= \sum_{p = 0}^m\sum_{q = 0}^n\binom{m}{p}\binom{n}{q}\bM_{1, 0}(\mu)^{m - p}\bM_{0, 1}(\mu)^{n - q}\sum_{\pi \in \B\I^*(\chi_{p, q})}\left(\prod_{V \in \pi}\gamma_{|V_\ell|, |V_r|}(\widetilde{\mu})\right).
\end{align*}
Thus the result follows by induction.
\end{proof}

\begin{cor}
Let $\mu_1$ and $\mu_2$ be compactly supported Borel probability measures on $\bR^2$, we have that
\[\gamma_{m, n}(\mu_1 \bullet\bullet \mu_2) = \gamma_{m, n}(\mu_1) + \gamma_{m, n}(\mu_2)\]
for all $m, n \geq 0$ with $m + n \geq 1$.
\end{cor}

In view of the above corollary, the linearizing transform with respect to the additive bi-Fermi convolution is defined as follows.

\begin{defn}
Let $\mu$ be a compactly supported Borel probability measure on $\bR^2$. The \textit{bi-Fermi transform} of $\mu$ is defined by
\[H_\mu(z, w) = \sum_{\substack{m, n \geq 0\\m + n \geq 1}}\gamma_{m, n}(\mu)z^mw^n.\]
\end{defn}

\begin{prop}
Let $\mu$ be a compactly supported Borel probability measure on $\bR^2$. The bi-Fermi transform $H_\mu$ of $\mu$ is given by
\[H_\mu(z, w) = H_{\mu^{(1)}}(z) + H_{\mu^{(2)}}(w) + \frac{G_\mu(1/z + \bM_{1, 0}(\mu), 1/w + \bM_{0, 1}(\mu))}{G_{\mu^{(1)}}(1/z + \bM_{1, 0}(\mu))G_{\mu^{(2)}}(1/w + \bM_{0, 1}(\mu))} - 1\]
for $(z, w) \in (\bC \setminus \bR)^2$.
\end{prop}

\begin{proof}
Let $\mu^{(1)}$ and $\mu^{(2)}$ be the marginal distributions of $\mu$, we have that
\[H_\mu(z, w) = H_{\mu^{(1)}}(z) + H_{\mu^{(2)}}(w) + \sum_{m, n \geq 1}\gamma_{m, n}(\mu)z^mw^n.\]
Moreover, with $\widetilde{\mu}$ as defined in Proposition \ref{ZeroMeanShift}, we have by Theorem \ref{PartialEta}that
\begin{align*}
\sum_{m, n \geq 1}\gamma_{m, n}(\mu)z^mw^n = \sum_{m, n \geq 1}\gamma_{m, n}(\widetilde{\mu})z^mw^n &= \sum_{m, n \geq 1}B_{m, n}(\widetilde{\mu})z^mw^n \\
&= \frac{G_{\widetilde{\mu}}(1/z, 1/w)}{G_{\widetilde{\mu}^{(1)}}(z)G_{\widetilde{\mu}^{(2)}}(w)} - 1\\
&= \frac{G_\mu(1/z + \bM_{1, 0}(\mu), 1/w + \bM_{0, 1}(\mu))}{G_{\mu^{(1)}}(1/z + \bM_{1, 0}(\mu))G_{\mu^{(2)}}(1/w + \bM_{0, 1}(\mu))} - 1
\end{align*}
as claimed.
\end{proof}

\subsection{Limit theorems}

In terms of limit theorems, if the measures have zero-mean marginal distributions, then the additive bi-Fermi convolution coincides with the additive bi-Boolean convolution, and thus the bi-Fermi central limit theorem and centred bi-Fermi Gaussian distributions are same as the bi-Boolean ones. For Poisson type limit theorems, recall that a compactly supported Borel probability measure $\pi_{\lambda, \sigma}^{\uplus\uplus}$ on $\bR^2$ is said to have a compound bi-Boolean Poisson distribution with rate $\lambda \geq 0$ and jump distribution $\sigma \neq \delta_{(0, 0)}$ if
\[B_{m, n}(\pi_{\lambda, \sigma}^{\uplus\uplus}) = \lambda\cdot\bM_{m, n}(\sigma)\]
for all $m, n \geq 0$ with $m + n \geq 1$. Analogously, a compound bi-Fermi Poisson distribution $\pi_{\lambda, \sigma}^{\bullet\bullet}$ with rate $\lambda \geq 0$ and jump distribution $\sigma \neq \delta_{(0, 0)}$ is characterized by the requirement that
\[\gamma_{m, n}(\pi_{\lambda, \sigma}^{\bullet\bullet}) = \lambda\cdot\bM_{m, n}(\sigma)\]
for all $m, n \geq 0$ with $m + n \geq 1$. For simplicity, let $\sigma = \delta_{(1, 1)}$ so that
\[\eta_{\pi_{\lambda, \delta_{(1, 1)}}^{\uplus\uplus}}(z, w) = H_{\pi_{\lambda, \delta_{(1, 1)}}^{\bullet\bullet}}(z, w) = \frac{\lambda z}{1 - z} + \frac{\lambda w}{1 - w} + \frac{\lambda zw}{(1 - z)(1 - w)}.\]
Therefore
\[\frac{G_{\pi_{\lambda, \delta_{(1, 1)}}^{\uplus\uplus}}(1/z, 1/w)}{G_{(\pi_{\lambda, \delta_{(1, 1)}}^{\uplus\uplus})^{(1)}}(1/z)G_{(\pi_{\lambda, \delta_{(1, 1)}}^{\uplus\uplus})^{(2)}}(1/w)}  = \frac{G_{\pi_{\lambda, \delta_{(1, 1)}}^{\bullet\bullet}}(1/z + \lambda, 1/w + \lambda)}{G_{(\pi_{\lambda, \delta_{(1, 1)}}^{\bullet\bullet})^{(1)}}(1/z + \lambda)G_{(\pi_{\lambda, \delta_{(1, 1)}}^{\bullet\bullet})^{(2)}}(1/w + \lambda)} = \frac{\lambda zw}{(1 - z)(1 - w)} +1.\]
Hence, the distributions of $\pi_{\lambda, \delta_{(1, 1)}}^{\uplus\uplus}$ and $\pi_{\lambda, \delta_{(1, 1)}}^{\bullet\bullet}$ are quite different since
\[G_{\pi_{\lambda, \delta_{(1, 1)}}^{\uplus\uplus}}(z, w) = \frac{\lambda + (z - 1)(w - 1)}{zw(z - 1 - \lambda)(w - 1 - \lambda)},\]
whereas
\[G_{\pi_{\lambda, \delta_{(1, 1)}}^{\bullet\bullet}}(z, w) = \frac{\lambda + (z - 1 - \lambda)(w - 1 - \lambda)}{((z - \lambda)^2 - z)((w - \lambda)^2 - w)}.\]

Using the additivity of the bi-Fermi cumulants, the following compound bi-Fermi Poisson limit theorem can be easily obtained.

\begin{thm}
Let $\lambda \geq 0$ and let $\sigma \neq \delta_{(0, 0)}$ be a compactly supported Borel probability measure on $\bR^2$. For $N \in \bN$, let $\mu_N = \left(1 - \frac{\lambda}{N}\right)\delta_{(0, 0)} + \frac{\lambda}{N}\sigma$.  Then $\lim_{N \to \infty}\underbrace{\mu_N \bullet\bullet \cdots \bullet\bullet \mu_N}_{N\,\mathrm{times}} = \pi_{\lambda, \sigma}^{\bullet\bullet}$.
\end{thm}

\section{Connection with bi-free independence}\label{sec:bifreeness}

In this section, we extend some of the results in \cite{BN2008-1} to pairs of algebras. Due to similar lattice structures, most of the combinatorial arguments for $\N\C(n)$ immediately generalize to $\B\N\C(\chi)$. In that which follows, a general two-faced family will be denoted as $\widehat{a} = ((a_i)_{i \in I}, (a_j)_{j \in J})$, and an arbitrary $n$-tuple of elements from $\widehat{a}$ is recorded by a map $\alpha: \{1, \dots, n\} \to I \sqcup J$ so that $\alpha$ corresponds to $(a_{\alpha(1)}, \dots, a_{\alpha(n)})$. In this case, the corresponding map $\chi_\alpha: \{1, \dots, n\} \to \{\ell, r\}$ is defined by $\chi_\alpha(k) = \ell$ if $\alpha(k) \in I$ and $\chi_\alpha(k) = r$ if $\alpha(k) \in J$ for $1 \leq k \leq n$.

\subsection{A two-faced extension of the maps $\mathbb{B}$ and $\mathrm{Reta}$}

In \cite{BN2008-1}, Belinschi and Nica introduced two bijections, $\mathbb{B}$ and $\mathrm{Reta}$, where $\mathbb{B}$ is a multi-variable analogue of the Boolean Bercovici-Pata bijection and $\mathrm{Reta}$ is a bijection which converts the free $\R$-transform to the Boolean $\eta$-transform. Various properties of $\mathbb{B}$ and $\mathrm{Reta}$ were proved including an explicit formula describing $\mathrm{Reta}$. The goal of this subsection is to extend these maps to the two-faced setting.

Denote by $\D_{\alg}(I \sqcup J)$ the set of all joint distributions of two-faced families with left index set $I$ and right index set $J$, which is naturally identified as the set of unital linear functionals from $\bC\langle Z_k : k \in I \sqcup J\rangle$ to $\bC$. On the other hand, denote by $\bC_0\langle\!\langle z_k : k \in I \sqcup J\rangle\!\rangle$ the set of series with complex coefficients in the non-commuting indeterminates $\{z_k\}_{k \in I \sqcup J}$ with vanishing constant term.

\begin{defn}\label{MultiTransforms}
Let $\widehat{a} = ((a_i)_{i \in I}, (a_j)_{j \in J})$ be a two-faced family in a non-commutative probability space $(\A, \varphi)$. We define the following series in $\bC_0\langle\!\langle z_k : k \in I \sqcup J\rangle\!\rangle$.
\begin{enumerate}[$\qquad(1)$]
\item The \textit{moment series} of $\widehat{a}$ is
\[M_{\widehat{a}}= \sum_{n \geq 1}\sum_{\alpha: \{1, \dots, n\} \to I \sqcup J}\varphi(a_{\alpha(1)}\cdots a_{\alpha(n)})z_{\alpha(1)}\cdots z_{\alpha(n)}.\]

\item The \textit{bi-free $\R$-transform} of $\widehat{a}$ is
\[\R_{\widehat{a}}= \sum_{n \geq 1}\sum_{\alpha: \{1, \dots, n\} \to I \sqcup J}\kappa_{\chi_\alpha}(a_{\alpha(1)}, \dots, a_{\alpha(n)})z_{\alpha(1)}\cdots z_{\alpha(n)}.\]

\item The \textit{bi-Boolean $\eta$-transform} of $\widehat{a}$ is
\[\eta_{\widehat{a}}= \sum_{n \geq 1}\sum_{\alpha: \{1, \dots, n\} \to I \sqcup J}B_{\chi_\alpha}(a_{\alpha(1)}, \dots, a_{\alpha(n)})z_{\alpha(1)}\cdots z_{\alpha(n)}.\]
\end{enumerate}
\end{defn}

Note given a series $f \in \mathbb{C}_0\langle\!\langle z_k : k \in I \sqcup J\rangle\!\rangle$, one can always find some $\widehat{a}$ in $(\A, \varphi)$ such that $M_{\widehat{a}} = f$ (or $\R_{\widehat{a}} = f$ or $\eta_{\widehat{a}} = f$) by taking $\A = \bC\langle Z_k : k \in I \sqcup J\rangle$, $a_k = Z_k$, and define $\varphi(a_{\alpha(1)}\cdots a_{\alpha(n)})$ (respectively $\kappa_{\chi_{\alpha}}(a_{\alpha(1)},\ldots,  a_{\alpha(n)})$, respectively $B_{\chi_{\alpha}}(a_{\alpha(1)},\ldots,  a_{\alpha(n)})$)  to be the coefficient of $z_{\alpha(1)}\cdots z_{\alpha(n)}$ in $f$. Consequently, the maps
\[M, \R, \eta: \D_\alg(I \sqcup J) \to \mathbb{C}_0\langle\!\langle z_k : k \in I \sqcup J\rangle\!\rangle\]
are bijections.

\begin{defn}
Let $I$ and $J$ be disjoint index sets. The maps $\bB$ and $\bReta$ are bijections
\[\bB: \D_\alg(I \sqcup J) \to \D_\alg(I \sqcup J) \qand \bReta: \mathbb{C}_0\langle\!\langle z_k : k \in I \sqcup J\rangle\!\rangle \to \mathbb{C}_0\langle\!\langle z_k : k \in I \sqcup J\rangle\!\rangle\]
defined by $\bB = \R^{-1} \circ \eta$ and $\bReta = \eta \circ \R^{-1}$.
\end{defn}

Following \cite{BN2008-1}*{Definition 3.2}, if $f \in \mathbb{C}_0\langle\!\langle z_k : k \in I \sqcup J\rangle\!\rangle$, then $\Cf_{(\alpha(1), \dots, \alpha(n))}(f)$ denotes the coefficient of $z_{\alpha(1)}\cdots z_{\alpha(n)}$ in $f$. More generally, if $\pi$ is a partition of $\{1, \dots, n\}$, then $\Cf_{(\alpha(1), \dots, \alpha(n)); \pi}(f)$ denotes the product
\[\prod_{V \in \pi}\Cf_{(\alpha(1), \dots, \alpha(n))|_V}(f),\]
which is in general not a coefficient in $f$. In deriving an explicit formula relating the coefficients of $f$ to those of $\bReta(f)$ for $f \in \mathbb{C}_0\langle\!\langle z_k : k \in I \sqcup J\rangle\!\rangle$, the following partial order on bi-non-crossing partitions will be used.

\begin{defn}
Let $n \geq 1$ and $\chi: \{1, \dots, n\} \to \{\ell, r\}$. For two bi-non-crossing partitions $\sigma$ and $\pi$ in $\B\N\C(\chi)$, we write $\sigma \ll_\chi \pi$ to mean that $\sigma \leq \pi$ and, in addition, for every block $V$ of $\pi$, there exists a block $W$ of $\sigma$ such that $\min_{\prec_\chi}(V), \max_{\prec_\chi}(V) \in W$.
\end{defn}

Note that if $\chi: \{1, \dots, n\} \to \{\ell, r\}$ is constant, then $\sigma, \pi \in \N\C(n)$ and $\ll_\chi$ is exactly the partial order $\ll$ inroduced in \cite{BN2008-1}*{Definition 2.5}.

\begin{prop}\label{bRetaFormula}
Let $f$ and $g$ be series in $\mathbb{C}_0\langle\!\langle z_k : k \in I \sqcup J\rangle\!\rangle$ such that $\bReta(f) = g$.
\begin{enumerate}[$\qquad(1)$]
\item For all $n \geq 1$ and $\alpha: \{1, \dots, n\} \to I \sqcup J$, we have that
\[\Cf_{(\alpha(1), \dots, \alpha(n))}(g) = \sum_{\substack{\pi \in \B\N\C(\chi_\alpha)\\\pi \ll_{\chi_\alpha} 1_{\chi_\alpha}}}\Cf_{(\alpha(1), \dots, \alpha(n)); \pi}(f).\]

\item For all $n \geq 1$ and $\alpha: \{1, \dots, n\} \to I \sqcup J$, we have that
\[\Cf_{(\alpha(1), \dots, \alpha(n))}(f) = \sum_{\substack{\pi \in \B\N\C(\chi_\alpha)\\\pi \ll_{\chi_\alpha} 1_{\chi_\alpha}}}(-1)^{|\pi| - 1}\Cf_{(\alpha(1), \dots, \alpha(n)); \pi}(g).\]
\end{enumerate}
\end{prop}

\begin{proof}
The proof is identical to that in \cite{BN2008-1}*{Proposition 3.9} once permutations are applied and thus is omitted.
\end{proof}

Returning to the transforms introduced in Definition \ref{MultiTransforms}, note that the operations $\boxplus\boxplus$ and $\uplus\uplus$ can be defined on $\D_\alg(I \sqcup J)$ by the requirements that for $\mu, \nu \in \D_\alg(I \sqcup J)$, $\mu \boxplus\boxplus \nu$ and $\mu \uplus\uplus \nu$ are the unique elements of $\D_\alg(I \sqcup J)$ such that $\R_{\mu \boxplus\boxplus \nu} = \R_{\mu} + \R_{\nu}$ and $\eta_{\mu \uplus\uplus \nu} = \eta_{\mu} + \eta_{\nu}$. Using Proposition \ref{bRetaFormula}, a mutli-variable version of the equivalence of assertions $(1)$ and $(2)$ in Theorem \ref{TDBP} can be easily obtained in the current algebraic framework. The proof is nearly identical to the one-sided case considered in \cite{BN2008-1}*{Section 5}, and hence omitted.

\begin{prop}
Let $\{\mu_n\}_{n \geq 1}$ be a sequence in $\D_{\alg}(I \sqcup J)$ and $\{k_n\}_{n \geq 1}$ be a sequence of positive integers with $\lim_{n \to \infty}k_n = \infty$. The following assertions are equivalent.
\begin{enumerate}[$\qquad(1)$]
\item The sequence $\underbrace{\mu_n \boxplus\boxplus \cdots \boxplus\boxplus \mu_n}_{k_n\,\mathrm{times}}$ converges in moments to some $\mu \in \D_{\alg}(I \sqcup J)$.

\item The sequence $\underbrace{\mu_n \uplus\uplus \cdots \uplus\uplus \mu_n}_{k_n\,\mathrm{times}}$ converges in moments to some $\nu \in \D_{\alg}(I \sqcup J)$.
\end{enumerate}
Moreover, if these assertions hold, then $\mu = \bB(\nu)$.
\end{prop}

\subsection{The twisted multiplicative bi-free convolution}

The second main result of \cite{BN2008-1} concerns a special property of $\mathbb{B}$; namely that $\mathbb{B}$ is a homomorphism with respect to the multiplicative free convolution $\boxtimes$. It is thus natural to expect that a similar property holds for $\bB$. However, it turns out that $\bB$ is not a homomorphism with respect to the `usual' multiplicative bi-free convolution $\boxtimes\boxtimes$ in the literature (see \cites{V2016, S2016-1}) but with respect to the one in \cites{CNS2015-1, CNS2015-2}. Consequently, we consider the following operation.

\begin{defn}
Let $\widehat{a} = ((a_i)_{i \in I}, (a_j)_{j \in J})$ and $\widehat{b} = ((b_i)_{i \in I}, (b_j)_{j \in J})$ be bi-free two-faced families in a non-commutative probability space $(\A, \varphi)$ with joint distributions $\mu_{\widehat{a}}$ and $\mu_{\widehat{b}}$ respectively. The \textit{twisted multiplicative bi-free convolution} of $\mu_{\widehat{a}}$ and $\mu_{\widehat{b}}$, denoted $\mu_{\widehat{a}} \widetilde{\boxtimes\boxtimes} \mu_{\widehat{b}}$, is defined to be the joint distribution of the two-faced family $((a_ib_i)_{i \in I}, (b_ja_j)_{j \in J})$.
\end{defn}

Note that the usual multiplicative bi-free convolution $\mu_{\widehat{a}} \boxtimes\boxtimes \mu_{\widehat{b}}$ of $\mu_{\widehat{a}}$ and $\mu_{\widehat{b}}$ is defined to be the joint distribution of the two-faced family $((a_ib_i)_{i \in I}, (a_jb_j)_{j \in J})$. On the other hand, this twisted multiplicative bi-free convolution is not new, and has been previously considered in \cites{CNS2015-1, CNS2015-2} for the following reason. In \cite{NS1996}, the operation of \textit{boxed convolution} $\freestar$ was defined using the Kreweras complementation map $K_{\N\C}$ on the lattice of non-crossing partitions, which plays an important role in proving the mentioned property of $\mathbb{B}$. On the other hand, the \textit{bi-non-crossing Kreweras complementation map} $K_{\B\N\C}$ on the lattice of bi-non-crossing partitions was defined in \cite{CNS2015-1}*{Definition 5.1.1} by
\[K_{\B\N\C}(\pi) = s_\chi\cdot K_{\N\C}(s_\chi^{-1}\cdot\pi)\]
for $\chi: \{1, \dots, n\} \to \{\ell, r\}$ and $\pi \in \B\N\C(\chi)$. As shown in \cite{CNS2015-2}*{Proposition 9.2.1}, the map $K_{\B\N\C}$ can be used to compute the twisted multiplicative bi-free convolution as follows. If $\widehat{a} = ((a_i)_{i \in I}, (a_j)_{j \in J})$ and $\widehat{b} = ((b_i)_{i \in I}, (b_j)_{j \in J})$ are bi-free, then
\begin{equation}\label{Kreweras}
\kappa_{\chi_\alpha}(c_{\alpha(1)}, \dots, c_{\alpha(n)}) = \sum_{\pi \in \B\N\C(\chi_\alpha)}\kappa_{\pi}(a_{\alpha(1)}, \dots, a_{\alpha(n)})\cdot\kappa_{K_{\B\N\C}(\pi)}(b_{\alpha(1)}, \dots, b_{\alpha(n)})
\end{equation}
for $\alpha: \{1, \dots, n\} \to I \sqcup J$, where $c_{\alpha(k)} = a_{\alpha(k)}b_{\alpha(k)}$ if $\alpha(k) \in I$ and $c_{\alpha(k)} = b_{\alpha(k)}a_{\alpha(k)}$ if $\alpha(k) \in J$ for $1 \leq k \leq n$. This motivates us to define following operation on $\bC_0\langle\!\langle z_k : k \in I \sqcup J\rangle\!\rangle$.

\begin{defn}
Let $f$ and $g$ be series in $\bC_0\langle\!\langle z_k : k \in I \sqcup J\rangle\!\rangle$. The series $f\,\widetilde{\freestar}\,g$ in $\mathbb{C}_0\langle\!\langle z_k : k \in I \sqcup J\rangle\!\rangle$ is defined by
\[\Cf_{(\alpha(1), \dots, \alpha(n))}(f\,\widetilde{\freestar}\,g) = \sum_{\pi \in \B\N\C(\chi_\alpha)}\Cf_{(\alpha(1), \dots, \alpha(n)); \pi}(f)\cdot\mathrm{Cf}_{(\alpha(1), \dots, \alpha(n)); K_{\B\N\C}(\pi)}(g)\]
for all $n \geq 1$ and $\alpha: \{1, \dots, n\} \to I \sqcup J$.
\end{defn}

Note that if $\mu, \nu \in \D_\alg(I \sqcup J)$, then
\[\R_{\mu \widetilde{\boxtimes\boxtimes} \nu} = \R_{\mu}\,\widetilde{\freestar}\,\R_{\nu}\]
by equation \eqref{Kreweras}. Moreover, if $\pi \leq \rho \in \B\N\C(\chi)$, then the \textit{relative bi-non-crossing Kreweras complement} of $\pi$ in $\rho$, denoted $K_{\B\N\C; \rho}(\pi)$, is defined by
\[K_{\B\N\C; \rho}(\pi) = \{K_{\B\N\C}(\pi|_V)\}_{V \in \rho},\]
where for every block $V$ of $\rho$, $K_{\B\N\C}(\pi|_V)$ is the bi-non-crossing Kreweras complement of $\pi|V$ (which is a bi-non-crossing partition with respect to $\chi|_V$). Under this notion, equation \eqref{Kreweras} can be generalized to
\[\kappa_{\rho}(c_{\alpha(1)}, \dots, c_{\alpha(n)}) = \sum_{\substack{\pi \in \B\N\C(\chi_\alpha)\\\pi \leq \rho}}\kappa_{\pi}(a_{\alpha(1)}, \dots, a_{\alpha(n)})\cdot\kappa_{K_{\B\N\C; \rho}(\pi)}(b_{\alpha(1)}, \dots, b_{\alpha(n)})\]
for $\alpha: \{1, \dots, n\} \to I \sqcup J$ and $\rho \in \B\N\C(\chi_\alpha)$. We are now ready to present the main result of this subsection whose proof trivially follows that of \cite{BN2008-1}*{Theorem 7.2} with a permutation.

\begin{thm}\label{BN2008-1Thm7.2}
For two series $f$ and $g$ in $\bC_0\langle\!\langle z_k : k \in I \sqcup J\rangle\!\rangle$, we have that
\[\bReta(f\,\widetilde{\freestar}\,g) = \bReta(f)\,\widetilde{\freestar}\,\bReta(g).\]
\end{thm}

\begin{cor}
For any two distributions $\mu$ and $\nu$ in $\D_{\alg}(I \sqcup J)$, $\bB(\mu \widetilde{\boxtimes\boxtimes} \nu) = \bB(\mu) \widetilde{\boxtimes\boxtimes} \bB(\nu)$.
\end{cor}

\begin{proof}
By Theorem \ref{BN2008-1Thm7.2}, we have that
\begin{align*}
\R_{\bB(\mu \widetilde{\boxtimes\boxtimes} \nu)} &= \eta_{\mu \widetilde{\boxtimes\boxtimes} \nu}\\
&= \bReta(\R_{\mu \widetilde{\boxtimes\boxtimes} \nu})\\
&= \bReta(\R_\mu\,\widetilde{\freestar}\,\R_\nu)\\
&= \bReta(\R_\mu)\,\widetilde{\freestar}\,\bReta(\R_\nu)\\
&= \eta_\mu\,\widetilde{\freestar}\,\eta_\nu\\
&= \R_{\bB(\mu)}\,\widetilde{\freestar}\,\R_{\bB(\nu)}\\
&= \R_{\bB(\mu) \widetilde{\boxtimes\boxtimes} \bB(\nu)}.
\end{align*}
Since $\bB(\mu \widetilde{\boxtimes\boxtimes} \nu)$ and $\bB(\mu) \widetilde{\boxtimes\boxtimes} \bB(\nu)$ have the same bi-free $\R$-transform, the result follows.
\end{proof}

By comparing the first equality with the fifth equality, we see that the bi-Boolean $\eta$-transform also turns $\widetilde{\boxtimes\boxtimes}$ to $\widetilde{\freestar}$, which is not surprising since the Boolean $\eta$-transform behaves similarly when it comes to $\boxtimes$ and $\freestar$ (see \cite{BN2008-1}*{Theorem 2'}).

\section{Bi-Boolean independence with amalgamation}\label{sec:op-valued}

We conclude this paper with an extension of bi-Boolean independence to the amalgamated setting over an arbitrary algebra.

\subsection{Definition and review of bi-free independence with amalgamation}

As the theory of bi-free independence with amalgamation is well-developed in \cite{CNS2015-2}, the same structures can be used with some slight modifications to fit the current framework. Throughout the rest, $\B$ denotes a unital algebra over $\bC$. Recall first from \cite{CNS2015-2}*{Definition 3.2.1} that in order to discuss independence for pairs of algebras over $\B$, the usual notion of a non-commutative probability space $(\A, \varphi)$ should be replaced by a \textit{$\B$-$\B$-non-commutative probability space}, which is a triple $(\A, \bE, \varepsilon)$ where $\A$ is a unital algebra over $\bC$, $\varepsilon: \B \otimes \B^{\mathrm{op}} \to \A$ is a unital homomorphism such that $\varepsilon|_{\B \otimes 1_\B}$ and $\varepsilon|_{1_\B \otimes \B^{\mathrm{op}}}$ are injective, and $\bE: \A \to \B$ is a unital linear map such that
\[\bE(\varepsilon(b_1 \otimes b_2)Z) = b_1\bE(Z)b_2 \qand \bE(Z\varepsilon(b \otimes 1_\B)) = \bE(Z\varepsilon(1_\B \otimes b))\]
for all $b_1, b_2, b \in \B$ and $Z \in \A$. Moreover, the unital subalgebras $\A_\ell$ and $\A_r$ of $\A$ defined by
\begin{align*}
\A_\ell &= \{Z \in \A\,|\,Z\varepsilon(1_\B \otimes b) = \varepsilon(1_\B \otimes b)Z\,\text{ for all }\,b \in \B\}\\
\A_r &= \{Z \in \A\,|\,Z\varepsilon(b \otimes 1_\B) = \varepsilon(b \otimes 1_\B)Z\,\text{ for all }\,b \in \B\}
\end{align*}
will be called the \textit{left and right algebras} of $\A$ respectively. For notational simplicity, we will write $L_b$ and $R_b$ instead of $\varepsilon(b \otimes 1_\B)$ and $\varepsilon(1_\B \otimes b)$ and refer to them as \textit{left and right $\B$-operators} respectively.

As demonstrated in \cite{CNS2015-2}*{Theorem 3.2.4}, $\B$-$\B$-non-commutative probability spaces are the correct objects to study because every such space can be concretely represented as linear operators on a $\B$-$\B$-bimodule with a specified $\B$-vector state in an expectation-preserving way. We refer to \cite{CNS2015-2}*{Section 3} for more details. Furthermore, in order to discuss the corresponding moment and cumulant functions, one needs the notion of operator-valued bi-multiplicative functions.

\begin{defn}\label{BiMulti}
Let $(\A, \bE, \varepsilon)$ be a $\B$-$\B$-non-commutative probability space and let
\[\Phi: \bigcup_{n \geq 1}\bigcup_{\chi: \{1, \dots, n\} \to \{\ell, r\}}\B\N\C(\chi) \times \A_{\chi(1)} \times \cdots \times \A_{\chi(n)} \to \B\]
be a function that is linear in each $\A_{\chi(j)}$. It is said that $\Phi$ is \textit{operator-valued bi-multiplicative} if for every $\chi: \{1, \dots, n\} \to \{\ell, r\}$, $Z_j \in \A_{\chi(j)}$, $b \in \B$, and $\pi \in \B\N\C(\chi)$, the following four conditions hold.
\begin{enumerate}[$\qquad(1)$]
\item Let
\[q = \max\{j \in \{1, \dots, n\}\,|\,\chi(j) \neq \chi(n)\}.\]
If $\chi(n) = \ell$, then
\[\Phi_{1_\chi}(Z_1, \dots, Z_{n - 1}, Z_nL_b) = \begin{cases}
\Phi_{1_\chi}(Z_1, \dots, Z_{q - 1}, Z_qR_b, Z_{q + 1}, \dots, Z_n) &\text{if } q \neq -\infty\\
\Phi_{1_\chi}(Z_1, \dots, Z_{n - 1}, Z_n)b &\text{if } q = -\infty
\end{cases}.\]
If $\chi(n) = r$, then
\[\Phi_{1_\chi}(Z_1, \dots, Z_{n - 1}, Z_nR_b) = \begin{cases}
\Phi_{1_\chi}(Z_1, \dots, Z_{q - 1}, Z_qL_b, Z_{q + 1}, \dots, Z_n) &\text{if } q \neq -\infty\\
b\Phi_{1_\chi}(Z_1, \dots, Z_{n - 1}, Z_n) &\text{if } q = -\infty
\end{cases}.\]

\item Let $p \in \{1, \dots, n\}$, and let
\[q = \max\{j \in \{1, \dots, n\}\,|\,\chi(j) = \chi(p), j < p\}.\]
If $\chi(p) = \ell$, then
\[\Phi_{1_\chi}(Z_1, \dots, Z_{p - 1}, L_bZ_p, Z_{p + 1}, \dots, Z_n) = \begin{cases}
\Phi_{1_\chi}(Z_1, \dots, Z_{q - 1}, Z_qL_b, Z_{q + 1}, \dots, Z_n) &\text{if } q \neq -\infty\\
b\Phi_{1_\chi}(Z_1, Z_2, \dots, Z_n) &\text{if } q = -\infty
\end{cases}.\]
If $\chi(p) = r$, then
\[\Phi_{1_\chi}(Z_1, \dots, Z_{p - 1}, R_bZ_p, Z_{p + 1}, \dots, Z_n) = \begin{cases}
\Phi_{1_\chi}(Z_1, \dots, Z_{q - 1}, Z_qR_b, Z_{q + 1}, \dots, Z_n) &\text{if } q \neq -\infty\\
\Phi_{1_\chi}(Z_1, Z_2, \dots, Z_n)b &\text{if } q = -\infty
\end{cases}.\]

\item Suppose that $V_1, \dots, V_m$ are $\chi$-intervals ordered by $\prec_\chi$ which partition $\{1, \dots, n\}$, each a union of blocks of $\pi$. Then
\[\Phi_\pi(Z_1, \dots, Z_n) = \Phi_{\pi|_{V_1}}((Z_1, \dots, Z_n)|_{V_1})\cdots\Phi_{\pi|_{V_m}}((Z_1, \dots, Z_n)|_{V_m}).\]

\item Suppose that $V$ and $W$ partition $\{1, \dots, n\}$, each a union of blocks of $\pi$, $V$ is a $\chi$-interval, and
\[\min_{\prec_\chi}(\{1, \dots, n\}), \max_{\prec_\chi}(\{1, \dots, n\}) \in W.\]
Let
\[p = \max_{\prec_\chi}\left(\left\{j \in W\,|\,j \prec_\chi \min_{\prec_\chi}(V)\right\}\right) \qand q = \min_{\prec_\chi}\left(\left\{j \in W\,|\,\max_{\prec_\chi}(V) \prec_\chi j\right\}\right).\]
Then
\begin{align*}
\Phi_\pi(Z_1, \dots, Z_n) &= \begin{cases}
\Phi_{\pi|_{W}}\left(\left(Z_1, \dots, Z_{p - 1}, Z_pL_{\Phi_{\pi|_{V}}\left((Z_1, \dots, Z_n)|_{V}\right)}, Z_{p + 1}, \dots, Z_n\right)|_{W}\right) &\text{if } \chi(p) = \ell\\
\Phi_{\pi|_{W}}\left(\left(Z_1, \dots, Z_{p - 1}, R_{\Phi_{\pi|_{V}}\left((Z_1, \dots, Z_n)|_{V}\right)}Z_p, Z_{p + 1}, \dots, Z_n\right)|_{W}\right) &\text{if } \chi(p) = r
\end{cases}\\
&= \begin{cases}
\Phi_{\pi|_{W}}\left(\left(Z_1, \dots, Z_{q - 1}, L_{\Phi_{\pi|_{V}}\left((Z_1, \dots, Z_n)|_{V}\right)}Z_q, Z_{q + 1}, \dots, Z_n\right)|_{W}\right) &\text{if } \chi(q) = \ell\\
\Phi_{\pi|_{W}}\left(\left(Z_1, \dots, Z_{q - 1}, Z_qR_{\Phi_{\pi|_{V}}\left((Z_1, \dots, Z_n)|_{V}\right)}, Z_{q + 1}, \dots, Z_n\right)|_{W}\right) &\text{if } \chi(q) = r
\end{cases}.
\end{align*}
\end{enumerate}
\end{defn}

Given an operator-valued bi-multiplicative function, conditions $(1)$ to $(4)$ above allow one to move $\B$-operators around and completely determine values on all bi-non-crossing partitions based on full non-crossing partitions.

For bi-Boolean independence, since the sublattice of bi-interval partitions is used instead, the corresponding functions need only satisfy some weaker conditions.

\begin{defn}
Let $(\A, \bE, \varepsilon)$ be a $\B$-$\B$-non-commutative probability space and let
\[\Phi: \bigcup_{n \geq 1}\bigcup_{\chi: \{1, \dots, n\} \to \{\ell, r\}}\B\I(\chi) \times \A_{\chi(1)} \times \cdots \times \A_{\chi(n)} \to \B\]
be a function that is linear in each $\A_{\chi(j)}$. We say that $\Phi$ is \textit{operator-valued b-bi-multiplicative} if conditions $(1)$ to $(3)$ of Definition \ref{BiMulti} are satisfied for every $\chi: \{1, \dots, n\} \to \{\ell, r\}$, $Z_j \in \A_{\chi(j)}$, $b \in \B$, and $\pi \in \B\I(\chi)$.
\end{defn}

Note that condition $(4)$ of Definition \ref{BiMulti} is irrelevant here because if $W$ is a union of blocks of a bi-interval partition $\pi$ containing $\min_{\prec_\chi}(\{1, \dots, n\})$ and $\max_{\prec_\chi}(\{1, \dots, n\})$, then $V$ does not exist as $W$ must be $\pi$. Moreover, as every operator-valued bi-multiplicative function is automatically operator-valued b-bi-multiplicative, the notion of bi-Boolean independence over $\B$ can be introduced using the \textit{operator-valued bi-free moment function}, which is defined in \cite{CNS2015-2}*{Definition 5.1.3} as the operator-valued bi-multiplicative function
\[\E: \bigcup_{n \geq 1}\bigcup_{\chi: \{1, \dots, n\} \to \{\ell, r\}}\B\N\C(\chi) \times \A_{\chi(1)} \times \cdots \times \A_{\chi(n)} \to \B\]
such that
\[\E_{1_\chi}(Z_1, \dots, Z_n) = \bE(Z_1\cdots Z_n)\]
for every $\chi: \{1, \dots, n\} \to \{\ell, r\}$ and $Z_j \in \A_{\chi(j)}$.  (Note this really is the second operator-valued moment function from the operator-valued c-bi-free construction in \cite{GS2016-2} using the same ideas as in Proposition \ref{BiBandCBF}.)

\begin{defn}
Let $(\A, \bE, \ep)$ be a $\B$-$\B$-non-commutative probability space.
\begin{enumerate}[$\qquad(1)$]
\item A \textit{non-unital pair of $\B$-algebras} in $(\A, \bE, \ep)$ is an ordered pair $(\C_\ell, \C_r)$ of non-unital subalgebras of $\A$ such that
\[\ep(\B \otimes 1_{\B}) \subset \C_\ell \subset \A_\ell \qand \ep(1_{\B} \otimes \B^{\mathrm{op}}) \subset \C_r \subset \A_r.\]

\item A family $\{(\A_{k, \ell}, \A_{k, r})\}_{k \in K}$ of non-unital pairs of $\B$-algebras in $(\A, \bE, \ep)$ is said to be \textit{bi-Boolean independent with amalgamation over $\B$} if for all $n \geq 1$, $\chi: \{1, \dots, n\} \to \{\ell, r\}$, $\omega: \{1, \dots, n\} \to K$, and $Z_1, \dots, Z_n \in \A$ with $Z_j \in \A_{\omega(j), \chi(j)}$, we have that
\[\bE(Z_1\cdots Z_n) = \E_{\pi_{\omega, \chi}}(Z_1, \dots, Z_n),\]
where $\pi_{\omega, \chi}$ is the partition of $\{1, \dots, n\}$ induced by $\chi$ and $\omega$ as described in Definition \ref{Partition}.
\end{enumerate}
\end{defn}

As with the scalar-valued case, by taking $\chi: \{1, \dots, n\} \to \{\ell, r\}$ to be constant and $\omega: \{1, \dots, n\} \to K$ such that $\omega(1) \neq \cdots \neq \omega(n)$, the families $\{\A_{k, \ell}\}_{k \in K}$ and $\{\A_{k, r}\}_{k \in K}$ are both Boolean independent over $\B$ if $\{(\A_{k, \ell}, \A_{k, r})\}_{k \in K}$ is bi-Boolean independent over $\B$.

\subsection{Operator-valued bi-Boolean cumulants}

It is now straightforward to define the operator-valued bi-Boolean cumulant function via convolution over the lattice of bi-interval partitions and check that it has the vanishing property.

\begin{defn}
Let $(\A, \bE, \ep)$ be a $\B$-$\B$-non-commutative probability space and let $\E$ be the operator-value bi-free moment function on $\A$. The \textit{operator-valued bi-Boolean cumulant function} on $\A$ is the function
\[B: \bigcup_{n \geq 1}\bigcup_{\chi: \{1, \dots, n\} \to \{\ell, r\}}\B\I(\chi) \times \A_{\chi(1)} \times \cdots \times \A_{\chi(n)} \to \B\]
defined by
\[B_\pi(Z_1, \dots, Z_n) = \sum_{\substack{\sigma \in \B\I(\chi)\\\sigma \leq \pi}}\E_\sigma(Z_1, \dots, Z_n)\mu_{\B\I}(\sigma, \pi)\]
for all $n \geq 1$, $\chi: \{1, \dots, n\} \to \{\ell, r\}$, $\pi \in \B\I(\chi)$, and $Z_j \in \A_{\chi(j)}$.
\end{defn}

Using the (bi-interval) M\"{o}bius inversion, an equivalent formulation of the above formula is
\[\E_\pi(Z_1, \dots, Z_n) = \sum_{\substack{\sigma \in \B\I(\chi)\\\sigma \leq \pi}}B_\sigma(Z_1, \dots, Z_n)\]
for all $n \geq 1$, $\chi: \{1, \dots, n\} \to \{\ell, r\}$, $\pi \in \B\I(\chi)$, and $Z_j \in \A_{\chi(j)}$.

Since $\E$ is operator-valued bi-multiplicative and $\mu_{\B\I}$ is multiplicative, a routine verification similar to (and simpler than) the proof of \cite{CNS2015-2}*{Theorem 6.2.1} shows that $B$ is operator-valued b-bi-multiplicative. Moreover, it is also easy to check that the operator-valued version of Theorem \ref{VanishingEquiv} holds.

\begin{thm}\label{OpVVanishingEquiv}
A family $\{(\A_{k, \ell}, \A_{k, r})\}_{k \in K}$ of non-unital pairs of $\B$-algebras in a $\B$-$\B$-non-commutative probability space $(\A, \bE, \ep)$ is bi-Boolean independent over $\B$ if and only if for all $n \geq 2$, $\chi: \{1, \dots, n\} \to \{\ell, r\}$, $\omega: \{1, \dots, n\} \to K$, and $Z_j \in \A_{\omega(j), \chi(j)}$, we have that
\[B_{1_\chi}(Z_1, \dots, Z_n) = 0\]
whenever $\omega$ is not constant.
\end{thm}

\begin{proof}
The proof is completely identical to the second proof of Theorem \ref{VanishingEquiv} for the scalar-valued case where the fact that vanishing of mixed operator-valued bi-Boolean cumulants implies bi-Boolean independence over $\B$ follows from a calculation with M\"{o}bius inversion and the converse follows from an induction argument.
\end{proof}

Moreover, using operator-valued b-bi-multiplicativity, Proposition \ref{ScalarEntry} can be generalized as follows.

\begin{prop}
Let $(\A, \bE, \varepsilon)$ be a $\B$-$\B$-non-commutative probability space, $\chi: \{1, \dots, n\} \to \{\ell, r\}$ with $n \geq 2$, and $Z_j \in \A_{\chi(j)}$. If there exist $q \in \{1, \dots, n\}$ and $b \in \B$ such that $Z_q = L_b$ if $\chi(q) = \ell$ or $Z_q = R_b$ if $\chi(q) = r$, then $B_{1_\chi}(Z_1, \dots, Z_n) = 0$ if $q \in \{\min_{\prec_\chi}(\{1, \dots, n\}), \max_{\prec_\chi}(\{1, \dots, n\})\}$ and otherwise
\begin{align*}
B_{1_\chi}(Z_1, \dots, Z_n) =
B_{1_{\chi|_{\setminus q}}}(Z_1, \dots, Z_{p - 1}, Z_pZ_q, Z_{p + 1}, \dots, Z_{q - 1}, Z_{q + 1}, \dots, Z_n)
\end{align*}
where $p = \max\{j \in \{1, \dots, n\}\,|\,\chi(j) = \chi(q), j < q\}$.
\end{prop}

\subsection{The operator-valued bi-Boolean partial $\eta$-transform}

In this subsection, we develop an operator-valued bi-Boolean partial $\eta$-transform and prove the operator-valued version of Theorem \ref{PartialEta}. The proof is combinatorial in nature following the techniques developed in \cite{S2016-2}*{Section 7}. The same techniques were also used in \cite{S2016}*{Section 5} and \cite{GS2016-2}*{Section 8} in deriving a functional equation for the operator-valued bi-free/c-bi-free partial $\R$-transform. However, these transforms are functions of three variables instead of two variables due to the fact that when one sums over bi-non-crossing partitions with the same block that contains both left and right indices, a $\B$-operator is created corresponding to the `bottom part' of the diagram below the common block and operator-valued bi-multiplicativity does not allow the $\B$-operator to escape. For bi-interval partitions, such `bottom part' does not exist due to the property that if $V$ is a block of a bi-interval partition that contains both a left index $s$ and a right index $t$, then $V$ contains all indices $j$ such that $s \prec_\chi j \prec_\chi t$. Consequently, the operator-valued bi-Boolean partial $\eta$-transform is a function of two-variables: one for a left $\B$-operator and one for a right $\B$-operator.

In that which follows, we require the additional assumption that $\B$ is a Banach algebra and all operators are elements of a \textit{Banach $\B$-$\B$-non-commutative probability space}, which is a $\B$-$\B$-non-commutative probability space $(\A, \bE, \varepsilon)$ such that $\A$ and $\B$ are Banach algebras, and $\bE$, $\varepsilon|_{\B \otimes 1_\B}$, and $\varepsilon|_{1_\B \otimes \B^{\mathrm{op}}}$ are bounded. To begin, let $Z_\ell \in \A_\ell$, $Z_r \in \A_r$, $b, d \in \B$, and consider the following series:
\begin{align*}
M_{Z_\ell}^\ell(b) &= 1 + \sum_{m \geq 1}\bE((L_bZ_\ell)^m)& & & M_{Z_r}^r(d) &= 1 + \sum_{n \geq 1}\bE((R_dZ_r)^n)\\
\eta_{Z_\ell}^\ell(b) &= 1 + \sum_{m \geq 1}B_{1_{\chi_{m, 0}}}(\underbrace{L_bZ_\ell, \dots, L_bZ_\ell}_{m\,\mathrm{times}}) & & & \eta_{Z_r}^r(d) &= 1 + \sum_{n \geq 1}B_{1_{\chi_{0, n}}}(\underbrace{R_dZ_r, \dots, R_dZ_r}_{n\,\mathrm{times}}).
\end{align*}

By similar arguments as in \cite{S2016}*{Remark 5.2}, all of the series above converge absolutely for $b, d$ near $0$. It is known (see, e.g., \cite{P2009}*{Section 4}) that $M_{Z_\ell}^\ell(b)$ and $M_{Z_r}^r(d)$ are invertible and that
\[\eta_{Z_\ell}^\ell(b) = 1 - M_{Z_\ell}^\ell(b)^{-1} \qand \eta_{Z_r}^r(d) = 1 - M_{Z_r}^r(d)^{-1}\]
for $b, d \in \B$ sufficiently small. In addition, consider the following two-variable series:
\begin{align*}
M_{(Z_\ell, Z_r)}(b, d) &= 1 + \sum_{\substack{m, n \geq 0\\m + n \geq 1}}\bE((L_bZ_\ell)^m(R_dZ_r)^n),\\
\eta_{(Z_\ell, Z_r)}(b, d) &= \sum_{\substack{m, n \geq 0\\m + n \geq 1}}B_{1_{\chi_{m, n}}}(\underbrace{L_bZ_\ell, \dots, L_bZ_\ell}_{m\,\mathrm{times}}, \underbrace{R_dZ_r, \dots, R_dZ_r}_{n\,\mathrm{times}}),
\end{align*}
which converge absolutely for $b, d$ near $0$ by similar arguments. Note that Theorem \ref{OpVVanishingEquiv} implies $\eta_{(Z_\ell, Z_r)}$ is the operator-valued analogue of $\eta_{(a, b)}$ for $(a, b)$ a (scalar-valued) pair in a non-commutative probability space. The goal of this subsection is to find a formula relating $\eta_{(Z_\ell, Z_r)}$ to $M_{(Z_\ell, Z_r)}$.

\begin{thm}\label{OpVPartialEta}
Let $(Z_\ell, Z_r)$ be an operator-valued two-faced pair in a Banach $\B$-$\B$-non-commutative probability space $(\A, \bE, \ep)$. The operator-valued bi-Boolean partial $\eta$-transform $\eta_{(Z_\ell, Z_r)}$ of $(Z_\ell, Z_r)$ is given by
\[\eta_{(Z_\ell, Z_r)}(b, d) = \eta_{Z_\ell}^\ell(b) + \eta_{Z_r}^r(d) + M_{Z_\ell}^\ell(b)^{-1}M_{(Z_\ell, Z_r)}(b, d)M_{Z_r}^r(d)^{-1} - 1\]
for $b, d \in \B$ sufficiently small.
\end{thm}

\begin{proof}
For $\chi: \{1, \dots, n\} \to \{\ell, r\}$, let $\B\I_{\mathrm{vs}}(\chi)$ denote the set of partitions in $\B\I(\chi)$ such that no block contains both left and right indices. For $m, n \geq 1$, using operator-valued b-bi-multiplicative properties, we have that
\begin{align*}
\bE((L_bZ_\ell)^m(R_dZ_r)^n) &= \sum_{\substack{\pi \in \B\I(\chi_{m, n})\\\pi \in \B\I_{\mathrm{vs}}(\chi_{m, n})}}B_\pi(\underbrace{L_bZ_\ell, \dots, L_bZ_\ell}_{m\,\mathrm{times}}, \underbrace{R_dZ_r, \dots, R_dZ_r}_{n\,\mathrm{times}})\\
&\quad + \sum_{\substack{\pi \in \B\I(\chi_{m, n})\\\pi \notin \B\I_{\mathrm{vs}}(\chi_{m, n})}}B_\pi(\underbrace{L_bZ_\ell, \dots, L_bZ_\ell}_{m\,\mathrm{times}}, \underbrace{R_dZ_r, \dots, R_dZ_r}_{n\,\mathrm{times}})\\
&= \bE((L_bZ_\ell)^m)\bE((R_dZ_r)^n) + \Theta_{m, n}(b, d),
\end{align*}
where $\Theta_{m, n}(b, d)$ denotes the sum
\[\sum_{\substack{\pi \in \B\I(\chi_{m, n})\\\pi \notin \B\I_{\mathrm{vs}}(\chi_{m, n})}}B_\pi(\underbrace{L_bZ_\ell, \dots, L_bZ_\ell}_{m\,\mathrm{times}}, \underbrace{R_dZ_r, \dots, R_dZ_r}_{n\,\mathrm{times}}).\]

For every partition $\pi \in \B\I(\chi_{m, n}) \setminus \B\I_{\mathrm{vs}}(\chi_{m, n})$, there is exactly one block of $\pi$ with both left and right indices, let $V_\pi$ denote this block. Moreover, since $\pi$ is a bi-interval partition, we have $j \in V_\pi$ for all $\min_{\prec_\chi}(V_\pi) \prec_\chi j \prec_\chi \max_{\prec_\chi}(V_\pi)$. Rearrange the sum in $\Theta_{m, n}(b, d)$ (which may be done as it converges absolutely) by first choosing $s \in \{1, \dots, m\}$, $t \in \{1, \dots, n\}$, and then summing over all $\pi \in \B\I(\chi_{m, n}) \setminus \B\I_{\mathrm{vs}}(\chi_{m, n})$ such that $V_\pi = \{s, \dots, m, m + t, \dots, m + n\}$, we obtain
\[\Theta_{m, n}(b, d) = \sum_{s = 1}^m\sum_{t = 1}^n\sum_{\substack{\pi \in \B\I(\chi_{m, n})\\\pi \notin \B\I_{\mathrm{vs}}(\chi_{m, n})\\V_\pi = \{s, \dots, m, m + t, \dots, m + n\}}}B_\pi(\underbrace{L_bZ_\ell, \dots, L_bZ_\ell}_{m\,\mathrm{times}}, \underbrace{R_dZ_r, \dots, R_dZ_r}_{n\,\mathrm{times}}).\]
Furthermore, using operator-valued b-bi-multiplicative properties, the right-most sum is
\[\bE((L_bZ_\ell)^{s - 1})B_{1_{\chi_{m - (s - 1), n - (t - 1)}}}(\underbrace{L_bZ_\ell, \dots, L_bZ_\ell}_{m - (s - 1)\,\mathrm{times}}, \underbrace{R_dZ_r, \dots, R_dZ_r}_{n - (t - 1)\,\mathrm{times}})\bE((R_dZ_r)^{t - 1}).\]
Consequently, we obtain $\Theta_{m, n}(b, d)$ equals
\[\sum_{s = 1}^m\sum_{t = 1}^n\bE((L_bZ_\ell)^{m - s})B_{1_{\chi_{s, t}}}(\underbrace{L_bZ_\ell, \dots, L_bZ_\ell}_{s\,\mathrm{times}}, \underbrace{R_dZ_r, \dots, R_dZ_r}_{t\,\mathrm{times}})\bE((R_dZ_r)^{n - t}).\]
By choosing $b, d \in \B$ sufficiently small so that $M_{Z_\ell}^{\ell}(b)$ and $M_{Z_r}^r(d)$ are invertible, and by summing over the above expression over all $m, n \geq 1$, we have that
\begin{align*}
\sum_{m, n \geq 1}\Theta_{m, n}(b, d) &= M_{Z_\ell}^{\ell}(b)\sum_{s, t \geq 1}B_{1_{\chi_{s, t}}}(\underbrace{L_bZ_\ell, \dots, L_bZ_\ell}_{s\,\mathrm{times}}, \underbrace{R_dZ_r, \dots, R_dZ_r}_{t\,\mathrm{times}})M_{Z_r}^r(d)\\
&= M_{Z_\ell}^{\ell}(b)(\eta_{(Z_\ell, Z_r)}(b, d) - \eta_{Z_\ell}^\ell(b) - \eta_{Z_r}^r(d))M_{Z_r}^r(d).
\end{align*}

On the other hand, expanding $M_{(Z_\ell, Z_r)}(b, d)$ using the fact it converges absolutely produces
\begin{align*}
M_{(Z_\ell, Z_r)}(b, d) &= 1 + \sum_{m \geq 1}\bE((L_bZ_\ell)^m) + \sum_{n \geq 1}\bE((R_dZ_r)^n) + \sum_{m, n \geq 1}\bE((L_bZ_\ell)^m(R_dZ_r)^n)\\
&= 1 + \sum_{m \geq 1}\bE((L_bZ_\ell)^m) + \sum_{n \geq 1}\bE((R_dZ_r)^n) + \sum_{m, n \geq 1}\bE((L_bZ_\ell)^m)\bE((R_dZ_r)^n) + \sum_{m, n \geq 1}\Theta_{m, n}(b, d)\\
&= M_{Z_\ell}^\ell(b)M_{Z_r}^r(d) + \sum_{m, n \geq 1}\Theta_{m, n}(b, d)\\
&= M_{Z_\ell}^{\ell}(b)(1 + \eta_{(Z_\ell, Z_r)}(b, d) - \eta_{Z_\ell}^\ell(b) - \eta_{Z_r}^r(d))M_{Z_r}^r(d),
\end{align*}
and the result follows.
\end{proof}

\subsection{Some final remarks}

In this final subsection, we briefly discuss how some of the known results also hold in the operator-valued bi-Boolean setting. As with Section \ref{sec:limitthms}, most of the details are omitted as the statements and proofs are essentially the same.

In \cite{CNS2015-2}*{Theorem 10.2.1}, it was demonstrated that if all left $\B$-algebras commute with all right $\B$-algebras, then under certain circumstances bi-free independence over $\B$ can be deduced from free independence over $\B$ of either the left $\B$-algebras or the right $\B$-algebras. The quantitative realizations of the arguments were proved in \cite{S2016}*{Lemmata 2.16 and 2.17} and shown to be useful in proving \cite{S2016}*{Theorem 3.2}. As the c-bi-free, operator-valued c-bi-free, and bi-Boolean analogues of these results are available in \cite{GS2016}*{Section 4}, \cite{GS2016-2}*{Section 7}, and Section \ref{sec:transforms} above respectively, it can be shown without any difficulty that similar results also hold in the framework of the current section.

Furthermore, like any other non-commutative probability theory, whenever a notion of independence (scalar-valued or operator-valued) is defined, a number of limit theorems can be obtained for various purposes. Since we do not have any immediate applications of the operator-valued bi-Boolean limit theorems available at the moment, we will simply mention that the bi-Boolean analogues of the operator-valued bi-free/c-bi-free limit theorems in \cite{GS2016-2}*{Section 9} can be easily proved by the same techniques.

\section*{Errata to ``Bi-Boolean Independence for Pairs of Algebras" - Joint with Takahiro Hasebe}

As it was realized during the study of \cite{GHS2017}, an essential assumption was omitted in many results of \cite{GS2016}: namely, the results hold provided the conditional bi-free product of states (unital positive linear functionals) is again a state.  When studying bi-Boolean independence above (a specific instance of conditional bi-free independence), this assumption was not verified.  In fact, this assumption cannot be verified as it is false.

To see this, suppose $(a_1, b_1)$ and $(a_2, b_2)$ are two pairs of commuting elements that are bi-Boolean independent with respect to $\varphi$ and whose distributions with respect to $\varphi$ are probability measures. 
For the distribution of $(a_1+a_2,b_1+b_2)$ to be a probability measure on $\mathbb R^2$, it is necessary that  for every $n \in \bN$ and every polynomial $p(x, y) = \sum^n_{k,m = 0} c_{k,m} x^k y^m$ where $c_{k,m} \in \bC$ we have 
\[
0 \leq \varphi(p(a_1+a_2, b_1+b_2)^*p(a_1+a_2, b_1+b_2)) = \sum^n_{i_1, i_2, j_1, j_2 = 0} c_{i_1, i_2} \overline{c_{j_1, j_2}} \varphi((a_1+a_2)^{i_1+j_1} (b_1+b_2)^{i_2+j_2}).
\]
Therefore, for a fixed $n$ we consider the $(n+1)^2 \times (n+1)^2$ matrix 
\[
X_n = [\varphi((a_1+a_2)^{i_1+j_1} (b_1+b_2)^{i_2+j_2})]
\]
where the rows are indexed (starting at 0 up to $n$) by the pairs $(i_1, i_2)$ and the columns are indexed by the pairs $(j_1, j_2)$.  Then
\[
\varphi(p(a_1+a_2, b_1+b_2)^*p(a_1+a_2, b_1+b_2)) = \langle X_n\vec{v}, \vec{v}\rangle, 
\]
where $\vec{v}$ is a vector of the $c_{k,m}$'s. For example the matrix $X_1$ is given by 
\[
X_1 = \begin{bmatrix}
1 & \varphi(a_1+a_2) & \varphi(b_1 + b_2) & \varphi((a_1+a_2)(b_1+b_2)) \\
\varphi(a_1 + a_2) & \varphi((a_1+a_2)^2) & \varphi((a_1+a_2)(b_1+b_2)) & \varphi((a_1+a_2)^2(b_1+b_2)) \\
\varphi(b_1 + b_2) & \varphi((a_1+a_2)(b_1+b_2)) & \varphi((b_1 + b_2)^2)  & \varphi((a_1+a_2)(b_1+b_2)^2) \\
 \varphi((a_1+a_2)(b_1+b_2)) &  \varphi((a_1+a_2)^2(b_1+b_2)) &  \varphi((a_1+a_2)(b_1+b_2)^2) &  \varphi((a_1+a_2)^2(b_1+b_2)^2)
\end{bmatrix}. 
\]

Assume that $(a_1,b_1)$ and $(a_2,b_2)$ both have the probability distribution $\mu = \frac{1}{2}(\delta_{(0,1)}+\delta_{(1,0)})$. Notice that 
\[
\int_{\bR^2} x^m y^n \, d\mu(x,y) = \begin{cases}
1 & \text{if }m= n=0, \\
\frac{1}{2} & \text{if }(m,n) \in \{(k, 0), (0, k) \, \mid \, k \in \bN\}, \\
0 & \text{otherwise}. 
\end{cases}
\]
Then one can verify using the definition of bi-Boolean independence in \cite{GS2017} to obtain 
\[
X_1 = \begin{bmatrix}
1 & 1 & 1 & \frac{1}{2} \\
1 & \frac{3}{2} & \frac{1}{2} & \frac{3}{4} \\
1 & \frac{1}{2} & \frac{3}{2} & \frac{3}{4} \\
\frac{1}{2} & \frac{3}{4} & \frac{3}{4} & \frac{9}{8}
\end{bmatrix}.
\]
One can check that $\det(X_1) = - \frac{1}{8}$.  This implies that $X_1$ is not positive semidefinite and thus the distribution of $(a_1+a_2, b_1+ b_2)$ is not a probability measure. This also shows that bi-Boolean product of states is not a state in general and the general non-existence of conditionally bi-free convolution of probability measures.

Furthermore, the compound bi-Boolean Poisson distribution defined in \cite{GS2017} is not a probability measure in general.  Suppose that $(a,b)$ is a pair of elements in a non-commutative space such that 
its bi-Boolean cumulants are given by 
$$
B_{m,n}(a,b) := \int_{\mathbb R^2}s^m t^n\, d\tau(s,t),
$$ 
where $\tau = 3\delta_{(1,1)}+3 \delta_{(-1,1)}+3\delta_{(1,-1)}$. Then $B_{m,n}(a,b)=3[(-1)^m + (-1)^n +1]$ for $m,n \geq0, (m,n) \neq (0,0)$. The two-variable $E$-transform is related to $B_{m,n}(a,b)$ via 
$$
E_{(a,b)}(z,w) = \sum_{\substack{ m,n=0 \\ (m,n)\neq(0,0)}}^\infty \frac{B_{m,n}(a,b)}{z^m w^n}. 
$$
The Boolean cumulant generating functions of marginals are 
\begin{align*}
E_a(1/z) &= \sum_{m=1}^\infty B_{m,0}(a,b) z^m = \sum_{n=1}^\infty B_{0,n}(a,b) z^n = E_b(1/z) 
\end{align*}
and so the Cauchy transforms of marginals are 
\begin{align*}
G_a(1/z) = z/(1-E_a(1/z)) = z + 3 z^2 + 18z^3 + \cdots=G_b(1/z). 
\end{align*}
Then
\begin{align*}
G_{(a,b)}(1/z,1/w) &= G_a(1/z) G_b(1/w)\left(1+ \sum_{m,n\geq1} B_{m,n}(a,b) z^m w^n\right) \\
&= z w+ 3z^2 w + 3 z w^2 + 18 z^3 w + 6 z^2 w^2 + 18 z w^3 + 48z^3w^2 + 48 z^2 w^3 + 324z^3w^3+ \cdots.   
\end{align*}
Denote by $M_{m,n}$ the coefficient of $z^{m+1} w^{n+1}$ in the above expression of $G_{(a,b)}(1/z,1/w)$. The determinant of $4 \times 4$ matrix $(M_{m_1+ m_2,n_1+n_2})$, with the 4 pairs $\{(m_1, m_2): m_1, m_2 =0,1\}$ indexing the column and 4 pairs $\{(n_1,n_2):n_1,n_2=0,1\}$ indexing the row, is 
$$
\left|\begin{matrix} 
1 & 3 & 3 & 6 \\
3 & 18 & 6 & 48 \\ 
3 & 6 & 18 & 48 \\
6 & 18 & 18 & 324
\end{matrix}\right|
= -864 <0. 
$$
Hence $M_{m,n}$ are not moments of a probability measure on $\mathbb R^2$.

Due to the bi-Boolean examples given above, the conditional bi-free product of states need not be a state and the conditional bi-free convolution of two positive probability measures need not be a probability measure.  Hence the results of Sections 6.2-6.6 of \cite{GS2016} are false as stated.  The results in these later sections should be stated only for conditional bi-free products where conditional bi-free product of states is a state.  Currently, the only non-trivial example known of when this is true is when the conditional bi-free product state is the bi-free product state.  Thus a natural question would be to determine all conditional bi-free products of states that produce states.

Due to the bi-Boolean examples given above, the bi-Boolean product of states need not be a state and the bi-Boolean convolution of two positive probability measures need not be a probability measure.  In particular, the results of Sections 5 and 6 of this paper for probability measures are false as stated, modulo those in Section 5.1 where the notion of probability measure is replaced with an arbitrary distribution.  As was demonstrated via Example 5.13, there existed a probability measure that was not bi-Boolean infinitely divisible inside the collection of probability measures as one of the asymptotic moments involved would have to be negative.  Due to the above, it is natural to ask whether the bi-Boolean (and, more generally, conditional bi-free) convolution of finite signed measures is a finite signed measure.

\end{document}